\documentclass[12pt,a4paper]{amsart}
\usepackage[T1]{fontenc}
\usepackage[a4paper,left=2.25cm,right=2.25cm,top=3cm,bottom=3cm,headsep=1cm]{geometry}
\usepackage{hyperref}
\usepackage{cleveref}
\usepackage{mathtools} 
\usepackage{tikz}\usetikzlibrary{graphs,quotes,fit,positioning,matrix,calc,decorations.markings,angles,decorations.pathmorphing,decorations.pathreplacing}
\tikzset{math mode/.style = {execute at begin node=$, execute at end node=$}}
\tikzset{arrow/.style={postaction={decorate,thick,decoration={markings,mark = at position #1 with {\arrow{>}}}}},arrow/.default=0.5}
\tikzset{invarrow/.style={postaction={decorate,thick,decoration={markings,mark = at position #1 with {\arrow{<}}}}},invarrow/.default=0.5}
\newcommand\bbull[3]{\filldraw[fill=black!35!blue, draw=black] (#1,#2) circle (#3cm);}
\newcommand\ebull[3]{\filldraw[fill=white,draw=black] (#1,#2) circle (#3cm);}
\tikzset{fermionic/.style={thin}}
\tikzset{bosonic/.style={ultra thick}}
\tikzset{gcol/.style={Bo}}
\tikzset{gccol/.style={red}}
\tikzset{Gcol/.style={blue}}
\DeclarePairedDelimiter\autobracket{(}{)}
\newcommand{\br}[1]{\autobracket*{#1}}
\definecolor{Bo}{rgb}{1.0, 0.49, 0.0}
\definecolor{Bs}{rgb}{1.0, 0.44, 0.37}
\definecolor{aj}{rgb}{0.8, 0.8, 1.0}

\newcommand{\sslash}{\mathbin{/\mkern-6mu/}}
\newcommand{\bra}[1]{\left\langle #1\right|}
\newcommand{\ket}[1]{\left|#1\right\rangle}
\usepackage{hyperref}
\allowdisplaybreaks[1]
\renewcommand\ss{\scriptstyle}
\newcommand\sss{\scriptscriptstyle}

\def\phid{\phi^\dagger}

\newcommand\g{{g^{(\alpha,\beta)}_{\lambda}}}
\newcommand\G{{G^{(\alpha,\beta)}_{\lambda}}}

\newcommand\ZZ{{\mathbb Z}}

\newcommand\CC{{\mathbb C}}

\newcommand\jl{{\mathfrak l}}
\newcommand\jt{{\mathfrak t}}
\theoremstyle{plain}
\newtheorem{thm}{Theorem}

\newtheorem{prop}{Proposition}
\newtheorem{cor}{Corollary}

\theoremstyle{remark}
\newtheorem*{rmk*}{Remark}
\newtheorem*{ex*}{Example}
\newdimen{\cellsize}
\newcommand\medboxes{\setlength{\cellsize}{14.22pt}\def\boxformat{}}

\medboxes
\tikzset{tableaubox/.style={draw=black,thin,sharp corners,solid,minimum size=\cellsize,inner sep=0pt}}
\tikzset{tableau/.style={matrix,name=tab,matrix anchor=tab-1-1.south west,inner sep=1pt,matrix of math nodes,cells={anchor=center,draw=black,thin,solid,arrows=-},nodes={tableaubox,execute at begin node=\boxformat},nodes in empty cells,row sep={\cellsize,between origins},column sep={\cellsize,between origins}}}
\newcommand\graycell{|[fill=gray]|}
\makeatletter
\newcommand\cellextra[1]{#1\expandafter\tikz@lib@matrix@start@cell}
\makeatother
\def\activate#1{\begingroup
  \lccode`\~=`#1%
  \lowercase{\endgroup \let~#1}%
  \catcode`#1=13\relax}
\activate &
\newcommand\tableau[1]{\tikz[baseline=0]
\node[tableau]{#1};}

\long\def\junk#1{}
\title{Vertex models for Canonical Grothendieck polynomials and their duals}
\author{Ajeeth Gunna}
\author{Paul Zinn-Justin}
\address{Ajeeth Gunna, Paul Zinn-Justin, School of Mathematics and Statistics, University of Melbourne, Parkville, Victoria 3010, Australia.}
\email{agunna@student.unimelb.edu.au}
\email{pzinn@unimelb.edu.au}
\thanks{}
\date{\today}
\begin{document}
\begin{abstract}
We study solvable lattice models associated to canonical Grothendieck polynomials and their duals. We derive inversion relations and Cauchy identities.
\end{abstract}

\maketitle

\tableofcontents

\section{Introduction}
Grothendieck polynomials were introduced by Lascoux and Schutzenberger in \cite{LS-groth} as representatives of $K$-theoretic Schubert classes in flag varieties. Their connection to quantum integrability was noticed as early as \cite{FK-Groth}, though it took some time to reformulate Grothendieck polynomials in the context of exactly solvable lattice models \cite{hdr}, where quantum integrability is most explicit. Recently, a large literature has developed around these ideas \cite{Ms-bosfer-ktheory,Ms-grothmodel,artic71,artic68,Bs-coloreddoublegroth}. In this work we focus on {\em symmetric}\/ Grothendieck polynomials, i.e., the ones that are related to the $K$-theory of Grassmannians, though we expect many of our ideas to be applicable to more general (partial) flag varieties. We also consider their duals, in the sense of product/coproduct duality\footnote{These should not be confused with the dual Grothendieck polynomials that are e.g.\ considered in \cite{artic68}. The latter are dual w.r.t.\ the natural scalar product of $K$-theory. In contrast, ours are dual w.r.t.\ the Hall inner product.}. We propose some new formulations of both Grothendieck and dual Grothendieck in terms of certain ``bosonic'' exactly solvable lattice models.

Let $\Lambda$ be the ring of symmetric functions. Even though the elements of ${\Lambda}$ are not polynomials, by abuse of language we shall refer to them as polynomials, identifying a symmetric function $F$ with the corresponding symmetric polynomials $F(x_1,\ldots,x_n)$. Schur polynomials $s_{\lambda}$ (where $\lambda$ runs over all partitions) form an orthogonal basis of $\Lambda$ under the Hall inner product. The involution map $\omega$, which sends $e_{k}$ (\emph{elementary symmetric polynomials}) to $h_{k}$ (\emph{complete homogeneous symmetric polynomials}), maps $s_{\lambda}$ to $s_{\lambda'}$. Let $\tilde{\Lambda}$ be the completion of $\Lambda$, which is obtained by allowing infinite linear combinations of $s_{\lambda}$.

Grothendieck polynomials $G_\lambda$ are inhomogeneous symmetric polynomials; with the
appropriate choice of variables, $G_\lambda=s_\lambda+\text{higher order terms}$. When the number of variables grows, their degree grows, so they must be considered as elements of $\tilde \Lambda$.
The structure constants $c^{\nu}_{\lambda,\mu}$ defined by
\[
G_{\lambda}G_{\mu}=\sum_{\nu}c^{\nu}_{\lambda,\mu}G_{\nu},
\]
 satisfy $c^{\nu'}_{\lambda',\mu'}=c^\nu_{\lambda,\mu}$ where $\lambda'$ is the transpose of $\lambda$. However the image of $G_{\lambda}$ under $\omega$ is {\em not} $G_{\lambda'}$. This implies that the family of polynomials $\omega(G_{\lambda'})$ has the same structure constants as Grothendieck polynomials. We shall not be dealing with structure constants in this paper, reserving them for subsequent work \cite{Vortex-honeycombs} and only mention them as motivation for what follows.

In  \cite{Lp-dualgroth}, Lam and Pylyavskyy defined \emph{dual Grothendieck polynomials} $(g_{\lambda})$ as certain generating functions of \emph{reverse plane partitions}. These polynomials are dual to $G_{\lambda}$ under the Hall inner product, and are of the form $g_\lambda=s_\lambda+\text{lower order terms}$. Similarly to Grothendieck polynomials, the image of $g_{\lambda}$ under the involution map is not $g_{\lambda'}$.

In \cite{canon-groth}, Yeliussizov introduced a two parameter version of Grothendieck polynomials and their dual, which he called \emph{canonical Grothendieck polynomials} and \emph{dual canonical  Grothendieck polynomials}. For more detailed combinatorial properties and definitions we refer the reader to \cite{canon-groth}. Canonical Grothendieck polynomials and their dual satisfy the following relation,
\[
\omega(\G)=G^{(\beta,\alpha)}_{\lambda'}\qquad  \omega(\g)=g^{(\beta,\alpha)}_{\lambda'}
\]

In this paper, we shall study two types of vertex models, based on the way partitions are encoded, for each $\G$ and $\g$. We call a vertex model \emph{row model} (resp.\ \emph{column model}) when the partitions are encoded by row (resp.\ column) multiplicities. Section 2 is devoted to the former, section 3 to the latter.

We then introduce (section 4) \emph{generalised Grothendieck  polynomials} 
which are obtained by attaching additional variables to the vertical lines of the underlying lattice model. Along the process, we recover the \emph{generalised dual Grothendieck polynomials} defined by Yeliussizov \cite{DY-dualcauchy}.

In section 5, we show that the transfer matrices of these lattice models satisfy remarkable {\em inversion relations}. These show a deep connection between row and column lattice models, thus embodying the involution $\omega$ at the level of transfer matrices. This should be reminiscent of similar relations satisfied by the usual free fermionic {\em vertex operators} related to Schur functions (see e.g.~\cite{cslectures}, or \cite{hdr} and references therein); indeed, our transfer matrices can be thought of as deformations of these vertex operators.

Finally, in section 6, we show how ``quantum integrability'' under the form of RLL relations immediately implies the \emph{Cauchy} identities
\begin{align}
\label{cauchyidentity1}
\sum_{\lambda}G^{(-\alpha,-\beta)}_{\lambda}(x_1,x_2,\dots,x_m) g^{(\alpha,\beta)}_{\lambda}(y_1,y_2,\dots,y_n)&=\prod_{1\le i \le m,1 \le j \le n} \dfrac{1}{1-x_i y_j}
\\
\label{cauchyidentity2}
\sum_{\lambda}G^{(-\beta,-\alpha)}_{\lambda'}(x_1,x_2,\dots,x_m) g^{(\alpha,\beta)}_{\lambda}(y_1,y_2,\dots,y_n)&=\prod_{1\le i \le m,1 \le j \le n} {(1+x_i y_j)}
\end{align}
for (generalised) Grothendieck polynomials and their duals.
By specializing $\alpha=0$ and $\beta=1$, we recover the \emph{Cauchy identity} of  Grothendieck polynomials and its dual \cite{las2014,cauchyid-groth}.

The appendix contains proofs of the RLL relations.

\junk{
This paper is organized as follows. In \cref{rowmodels}, we study vertex models for $\G$ and $\g$ where the partitions are encoded by the multiplicity of the rows. In \cref{columnmodels}, we do the same but with partitions recorded by their column multiplicities. Then in \cref{generalisedpolynomials}, we generalise the polynomials by introducing another set of variables, which are attached to the vertical lines in the lattice models. In \cref{dualitybetweenmodels}, we study the inversion relation between the transfer matrices from row and column models. Finally, in \cref{cauchyidentities} we prove the \emph{Cauchy } identities between $\G$ and $\g$.}

\section{Row Vertex Models}
\label{rowmodels}
\subsection{Definition of Physical space.}

Let $V^r$ be an infinite dimensional vector space with basis
indexed by collections of nonnegative integers $(m_i)_{i\in \ZZ_{>0}}$ such that only a finite number of $m_i$s are nonzero; we view it as a subspace of $\bigotimes_{i=1}^\infty V_i$ where each $V_i=\text{Span}(\ket{0},\ket{1},\ldots)$ has a basis indexed by a single nonnegative integer:
\begin{align}
    \label{rowPhysicalspacedefn}
    V^{r}=\text{Span}
    \left\{
    \ket{m_1} \otimes \ket{m_2} \otimes \ket{m_3} \cdots
    \right\}
    \qquad
    m_i \geq0,\ i\geq 1. 
    \end{align}
    
We shall identify partitions with basis elements of $V^{r}$. Given
a partition $\lambda,$ which we view as a Young diagram, let $\ket{\lambda}$ be the basis vector with integers 
\[
m_i(\lambda) = \text{number of rows of size $i$ of $\lambda$}
\]
(hence, the superscript $r$). For example, we identify the partition $\lambda=(4,4,4,3,1)$ with the basis element $\ket{1}\otimes\ket{0}\otimes \ket{1}\otimes\ket{3}\otimes\ket{0}\ldots$ of $V^{r}$:
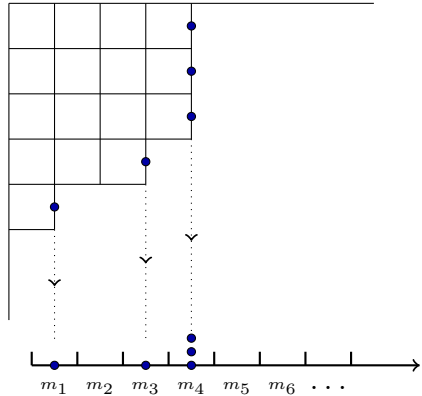
\begin{center}
\label{rowphysicalspacepic}
\begin{tikzpicture}[scale=0.6]
%
\draw (0,4) -- (8,4);
\draw (0,3) -- (4,3);
\draw (0,2) -- (4,2);
\draw (0,1) -- (4,1);
\draw (0,0) -- (3,0);
\draw (0,-1) -- (1,-1);
\draw (0,-3) -- (0,4);
\draw (1,-1) -- (1,4);
\draw (2,0) -- (2,4);
\draw (3,0) -- (3,4);
\draw (4,1) -- (4,4);
\bbull{1}{-0.5}{0.09};
\bbull{3}{0.5}{0.09};
\bbull{4}{1.5}{0.09};
\bbull{4}{2.5}{0.09};
\bbull{4}{3.5}{0.09};
\draw[dotted,arrow=0.5] (1,-1) -- (1,-3.5);
\draw[dotted,arrow=0.5] (3,0) -- (3,-3.5);
\draw[dotted,arrow=0.5] (4,1) -- (4,-3.5);
\draw[thick,arrow=1] (0.5,-4) -- (9,-4);
\foreach\x in {0,1,...,7}{
\draw[thick] (0.5+\x,-4) -- (0.5+\x,-3.7);
}
\bbull{1}{-4}{0.09};
\bbull{3}{-4}{0.09};
\bbull{4}{-3.4}{0.09};\bbull{4}{-3.7}{0.09};\bbull{4}{-4}{0.09};
\node at (1,-4.5) {$\sss m_1$};
\node at (2,-4.5) {$\sss m_2$};
\node at (3,-4.5) {$\sss m_3$};
\node at (4,-4.5) {$\sss m_4$};
\node at (5,-4.5) {$\sss m_5$};
\node at (6,-4.5) {$\sss m_6$};
\node at (7,-4.5) {$\hdots$};
\end{tikzpicture}
\end{center}

All the vertex models studied in this paper follow a general template. In order to not repeat ourselves, we shall study this model in detail and then skip the general arguments in other models.

\subsection{Row vertex model for canonical Grothendieck polynomials.}
\label{rowcGmodel}
\subsubsection{Conventions}
We use the standard diagrammatic formalism to interpret lattice models in terms of linear operators. We briefly review it here, and fix conventions.

All our lattice models are defined on some domain of the plane which consists of edges and vertices of valency 4. Edges traverse vertices to form lines, which are given a certain orientation: in all that follows, the domain is a (rectangular) region of the square lattice, so that lines can be either horizontal (also called ``auxiliary'' lines), in which case they are oriented left to right, or vertical (also called ``physical'' lines), in which case they are oriented bottom to top.

To each line is associated a vector space, and juxtaposition of lines corresponds to tensor product (the order of the factors is the order of the incoming external lines). These vector spaces come equipped with a basis labelled by the various states that edges of the lattice model carry.
In our case, vertical lines are numbered $1,2,\ldots$ from left to right, and vertical edges carry a nonnegative integer, so that to vertical line numbered $i$ we assign the vector space $V_i$ (and collectively they form the ``physical space'' $V^r$). Horizontal edges can carry either labels $0,1$, in which case we call the horizontal line fermionic and assign to it a space $F\cong \CC^2$ (possibly adding a subscript to distinguish the various horizontal lines), or it can carry a nonnegative integer (bosonic line), in which case we call the corresponding vector space $W$.
Graphically, when the auxiliary line is fermionic, we draw thin lines. When they are bosonic, we draw thick lines.

Finally, an important convention is that we transpose all linear operators in order to facilitate reading expressions from left to right; this means that if incoming lines at a vertex form $A\otimes B$ and outgoing lines form $C\otimes D$, then to the vertex is associated a linear operator from $C\otimes D$ to $A\otimes B$. We hope that this does not cause any confusion.

\subsubsection{Definition of the \texorpdfstring{$L$}{L} matrix.}
In this subsection, the auxiliary line is fermionic. 
To every vertex we assign a (Boltzmann) weight that depends on the local configuration (i.e., states of the edges) around it. The weights are given as follows:
\begin{align}
\label{canGboltz}
w^{}_{x}
\left(
\begin{gathered}
\begin{tikzpicture}[scale=0.4,baseline=-2pt]
\draw[fermionic,Gcol,arrow=0.25] (-1,0) node[left,black] {$a$} -- (1,0) node[right,black] {$c$};
\draw[bosonic,arrow=0.25] (0,-1) node[below] {$b$} -- (0,1) node[above] {$d$};
\end{tikzpicture}
\end{gathered}
\right)
\equiv
w^{}_{x}(a,b;c,d)= \delta_{a+b,c+d}
\begin{cases}
\frac{ x}{1- \alpha x}  & \text{when } a=1,\\
\frac{1+\beta x}{1-\alpha x}  & \text{when } a=0,\\
1 & a,b,c,d =0,
\end{cases}
\quad
\end{align}
where $a,c \in \{ 0,1\}, \text{ and } b,d \in \mathbb{Z}_{\geq 0}$.

Let us now represent the vertices graphically with their Boltzmann weights written below them.
\begin{equation}
\label{canGtiles}
\begin{tabular}{ccccc}

\begin{tikzpicture}[scale=0.4,baseline=-2pt]
\draw[fermionic,Gcol,arrow=0.25] (-1,0) node[left,black] {$\ss {0}$} -- (1,0) node[right,black] {$\ss {0}$};
\draw[bosonic,arrow=0.25] (0,-1) node[below] {$\ss {0}$} -- (0,1) node[above] {$\ss {0}$};
\node[text width=1cm] at (1,-3.5){1};
\end{tikzpicture}
\qquad
\begin{tikzpicture}[scale=0.4,baseline=-2pt]
\draw[fermionic,Gcol,arrow=0.25] (-1,0) node[left,black] {$\ss {0}$} -- (1,0) node[right,black] {$\ss {0}$};
\draw[bosonic,arrow=0.25] (0,-1) node[below] {$\ss {m}$} -- (0,1) node[above] {$\ss {m}$};
\node[text width=1cm] at (0.4,-3.5){$ {\frac{1+\beta x}{1-\alpha x}}$};
\end{tikzpicture}
\qquad
\begin{tikzpicture}[scale=0.4,baseline=-2pt]
\draw[fermionic,Gcol,arrow=0.25] (-1,0) node[left,black] {$\ss {0}$} -- (1,0) node[right,black] {$\ss {1}$};
\draw[bosonic,arrow=0.25] (0,-1) node[below] {$\ss {m}$} -- (0,1) node[above] {$\ss{m-1}$};
\node[text width=1cm] at (0.4,-3.5){$ {\frac{1+\beta x}{1-\alpha x}}$};
\end{tikzpicture}
\qquad
\begin{tikzpicture}[scale=0.4,baseline=-2pt]
\draw[fermionic,Gcol,arrow=0.25] (-1,0) node[left,black] {$\ss {1}$} -- (1,0) node[right,black] {$\ss {0}$};
\draw[bosonic,arrow=0.25] (0,-1) node[below] {$\ss {m}$} -- (0,1) node[above] {$\ss {m
+1}$};
\node[text width=1cm] at (0.4,-3.5){$\frac{x}{1-\alpha x}$};
\end{tikzpicture}
\qquad
\begin{tikzpicture}[scale=0.4,baseline=-2pt]
\draw[fermionic,Gcol,arrow=0.25] (-1,0) node[left,black] {$\ss {1}$} -- (1,0) node[right,black] {$\ss {1}$};
\draw[bosonic,arrow=0.25] (0,-1) node[below] {$\ss {m}$} -- (0,1) node[above] {$\ss {m}$};
\node[text width=1cm] at (0.4,-3.5){$\frac{x}{1-\alpha x}$};
\end{tikzpicture}\\
\end{tabular}
\end{equation}

The corresponding linear operator is the so-called $L$ matrix; it acts on $F_i \otimes V_j$, where $F_i= \text{span}\{ \ket{0},\ket{1}\}$. Let us first define annihilation $\phi_j$ and creation $\phid_j$ operators acting on the $j^{th}$ factor $V_j$ of $V^{r}$:
\begin{align*}
\phi_j \ket{m}
&=
\ket{m-1}
\qquad
\phid_j \ket{m}
= \ket{m+1}\\
\phi_j \ket{0}
&=
\ket{0}
\end{align*}
Then
\begin{align}
\label{canLmat}
{{L}}_{i,j}(x)
=
{\frac{1}{1-\alpha x}}
\begin{pmatrix}
\delta_{0,m}(1-\alpha x)+(1-\delta_{0m})(1+\beta x)  & (1+\beta x)\phi_j \\
x \phid_j & x
\end{pmatrix}_{i,j}
\end{align}

We shall now define dual $L$ matrices, $L^{*}$. We obtain $L^{*}$ by flipping the vertices upside down and replacing $0's$ and $1's$.
\begin{align}
\label{dcanGtiles}
\begin{tabular}{ccccc}
\begin{tikzpicture}[scale=0.4,baseline=-2pt]
\draw[fermionic,Gcol,arrow=0.25]  (-1,0) node[left,black] {$\ss {1}$}--(1,0) node[right,black] {$\ss {1}$};
\draw[bosonic,arrow=0.25] (0,-1) node[below] {$\ss {0}$}--(0,1) node[above] {$\ss {0}$};
\node[text width=1cm] at (1,-3.5){1};
\end{tikzpicture}
\qquad
\begin{tikzpicture}[scale=0.4,baseline=-2pt]
\draw[fermionic,Gcol,arrow=0.25] (-1,0) node[left,black] {$\ss {1}$}--(1,0) node[right,black] {$\ss {1}$};
\draw[bosonic,arrow=0.25] (0,-1) node[below] {$\ss {m}$}--(0,1) node[above] {$\ss {m}$};
\node[text width=1cm] at (0.4,-3.5){$ {\frac{1+\beta x}{1-\alpha x}}$};
\end{tikzpicture}
\qquad
\begin{tikzpicture}[scale=0.4,baseline=-2pt]
\draw[fermionic,Gcol,arrow=0.25] (-1,0) node[left,black] {$\ss {1}$} -- (1,0) node[right,black] {$\ss {0}$};
\draw[bosonic,arrow=0.25]  (0,-1) node[below] {$\ss {m-1}$}--(0,1) node[above] {$\ss {m}$};
\node[text width=1cm] at (0.4,-3.5){$ {\frac{1+\beta x}{1-\alpha x}}$};
\end{tikzpicture}
\qquad
\begin{tikzpicture}[scale=0.4,baseline=-2pt]
\draw[fermionic,Gcol,arrow=0.25](-1,0) node[left,black] {$\ss {0}$}-- (1,0) node[right,black] {$\ss {1}$};
\draw[bosonic,arrow=0.25] (0,-1) node[below] {$\ss {m+1}$}--(0,1) node[above] {$\ss {m}$};
\node[text width=1cm] at (0.4,-3.5){$\frac{x}{1-\alpha x}$};
\end{tikzpicture}
\qquad
\begin{tikzpicture}[scale=0.4,baseline=-2pt]
\draw[fermionic,Gcol,arrow=0.25]  (-1,0) node[left,black] {$\ss {0}$}--(1,0) node[right,black] {$\ss {0}$} ;
\draw[bosonic,arrow=0.25]  (0,-1) node[below] {$\ss {m}$}--(0,1) node[above] {$\ss {m}$};
\node[text width=1cm] at (0.4,-3.5){$\frac{x}{1-\alpha x}$};
\end{tikzpicture}\\
\end{tabular}
\end{align}

Then define $L^{*}$ acting on $F_i\otimes V_j$ as follows:
\begin{align}
\label{dcanGLmat}
L^{*}_{i,j}(x)
=
{\frac{1}{1-\alpha x}}
\begin{pmatrix}
x & x \phid_i \\
(1+\beta x)\phi_i & \delta_{0,m}(1-\alpha x)+(1-\delta_{0m})(1+\beta x) 
\end{pmatrix}_{i,j}
\end{align}

\subsubsection{\texorpdfstring{$R$}{R}-matrix and Yang--Baxter relations.}
Consider the vector spaces $F_i, F_j$ where $i<j$. Then we define a ${R}$-matrix which acts linearly on $F_i\otimes F_j$ as follows,
$$
{R}_{i,j}(x_i,x_j):\ket{a}\otimes \ket{b}
\mapsto  \sum_{c,d \hspace{1mm} \text{where} \hspace{1mm}  a+b=c+d}
R^{a,d}_{b,c}(x_i,x_j)
\ket{c}\otimes \ket{d}.
$$
Graphically, we represent the entry $R^{a,d}_{b,c}$ as
\begin{tikzpicture}[baseline=-3pt,scale=0.9]
\draw[fermionic,arrow=0.35,Gcol] (-2,0.5) node[left,black] {$a$} -- (-1,-0.5) node[right,black] {$c$};
\draw[fermionic,arrow=0.35,Gcol] (-2,-0.5) node[left,black] {$b$} -- (-1,0.5) node[right,black] {$d$};
\end{tikzpicture}.
We now give the $R$ matrix that underpins the integrability of the vertex model presented above. 
For convenience, let us represent $\ket{0}$ and $\ket{1}$ of $F$ as empty or occupied:
\begin{align}
\label{GRmatrix}
 R(x,y)=
 \begin{pmatrix}
\begin{tikzpicture}[scale=0.6,baseline=4pt]
\draw[Gcol] (0,0)--(1,1);
\draw[Gcol] (0,1)--(1,0);
\ebull{0}{1}{0.1};
\ebull{1}{0}{0.1};
\ebull{1}{1}{0.1};
\ebull{0}{0}{0.1};
\end{tikzpicture}& 0 & 0 & 0 \\[1ex]
0 & \begin{tikzpicture}[scale=0.6,baseline=4pt]
\draw[Gcol] (0,0)--(1,1);
\draw[Gcol] (0,1)--(1,0);
\bbull{0}{0}{0.1};
\ebull{0}{1}{0.1};
\ebull{1}{0}{0.1};
\bbull{1}{1}{0.1};
\end{tikzpicture} & \begin{tikzpicture}[scale=0.6,baseline=4pt]
\draw[Gcol] (0,0)--(1,1);
\draw[Gcol] (0,1)--(1,0);
\ebull{0}{1}{0.1};
\ebull{1}{1}{0.1};
\bbull{1}{0}{0.1};
\bbull{0}{0}{0.1};
\end{tikzpicture} & 0 \\[1ex]
0 &\begin{tikzpicture}[scale=0.6,baseline=4pt]
\draw[Gcol] (0,0)--(1,1);
\draw[Gcol] (0,1)--(1,0);
\bbull{0}{1}{0.1};
\ebull{0}{0}{0.1};
\ebull{1}{0}{0.1};
\bbull{1}{1}{0.1};
\end{tikzpicture} & \begin{tikzpicture}[scale=0.6,baseline=4pt]
\draw[Gcol] (0,0)--(1,1);
\draw[Gcol] (0,1)--(1,0);
\bbull{1}{0}{0.1};
\bbull{0}{1}{0.1};
\ebull{1}{1}{0.1};
\ebull{0}{0}{0.1};
\end{tikzpicture} & 0 \\[1ex]
0 & 0 & 0 & \begin{tikzpicture}[scale=0.6,baseline=4pt]
\draw[Gcol] (0,0)--(1,1);
\draw[Gcol] (0,1)--(1,0);
\bbull{1}{0}{0.1};
\bbull{0}{1}{0.1};
\bbull{1}{1}{0.1};
\bbull{0}{0}{0.1};
\end{tikzpicture} \\
\end{pmatrix}_{ij}
=
\begin{pmatrix}
1 & 0 & 0 & 0 \\[1ex]
0 & 0 & { {\frac{(1+\beta x)}{(1+\beta y)} \frac{y}{x}}} & 0 \\[1ex]
0 & 1 & {1}- \frac{(1+\beta x)}{(1+\beta y)} \frac{y}{x} & 0 \\[1ex]
0 & 0 & 0 & 1 \\
\end{pmatrix}_{ij}
\in 
{\rm End}(F_i \otimes F_j).
\end{align}
One recognizes this as the $R$-matrix of the five-vertex model \cite{HWKK-5v} with spectral parameter $\frac{x}{1-\beta x}$.
It can be obtained as a limit of the $R$ matrix of the stochastic six-vertex model where the quantum parameter is sent to $0$.

Together with $L_i$ and $L_j$ matrices, $R_{ij}$ satisfies the $\mathrm {RLL}$ relation  \text{in End} $(F_i \otimes F_j \otimes V_n)$:
\begin{equation}
\label{rll}
R_{ij}(x,y)L_i(x)L_j(y)=L_j(y)L_i(x)R_{ij}(x,y)
\quad
\left(
\begin{tikzpicture}[scale=0.6,baseline=7pt]
\draw[bosonic](1,-1)--(1,2);
\draw[fermionic,arrow=0.15,Gcol,rounded corners] (-1,1) node[left,black]{$\ss x$}--(0,0)--(2,0);
\draw[fermionic,arrow=0.15,Gcol,rounded corners] (-1,0) node[left,black]{$\ss y$}--(0,1)--(2,1);
\end{tikzpicture}
=
\begin{tikzpicture}[scale=0.6,baseline=7pt]
\draw[bosonic] (1,2)--(1,-1);
\draw[fermionic,arrow=0.15,Gcol,rounded corners] (0,0) node[left,black]{$\ss y$}--(2,0)--(3,1);
\draw[fermionic,arrow=0.15,Gcol,rounded corners](0,1) node[left,black]{$\ss x$}-- (2,1)--(3,0);
\end{tikzpicture}
\right)
\end{equation}

\subsubsection{ Transfer matrices.}
We shall now build a vertex model based on the $L$-matrix above. It is convenient to define a single row of the model as on the following picture:
\[
w(\{i_1,i_2,\dots\};\{k_1,k_2,\dots\})=
\begin{tikzpicture}[scale=0.5,baseline=-1pt]
\draw[arrow=0.05,fermionic,Gcol] (-1,0)node[left,black]{$*$} -- (7,0) node[right,black] {$\ss 0$};
\foreach\x in {0,1,...,6}{
\draw[arrow=0.25,bosonic] (\x,-1) -- (\x,1);
}
\node[below] at (0,-0.8) {\tiny $i_1$};\node[above] at (0,0.8) {\tiny $k_1$};
\node[below] at (1,-0.8) {\tiny $i_2$};\node[above] at (1,0.8) {\tiny $k_2$};
\node[below] at (2,-0.8) {\tiny $i_3$};\node[above] at (2,0.8) {\tiny $k_3$};
\node[below] at (6,-0.8) {$\cdots$};\node[above] at (6,0.8) {$\cdots$};
\end{tikzpicture}
\]
where the $*$ on the left means that we are allowing for arbitrary edge states.
Even though we are considering infinitely large row of vertices, the weight is uniquely defined. To see this, fix the labels on the top and bottom. Since there are only finitely many non zero labels on top and bottom, sufficiently far to the right the horizontal labels are constant, and we choose them to be $0$s. Graphically, we show this by assigning $0$ to the horizontal edge on the far right. Then, when the bottom and top labels are fixed, there is a unique configuration because of the local conservation around every vertex. 

We now define the corresponding \emph{transfer matrix}\/ $T$ which acts linearly on $V^r$ as follows,
\begin{align}
\label{transferG}
T(x):\ket{i_1}\otimes \ket{i_2}\otimes {\cdots}
\mapsto  \sum_{k_1,k_2\ldots \ge 0}
w(\{i_1,i_2,\dots\};\{k_1,k_2,\dots\})
\ket{k_1}\otimes \ket{k_2}\otimes {\cdots}
\end{align}
One can rewrite in terms of the $L$-matrix as
\begin{equation}
    T(x) =\lim_{n\to\infty} \bra{*} L_{01}(x)L_{02}(x)\ldots L_{0n}(x) \ket{0}
\end{equation}
where the vector space attached to the horizontal line is labelled $0$, whereas the vertical lines are labelled $1,2,\ldots$. Here $\ket{0}$ is the basis vector of the horizontal space, whereas $\bra{*}$ is the sum of basis vectors of the dual of the horizontal space.
The limit is entry-wise and is well-defined because of the aforementioned uniqueness of the configuration.

Similarly, we can define the dual transfer matrices $T^*$:
\begin{equation}
\label{dualtransferG}
{T^{*}}(x):\ket{i_1}\otimes \ket{i_2}\otimes {\cdots}
\mapsto  \sum_{k_1,k_2,\cdots\geq 0}
w^{*}(\{i_1,i_2,\dots\};\{k_1,k_2,\dots\})
\ket{k_1}\otimes \ket{k_2}\otimes {\cdots}
\end{equation}
where the right boundary is fixed to be $1$:
\[
w^{*}(\{i_1,i_2,\dots\};\{k_1,k_2,\dots\})=
\begin{tikzpicture}[scale=0.5,baseline=-1pt]
\draw[arrow=0.05,fermionic,Gcol] (-1,0)node[left,black]{$*$} -- (7,0) node[right,black] {$\ss 1$};
\node[left] at (-0.8,0) {\tiny };\node[right] at (6.8,0) {\tiny };
\foreach\x in {0,1,...,6}{
\draw[arrow=0.25,bosonic] (\x,-1) -- (\x,1);
}
\node[below] at (0,-0.8) {\tiny $i_1$};\node[above] at (0,0.8) {\tiny $k_1$};
\node[below] at (1,-0.8) {\tiny $i_2$};\node[above] at (1,0.8) {\tiny $k_2$};
\node[below] at (2,-0.8) {\tiny $i_3$};\node[above] at (2,0.8) {\tiny $k_3$};
\node[below] at (6,-0.8) {$\cdots$};\node[above] at (6,0.8) {$\cdots$};
\end{tikzpicture}
\]
Equivalently,
\begin{equation}
    T^*(x) =\lim_{n\to\infty} \bra{*} L^*_{01}(x)L_{02}(x)\ldots L^*_{0n}(x) \ket{1}
\end{equation}

Throughout this paper we use the same conventions to define transfer matrices. 

\subsubsection{Commutation relation of the transfer matrices.}
In order to prove that the transfer matrices commute, we need an eigenvector property of the $R$ matrix. Observe that the sum of the entries in a column of the $R$ matrix is always $1$. This means that the state which is the sum of all possible states is an eigenvector of the $R$ matrix with eigenvalue $1$. This property can reinterpreted as the fact that the partition function of a single vertex with fixed boundaries on the right is always $1$:
\[
\begin{tikzpicture}[scale=0.8,baseline=5pt]
\draw[fermionic,Gcol] (0,1) node[left,black]{$*$}--(1,0);
\draw[fermionic,Gcol] (0,0) node[left,black]{$*$}--(1,1);
\node at (-0.8,0) {$\ss y$};
\node at (-0.8,1) {$\ss x$};
\end{tikzpicture}
=
\begin{tikzpicture}[scale=0.8,baseline=5pt]
\draw[fermionic,Gcol] (0,1) node[left,black]{$*$}--(1,1);
\draw[fermionic,Gcol] (0,0) node[left,black]{$*$}--(1,0);
\node at (-0.8,0) {$\ss y$};
\node at (-0.8,1) {$\ss x$};
\end{tikzpicture}
\]
Consider the product of two transfer matrices, $T(x)$ and $T(y)$. Graphically, taking the product amounts to stacking the two row to row transfer matrices one upon the other. Finally, observe that the boundary on the left is free and for sufficiently large $n$, the boundary on the right is fixed. Recall that an edge with $*$ is a free boundary. Thus
$T(x)T(y)$ is
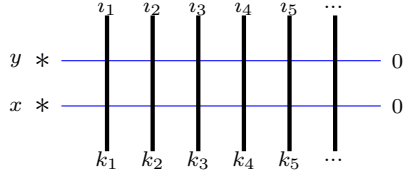
\[
\begin{tikzpicture}[scale=0.6]
\draw[fermionic,Gcol] (0,0) node[left,black]{$ *$}--(7,0) node[right,black]{$\ss 0$};
\draw[fermionic,Gcol] (0,1) node[left,black]{$ *$}--(7,1) node[right,black]{$\ss 0$};
\foreach \x in {1,2,3,4,5}{
\draw[bosonic] (\x,-1)--(\x,2);
\node at (\x,2.2) {$\ss i_{\x}$};
\node at (\x,-1.2) {$\ss k_{\x}$};
}
\draw[bosonic] (6,-1)--(6,2);
\node at (6,2.2) {$\ss \dots$};
\node at (6,-1.2) {$\ss \dots$};
\node at (-1,0) {$\ss x$};
\node at (-1,1) {$\ss y$};
\end{tikzpicture}
\]

Now multiply $T(x)T(y)$ on the left by $R(x,y)$, and apply the $\mathrm{RLL}$ relation finitely many times:
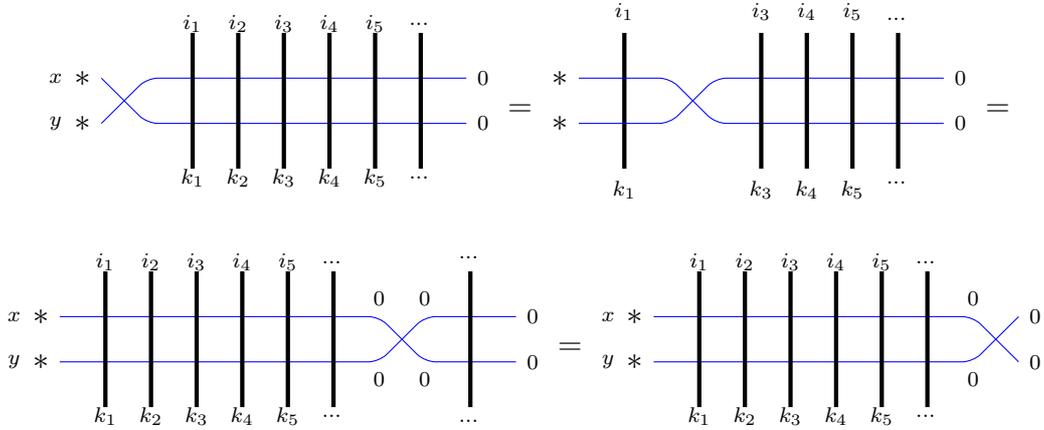
\[
\begin{tikzpicture}[scale=0.6,baseline=3pt]
\draw[fermionic,Gcol,rounded corners] (-1,1)node[black,left]{$*$}--(0,0) node[below,black] {}--(7,0) node[black,right]{$\ss 0$};
\draw[fermionic,Gcol,rounded corners](-1,0)node[black,left]{$*$}-- (0,1)node[above,black] {}--(7,1)  node[black,right]{$\ss 0$};
\foreach \x in {1,2,3,4,5}{
\draw[bosonic] (\x,-1)--(\x,2);
\node at (\x,2.2) {$\ss i_{\x}$};
\node at (\x,-1.2) {$\ss k_{\x}$};
}
\draw[bosonic] (6,-1)--(6,2);
\node at (6,2.2) {$\ss \dots$};
\node at (6,-1.2) {$\ss \dots$};
\node at (-2,1) {$\ss x$};
\node at (-2,0) {$\ss y$};
\end{tikzpicture}
=
\begin{tikzpicture}[scale=0.6,baseline=3pt]
\draw[bosonic] (0,-1)node[below]{$\ss k_1$}--(0,2)node[above]{$\ss i_1$};
\draw[fermionic,Gcol,rounded corners] (-1,0) node[left,black]{$*$}--(1,0)--(2,1)--(7,1)node[right,black]{$\ss 0$};
\draw[fermionic,Gcol,rounded corners] (-1,1)node[left,black]{$*$}--(1,1)--(2,0)--(7,0)node[right,black]{$\ss 0$};
\foreach\x in {3,4,5}{
\draw[bosonic](\x,-1) node[below] {$\ss k_{\x}$}--(\x,2) node[above]{$\ss i_{\x}$ };}
\draw[bosonic] (6,-1)node[below]{$\ss \dots$}--(6,2)node[above]{$\ss \dots$};
\end{tikzpicture}
=
\]

\[
\begin{tikzpicture}[scale=0.6,baseline=3pt]
\draw[fermionic,Gcol,rounded corners] (0,0) node[left,black]{$*$}--(7,0) node[below,black]{$\ss 0$}--(8,1) node[above,black]{$\ss 0$}--(10,1) node[right,black]{$\ss 0$};
\draw[fermionic,Gcol,rounded corners] (0,1)node[left,black]{$*$}--(7,1)node[above,black]{$\ss 0$}--(8,0) node[below,black]{$\ss 0$}--(10,0) node[right,black]{$\ss 0$};
\foreach \x in {1,2,3,4,5}{
\draw[bosonic] (\x,-1)--(\x,2);
\node at (\x,2.2)  {$\ss i_{\x}$};
\node at (\x,-1.2) {$\ss k_{\x}$};
}
\draw[bosonic] (9,-1) node[below]{$\ss \dots$}--(9,2) node[above]{$\ss \dots$};
\draw[bosonic] (6,-1)--(6,2);
\node at (6,2.2)  {$\ss \dots$};
\node at (6,-1.2) {$\ss \dots$};
\node at (-1,0) {$\ss y$};
\node at (-1,1) {$\ss x$};
\end{tikzpicture}
=
\begin{tikzpicture}[scale=0.6,baseline=3pt]
\draw[fermionic,Gcol,rounded corners] (0,0) node[left,black]{$*$}--(7,0) node[below,black]{$\ss 0$}--(8,1) node[right,black]{$\ss 0$};
\draw[fermionic,Gcol,rounded corners] (0,1)node[left,black]{$*$}--(7,1)node[above,black]{$\ss 0$}--(8,0) node[right,black]{$\ss 0$};
\foreach \x in {1,2,3,4,5}{
\draw[bosonic] (\x,-1)--(\x,2);
\node at (\x,2.2)  {$\ss i_{\x}$};
\node at (\x,-1.2) {$\ss k_{\x}$};
}
\draw[bosonic] (6,-1)--(6,2);
\node at (6,2.2)  {$\ss \dots$};
\node at (6,-1.2) {$\ss \dots$};
\node at (-1,0) {$\ss y$};
\node at (-1,1) {$\ss x$};
\end{tikzpicture}
\]
Sufficiently far to the right, we are left with a cross where all nodes $0$:
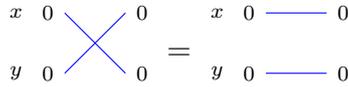
\[
\begin{tikzpicture}[scale=0.8,baseline=5pt]
\draw[fermionic,Gcol] (0,1) node[left,black]{$\ss 0$} --(1,0)node[right,black]{$\ss 0$};
\draw[fermionic,Gcol] (0,0) node[left,black]{$\ss 0$} --(1,1)node[right,black]{$\ss 0$};
\node at (-0.8,0) {$\ss y$};
\node at (-0.8,1) {$\ss x$};
\end{tikzpicture}
=
\begin{tikzpicture}[scale=0.8,baseline=5pt]
\draw[fermionic,Gcol] (0,1) node[left,black]{$\ss 0$} --(1,1) node[right,black]{$\ss 0$};
\draw[fermionic,Gcol] (0,0) node[left,black]{$\ss 0$} --(1,0)node[right,black]{$\ss 0$};
\node at (-0.8,0) {$\ss y$};
\node at (-0.8,1) {$\ss x$};
\end{tikzpicture}
\]

Since, the entry of the $R$ matrix where all the nodes are $0$ is $1$, we get $T(y)T(x)$. 
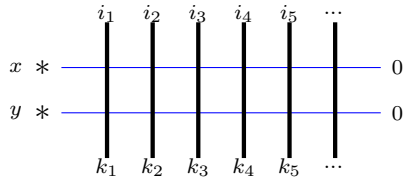
\[
\begin{tikzpicture}[scale=0.6,baseline=3pt]
\draw[fermionic,Gcol,rounded corners] (0,0)node[left,black]{$*$}--(7,0) node[right,black]{$\ss 0$};
\draw[fermionic,Gcol,rounded corners] (0,1) node[left,black]{$*$}--(7,1)node[right,black]{$\ss 0$};
\foreach \x in {1,2,3,4,5}{
\draw[bosonic] (\x,-1)--(\x,2);
\node at (\x,2.2)  {$\ss i_{\x}$};
\node at (\x,-1.2) {$\ss k_{\x}$};
}
\draw[bosonic] (6,-1)--(6,2);
\node at (6,2.2)  {$\ss \dots$};
\node at (6,-1.2) {$\ss \dots$};
\node at (-1,0) {$\ss y$};
\node at (-1,1) {$\ss x$};
\end{tikzpicture}
\]
\subsubsection{Canonical Grothendieck polynomials.}
Given that the transfer matrices commute, the polynomials defined using them are invariant under permutation of the variables. It is also not hard to see that $T(0)=1$, so that these polynomials satisfy the stability property which makes them collectively an element of $\tilde\Lambda$. We now prove that the polynomials defined using $T$ are \emph{canonical Grothendieck polynomials}.

Before we prove it, let us recall the branching formula for $\G$ from  \cite[Proposition~8.8]{canon-groth}. For a partition $\nu=(\nu_1,\nu_2,\nu_3,\dots)$, denote $\Bar{\nu}=(\nu_2,\nu_3,\dots)$.

We have
\begin{align}
\label{branchG}
\G(x_1,\dots,x_n,x_{n+1})= \sum_{\lambda/\mu  \text{ hor. strip}}  G^{(\alpha,\beta)}_{\mu}(x_1,\dots,x_n) G^{(\alpha,\beta)}_{\lambda\sslash\mu}(x_{n+1}),
\end{align}
and
\begin{align}
\label{singleG}
G^{(\alpha,\beta)}_{\lambda\sslash\mu}(x)= 
\left(
\frac{x}{1-\alpha x}
\right)^{|\lambda/\mu|}
\left(
\frac{1+\beta x}{1-\alpha x}
\right)^{r(\mu/\Bar{\lambda})},
\end{align}
where $r(\lambda/\mu)$ is number of non zero rows of $\lambda/\mu$.

\begin{rmk*}
In order to dispel any confusion, we point out that $G_{\lambda/\sslash\mu}(x_1,\dots,x_n)$ polynomials are not the same as skew Grothendieck polynomials $G_{\lambda/\mu}$. For a simple counter example, observe that for any partition $\lambda$, we have
$G^{(\alpha,\beta)}_{\lambda\sslash \lambda}(x)=\br{\dfrac{1+\beta x}{1-\alpha x}}^{r(\lambda/\tilde{\lambda})}\neq G^{(\alpha,\beta)}_{\lambda/\lambda}(x)=1$. 
\end{rmk*}
Let us now look at an example to understand $r(\mu/\Bar{\lambda})$. Consider partitions $\lambda=(4,3,2,1)$  and $\mu=(3,2,2,1)$. Then, $\Bar{\lambda}=(3,2,1)$ and $r(\mu/\Bar{\lambda})=2$.
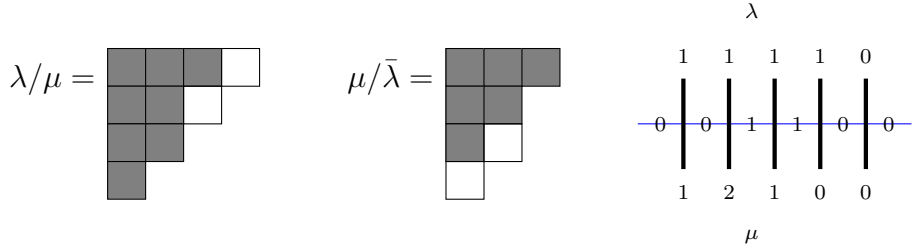
\[
\begin{tikzpicture}[scale=0.6,baseline=-7pt]
\node at (-4,0) {$\mu/\Bar{\lambda}=
\tableau{\graycell&\graycell&\graycell\\
\graycell&\graycell\\
\graycell&\\
\\}$};
\draw[fermionic,Gcol] (0,0) -- (6,0);
\foreach\x in {1,2,3,4,5}{
\draw[bosonic] (\x,-1) -- (\x,1);
}
\node at (1,-1.5) {$\ss 1$};
\node at (2,-1.5) {$\ss 2$};
\node at (3,-1.5) {$\ss 1$};
\node at (4,-1.5) {$\ss 0$};
\node at (5,-1.5) {$\ss 0$};
\node at (1,1.5) {$\ss 1$};
\node at (2,1.5) {$\ss 1$};
\node at (3,1.5) {$\ss 1$};
\node at (4,1.5) {$\ss 1$};
\node at (5,1.5) {$\ss 0$};
\node at (0.5,0) {$\ss 0$};
\node at (1.5,0) {$\ss 0$};
\node at (2.5,0) {$\ss 1$};
\node at (3.5,0) {$\ss 1$};
\node at (4.5,0) {$\ss 0$};
\node at (5.5,0) {$\ss 0$};
\node at (2.5,-2.5) {$\ss \mu$};
\node at (2.5,2.5) {$\ss \lambda$};
\node at (-11,0) {$\lambda/\mu=
\tableau{\graycell &\graycell& \graycell&\\
\graycell&\graycell&\\
\graycell&\graycell\\
\graycell\\}$};
\end{tikzpicture} 
\]
We can alternatively formulate $r(\mu/\tilde{\lambda})$ as the number of removable boxes of $\mu$ that do not lie in the same column with any box of $\lambda/\mu$.

As a consequence of recording partitions with row multiplicities, every vertex with a non zero bottom node corresponds to a removable box of $\mu$. If a box is added to $i^{th}$ column of $\mu$, then the removable box corresponding to that vertex at site $i$ will be in the same column as the new box. So, $r(\mu/\tilde{\lambda})$ is precisely the number of vertices with zero label on the left node and a non zero label on the bottom node.

\begin{thm}
\label{thm:rowG}
The canonical Grothendieck polynomials $\G(x)$  are given by
\begin{align}
\label{J}
\G(x_1,\dots,x_n)
&=
\bra{0}
T(x_1)
\dots
T(x_n)
\ket{\lambda}
\\
\label{dJ}
\G(x_1,\dots,x_n)
&=
\bra{\lambda}
T^{*}(x_n)
\dots
T^{*}(x_1)
\ket{0}
\end{align}
where $\ket{\lambda} = \bigotimes_{i=1}^{\infty} \ket{m_i(\lambda)}$, and similarly for the dual state $\bra{\lambda}$.
\end{thm}

\begin{proof}
We shall prove \eqref{J}, and \eqref{dJ} follows immediately as a consequence of the way we obtained the dual tiles. Fix $\lambda$, then we can just consider the finite transfer matrix of size $\lambda_1$. After inserting the partition states, we have the branching formula,

\begin{align*}
    \G(x_1,\dots,x_n,x)
&=
\sum_{\mu}
\bra{0}
T(x_1)
\dots
T(x_n)
\ket{\mu}
\bra{\mu}
T(x)
\ket{\lambda}.
\end{align*}

On comparing the branching formula for $\G$, 
\begin{align*}
\G(x_1,\dots,x_n,x)= \sum_{\lambda/\mu  \text{ hor. strip}} G^{(\alpha,\beta)}_{\mu}(x_1,\dots,x_n) G^{(\alpha,\beta)}_{\lambda\sslash \mu}(x)
\end{align*}
it is enough to show $G^{(\alpha,\beta)}_{\lambda\sslash\mu}(x)=\bra{\mu}T(x)\ket{\lambda}$.
For a horizontal strip $\lambda/\mu$, we have
\begin{align*}
G^{(\alpha,\beta)}_{\lambda\sslash\mu}(x)= 
\left(
\frac{x}{1-\alpha x}
\right)^{|\lambda/\mu|}
\left(
\frac{1+\beta x}{1-\alpha x}
\right)^{r(\mu/\Bar{\lambda})}.
\end{align*}

Based on the tiles one easily observes that $\bra{\mu}T(x) \ket{\lambda}\neq 0$ if and only if $\lambda/\mu$ is a horizontal strip. The label $1$ on the left edge at site $i$ amounts to adding a box in $i^{th}$ column from left. For every such vertex, we get $\frac{x}{1-\alpha x}$. From our previous analysis, we see that $r(\mu/\Bar{\lambda})$ is exactly the number of vertices with the label $0$ on the left edge and a non zero label on the bottom node.
\end{proof}

\begin{ex*} For partition $\lambda=(1,0)$, we have
\[
\begin{tikzpicture}[scale=1]
\draw[fermionic,Gcol] (0.5,0) node[left,black]{$\ss 1$}--(1.5,0) node[above,black]{$\ss 0$} --(2.5,0) node[above,black]{$\ss 0$}--(3.5,0) node[right,black]{$\ss 0$};
\draw[fermionic,Gcol] (0.5,1) node[left,black]{$\ss 0$}--(1.5,1) node[above,black]{$\ss 0$}--(2.5,1)node[above,black]{$\ss 0$}--(3.5,1)node[right,black]{$\ss 0$};
\draw[bosonic]  (1,-0.5) node[below] {$\ss 0$}--(1,0.5) node [right] {$\ss 1$}--(1,1.5) node[above] {$\ss 1$};
\draw[bosonic]  (2,-0.5)node[below] {$\ss 0$}--(2,0.5) node[right]{$\ss 0$}--(2,1.5) node[above] {$\ss 0$};
\draw[bosonic]  (3,-0.5)node[below] {$\ss 0$}--(3,0.5) node[right]{$\ss 0$}--(3,1.5)node[above] {$\ss 0$};
\node at (-0.5,0) {$x_1$};
\node at (-0.5,1) {$x_2$};
\end{tikzpicture}
\qquad
\begin{tikzpicture}[scale=1]
\draw[fermionic,Gcol] (0.5,0) node[left,black]{$\ss 0$}--(1.5,0) node[above,black]{$\ss 0$} --(2.5,0) node[above,black]{$\ss 0$}--(3.5,0) node[right,black]{$\ss 0$};
\draw[fermionic,Gcol] (0.5,1) node[left,black]{$\ss 1$}--(1.5,1) node[above,black]{$\ss 0$}--(2.5,1)node[above,black]{$\ss 0$}--(3.5,1)node[right,black]{$\ss 0$};
\draw[bosonic]  (1,-0.5) node[below] {$\ss 0$}--(1,0.5) node [right] {$\ss 0$}--(1,1.5) node[above] {$\ss 1$};
\draw[bosonic]  (2,-0.5)node[below] {$\ss 0$}--(2,0.5) node[right]{$\ss 0$}--(2,1.5) node[above] {$\ss 0$};
\draw[bosonic]  (3,-0.5)node[below] {$\ss 0$}--(3,0.5) node[right]{$\ss 0$}--(3,1.5)node[above] {$\ss 0$};
\end{tikzpicture}
\]
\[
G^{(\alpha,\beta)}_{\lambda}(x_1,x_2)=\br{\frac{x_1}{1-\alpha x_1}}\br{\frac{1+\beta x_2}{1-\alpha x_2}}+\br{\frac{ x_2}{1-\alpha x_2}}
\]
\end{ex*}

\begin{ex*} For partition $\lambda=(2,0)$, we have
\[
\begin{tikzpicture}[scale=1]
\draw[fermionic,Gcol] (0.5,0) node[left,black]{$\ss 1$}--(1.5,0) node[above,black]{$\ss 1$} --(2.5,0) node[above,black]{$\ss 0$}--(3.5,0) node[right,black]{$\ss 0$};
\draw[fermionic,Gcol] (0.5,1) node[left,black]{$\ss 0$}--(1.5,1) node[above,black]{$\ss 0$}--(2.5,1)node[above,black]{$\ss 0$}--(3.5,1)node[right,black]{$\ss 0$};
\draw[bosonic]  (1,-0.5) node[below] {$\ss 0$}--(1,0.5) node [right] {$\ss 0$}--(1,1.5) node[above] {$\ss 0$};
\draw[bosonic]  (2,-0.5)node[below] {$\ss 0$}--(2,0.5) node[right]{$\ss 1$}--(2,1.5) node[above] {$\ss 1$};
\draw[bosonic]  (3,-0.5)node[below] {$\ss 0$}--(3,0.5) node[right]{$\ss 0$}--(3,1.5)node[above] {$\ss 0$};
\node at (-0.5,0) {$x_1$};
\node at (-0.5,1) {$x_2$};
\end{tikzpicture}
\qquad
\begin{tikzpicture}[scale=1]
\draw[fermionic,Gcol] (0.5,0) node[left,black]{$\ss 1$}--(1.5,0) node[above,black]{$\ss 0$} --(2.5,0) node[above,black]{$\ss 0$}--(3.5,0) node[right,black]{$\ss 0$};
\draw[fermionic,Gcol] (0.5,1) node[left,black]{$\ss 0$}--(1.5,1) node[above,black]{$\ss 1$}--(2.5,1)node[above,black]{$\ss 0$}--(3.5,1)node[right,black]{$\ss 0$};
\draw[bosonic]  (1,-0.5) node[below] {$\ss 0$}--(1,0.5) node [right] {$\ss 1$}--(1,1.5) node[above] {$\ss 0$};
\draw[bosonic]  (2,-0.5)node[below] {$\ss 0$}--(2,0.5) node[right]{$\ss 0$}--(2,1.5) node[above] {$\ss 1$};
\draw[bosonic]  (3,-0.5)node[below] {$\ss 0$}--(3,0.5) node[right]{$\ss 0$}--(3,1.5)node[above] {$\ss 0$};
\end{tikzpicture}
\qquad
\begin{tikzpicture}[scale=1]
\draw[fermionic,Gcol] (0.5,0) node[left,black]{$\ss 0$}--(1.5,0) node[above,black]{$\ss 0$} --(2.5,0) node[above,black]{$\ss 0$}--(3.5,0) node[right,black]{$\ss 0$};
\draw[fermionic,Gcol] (0.5,1) node[left,black]{$\ss 1$}--(1.5,1) node[above,black]{$\ss 1$}--(2.5,1)node[above,black]{$\ss 0$}--(3.5,1)node[right,black]{$\ss 0$};
\draw[bosonic]  (1,-0.5) node[below] {$\ss 0$}--(1,0.5) node [right] {$\ss 0$}--(1,1.5) node[above] {$\ss 0$};
\draw[bosonic]  (2,-0.5)node[below] {$\ss 0$}--(2,0.5) node[right]{$\ss 0$}--(2,1.5) node[above] {$\ss 1$};
\draw[bosonic]  (3,-0.5)node[below] {$\ss 0$}--(3,0.5) node[right]{$\ss 0$}--(3,1.5)node[above] {$\ss 0$};
\end{tikzpicture}
\]
\[
G^{(\alpha,\beta)}_{\lambda}(x_1,x_2)=\br{\frac{x_1}{1-\alpha x_1}}^{2}\br{\frac{1+\beta x_2}{1-\alpha x_2}}+\br{\frac{x_1}{1-\alpha x_1}}\br{\frac{x_2}{1-\alpha x_2}}\br{\frac{1+\beta x_2}{1-\alpha x_2}}+\br{\frac{ x_2}{1-\alpha x_2}}^{2}
\]
\end{ex*}
\begin{ex*} For partition $\lambda=(1,1)$, we have
\[
\begin{tikzpicture}[scale=1]
\draw[fermionic,Gcol] (0.5,0) node[left,black]{$\ss 1$}--(1.5,0) node[above,black]{$\ss 0$} --(2.5,0) node[above,black]{$\ss 0$}--(3.5,0) node[right,black]{$\ss 0$};
\draw[fermionic,Gcol] (0.5,1) node[left,black]{$\ss 1$}--(1.5,1) node[above,black]{$\ss 0$}--(2.5,1)node[above,black]{$\ss 0$}--(3.5,1)node[right,black]{$\ss 0$};
\draw[bosonic]  (1,-0.5) node[below] {$\ss 0$}--(1,0.5) node [right] {$\ss 1$}--(1,1.5) node[above] {$\ss 2$};
\draw[bosonic]  (2,-0.5)node[below] {$\ss 0$}--(2,0.5) node[right]{$\ss 0$}--(2,1.5) node[above] {$\ss 0$};
\draw[bosonic]  (3,-0.5)node[below] {$\ss 0$}--(3,0.5) node[right]{$\ss 0$}--(3,1.5)node[above] {$\ss 0$};
\node at (-0.5,0) {$x_1$};
\node at (-0.5,1) {$x_2$};
\end{tikzpicture}
\]
\[
G^{(\alpha,\beta)}_{\lambda}(x_1,x_2)=\br{\frac{x_1}{1-\alpha x_1}}\br{\frac{x_2}{1-\alpha x_2}}
\]
\end{ex*}

\subsection{Row vertex model for dual canonical Grothendieck polynomials.}
\label{rowmodeldualg}
In this section, we consider a similar vertex model as the one introduced in \cref{rowcGmodel}, but with a bosonic auxiliary line. This means that that we shall associate an infinite dimensional vector space to the values a horizontal line can carry. The Boltzmann weights of the vertices are the following:
\begin{align}
\label{dGboltz}
w^{}_{x}
\left(
\begin{gathered}
\begin{tikzpicture}[scale=0.4,baseline=-2pt]
\draw[bosonic,gcol,arrow=0.25] (-1,0) node[left,black] {$a$} -- (1,0) node[right,black] {$c$};
\draw[bosonic,arrow=0.25] (0,-1) node[below] {$b$} -- (0,1) node[above] {$d$};
\end{tikzpicture}
\end{gathered}
\right)
\equiv
w^{}_{x}(a,b;c,d)
= \delta_{a+b,c+d} 
\begin{cases}
(\alpha+\beta)^{a-d-1} (x+\alpha) \beta^{d}& a>d\\
\beta^{a-1} x & 0<a\leq d\\
1& a=0,
\end{cases}
\end{align}
where 
$a,b,c,d$ $\in \mathbb{Z}_{\geq 0}$.

Let $W$ = Span$\{| j \rangle\}_{j \in \mathbb{Z}_{\geq 0}}$ be an infinite dimensional vector space, and for $1\leq i \leq n $, let $W_i$ be a copy of $W$. Then we define a $l$ matrix which acts linearly on $W_i\otimes V_j$ as follows,
\begin{align}
{l}_{i,j}\left(x_i\right):\ket{a}\otimes \ket{b}
\mapsto  \sum_{c,d \hspace{1mm} \text{where} \hspace{1mm} a+b=c+d}
w_{{x_i}}
\Big(a,b;c,d\Big)
\ket{c}\otimes \ket{d}.
\end{align}

Let $w(\{i_1,i_2,\dots\};\{k_1,k_2,\dots\})$ be the weight of single row of vertices.
\[
w(\{i_1,i_2,\dots\};\{k_1,k_2,\dots\})=
\begin{tikzpicture}[scale=0.5,baseline=-1pt]
\draw[arrow=0.05,bosonic,gcol] (-1,0)node[left,black]{$*$}-- (7,0) node[right,black]{$\ss 0$};
\node[left] at (-0.8,0) {\tiny };\node[right] at (6.8,0) {\tiny };
\foreach\x in {0,1,...,6}{
\draw[arrow=0.25,bosonic] (\x,-1) -- (\x,1);
}
\node[below] at (0,-0.8) {\tiny $i_1$};\node[above] at (0,0.8) {\tiny $k_1$};
\node[below] at (1,-0.8) {\tiny $i_2$};\node[above] at (1,0.8) {\tiny $k_2$};
\node[below] at (2,-0.8) {\tiny $i_3$};\node[above] at (2,0.8) {\tiny $k_3$};
\node[below] at (6,-0.8) {$\cdots$};\node[above] at (6,0.8) {$\cdots$};
\end{tikzpicture}
\]

We now define the transfer matrix $t$ which acts linearly on $V^c$ as follows,

\begin{align}
\label{rowtransferg}
t(x):\ket{i_1}\otimes \ket{i_2}\otimes {\cdots}
\mapsto  \sum_{k_1,k_2\ldots \ge 0}
w(\{i_1,i_2,\dots\};\{k_1,k_2,\dots\})
\ket{k_1}\otimes \ket{k_2}\otimes {\cdots}
\end{align}

As the horizontal lines are bosonic, we now represent $r$ matrix as a cross of thick lines $\left(
\begin{gathered}
\begin{tikzpicture}[baseline=-3pt,scale=0.9]
\draw[bosonic,arrow=0.35,gcol] (-2,0.5) node[left] {} -- (-1,-0.5) node[below] {};
\draw[bosonic,arrow=0.35,gcol] (-2,-0.5) node[left] {} -- (-1,0.5) node[above] {};
\end{tikzpicture}
\end{gathered}
\right).$
Consider the vector spaces $W_i, W_j$ where $i<j$. Define a ${r}$-matrix which acts linearly on $W_i\otimes W_j$ as follows,
\begin{align}
\label{rowdrmatrix}  
{r}_{i,j}(x_i,x_j):\ket{a}\otimes \ket{b}
\mapsto  \sum_{c,d \hspace{1mm} \text{where} \hspace{1mm}  a+b=c+d}
r^{a,d}_{b,c}(x_i,x_j)
\ket{c}\otimes \ket{d}.
\end{align}
where the entries of $r$ matrix here are the following:

\begin{align}
r^{a,d}_{b,c} (x,y)=
\begin{gathered}
\begin{tikzpicture}[baseline=-3pt,scale=0.9]
\draw[bosonic,arrow=0.35,gcol] (-2,0.5) node[left,black] {$a$} -- (-1,-0.5) node[below] {};
\node[text width=1cm] at (-0.4,-0.55){$c$};
\draw[bosonic,arrow=0.35,gcol] (-2,-0.5) node[left,black] {$b$} -- (-1,0.5) node[above] {};
\node[text width=1cm] at (-0.4,0.5){$d$};
\end{tikzpicture}
\end{gathered}=
\begin{cases}
0 & b>c\\
1 & b=c=0\\
\dfrac{y}{x} & b=c>0\\
\left(1- \dfrac{y}{x}\right) \left(1-\dfrac{y}{\beta} \right)^{a-d-1} & b=0\\
\left(1- \dfrac{y}{x} \right) \left(1-\dfrac{y}{\beta} \right)^{a-d-1} \left(\dfrac{y}{\beta}\right) & b>0\\
\end{cases}
\end{align}

Together with $l_i$ and $l_j$ matrices, $r_{ij}$ satisfies the $\mathrm {RLL}$ relation \text{in End} $(W_i \otimes W_j \otimes V_n)$.
\begin{equation}
\label{rowdrll}
r_{ij}(x,y)l_i(x)l_j(y)=l_j(y)l_i(x)r_{ij}(x,y)
\quad
\left(
\begin{tikzpicture}[scale=0.6,baseline=7pt]
\draw[bosonic](1,-1)--(1,2);
\draw[bosonic,gcol,arrow=0.15,rounded corners] (-1,1) node[left,black]{$\ss x$}--(0,0)--(2,0);
\draw[bosonic,gcol,arrow=0.15,rounded corners] (-1,0) node[left,black]{$\ss y$}--(0,1)--(2,1);
\end{tikzpicture}
=
\begin{tikzpicture}[scale=0.6,baseline=7pt]
\draw[bosonic,gcol,arrow=0.15,rounded corners] (0,0) node[left,black]{$\ss y$}--(2,0)--(3,1);
\draw[bosonic,gcol,arrow=0.15,rounded corners] (0,1)node[left,black]{$\ss x$}--(2,1)--(3,0);
\draw[bosonic] (1,2)--(1,-1);
\end{tikzpicture}
\right)
\end{equation}

\begin{rmk*}
Observe that the $r$ matrix is not defined for $\beta =0$. So, we shall study a different model for for $g^{(\alpha,0)}_{\lambda}$ polynomials in \cref{vmodelforj}.
\end{rmk*}

\subsubsection{Eigenvector property of the $r$-matrix.}
We proceed as in the previous section, showing an eigenvector property for the $r$-matrix in order to prove the commutation of transfer matrices. Formulated differently, we show that the partition function with a single vertex and fixed right boundary:
\[
\begin{tikzpicture}[baseline=-3pt,rounded corners]
\node at (-3.5,0) {$ \mathcal{Z}(c,d) =$};
\draw[bosonic,arrow=0.35,gcol] (-2,0.5) node[left,black] {$*$} -- (-1,-0.5) node[below] {};
\node[text width=1cm] at (-0.4,-0.55){$c$};
\draw[bosonic,arrow=0.35,gcol] (-2,-0.5) node[left,black] {$*$} -- (-1,0.5) node[above] {};
\node[text width=1cm] at (-0.4,0.5){$d$};
\end{tikzpicture}
\]
(where $c$, $d$ are non-negative integers) is constant and equal to $1$. 

We compute:
\begin{align*}
    \mathcal{Z}(c,d)&= \sum^{c}_{i=0} \begin{tikzpicture}[baseline=-3pt,rounded corners]
\draw[bosonic,arrow=0.35,gcol] (-2,0.5) node[left,black] {$c+d-i$} -- (-1,-0.5) node[below] {};
\node[text width=1cm] at (-0.4,-0.55){$c$};
\draw[bosonic,arrow=0.35,gcol] (-2,-0.5) node[left,black] {$i$} -- (-1,0.5) node[above] {};
\node[text width=1cm] at (-0.4,0.5){$d$};
\end{tikzpicture}\\
&= \begin{tikzpicture}[baseline=-3pt,rounded corners]
\draw[bosonic,arrow=0.35,gcol] (-2,0.5) node[left,black] {$c+d$} -- (-1,-0.5) node[below] {};
\node[text width=1cm] at (-0.4,-0.55){$c$};
\draw[bosonic,arrow=0.35,gcol] (-2,-0.5) node[left,black] {$0$} -- (-1,0.5) node[above] {};
\node[text width=1cm] at (-0.4,0.5){$d$};
\end{tikzpicture}
+\sum^{c-1}_{i=1}
\begin{tikzpicture}[baseline=-3pt,rounded corners]
\draw[bosonic,arrow=0.35,gcol] (-2,0.5) node[left,black] {$c+d-i$} -- (-1,-0.5) node[below] {};
\node[text width=1cm] at (-0.4,-0.55){$c$};
\draw[bosonic,arrow=0.35,gcol] (-2,-0.5) node[left,black] {$i$} -- (-1,0.5) node[above] {};
\node[text width=1cm] at (-0.4,0.5){$d$};
\end{tikzpicture}
+
\begin{tikzpicture}[baseline=-3pt,rounded corners]
\draw[bosonic,arrow=0.35,gcol] (-2,0.5) node[left,black] {$d$} -- (-1,-0.5) node[below] {};
\node[text width=1cm] at (-0.4,-0.55){$c$};
\draw[bosonic,arrow=0.35,gcol] (-2,-0.5) node[left,black] {$c$} -- (-1,0.5) node[above] {};
\node[text width=1cm] at (-0.4,0.5){$d$};
\end{tikzpicture}\\
&= \left(1- \dfrac{y}{x}\right) \left(1-\dfrac{y}{\beta} \right)^{c-1}+
\sum^{c-1}_{i=1}
\left(1- \dfrac{y}{x} \right) \left(1-\dfrac{y}{\beta} \right)^{c-i-1} \left(\dfrac{y}{\beta}\right)+
\dfrac{y}{x}\\
&= \left(1- \dfrac{y}{x}\right) \left(1-\dfrac{y}{\beta} \right)^{c-1} +
\dfrac{y}{\beta}\left(\dfrac{1-\dfrac{y}{x}}{1-\dfrac{y}{\beta}}\right)\sum^{c-1}_{i=1}\left(1-\dfrac{y}{\beta} \right)^{i}+
\dfrac{y}{x}\\
&= \left(1- \dfrac{y}{x}\right) \left(1-\dfrac{y}{\beta} \right)^{c-1} +
\dfrac{y}{\beta}\left(\dfrac{1-\dfrac{y}{x}}{1-\dfrac{y}{\beta}}\right)  \dfrac{\left(1-\dfrac{y}{\beta}\right) \left(  1-\left(1-\dfrac{y}{\beta}\right)^{c-1}\right)}{1-\left(1-\dfrac{y}{\beta} \right)}
+\dfrac{y}{x}\\
&= \left(1- \dfrac{y}{x}\right) \left(1-\dfrac{y}{\beta} \right)^{c-1} +\left(
1-\dfrac{y}{x}\right)  \left(  1-\left(1-\dfrac{y}{\beta}\right)^{c-1}\right)
+\dfrac{y}{x}\\
&= \left(1- \dfrac{y}{x}\right) \left(1-\dfrac{y}{\beta} \right)^{c-1} +\left(
1-\dfrac{y}{x}\right) - \left(
1-\dfrac{y}{x}\right)\left(1-\dfrac{y}{\beta}\right)^{c-1}
+\dfrac{y}{x}\\
&=1
\end{align*}

The $r$-matrix has the eigenvector property implies the commutation relation
\[
t(x)t(y)=t(y)t(x).
\]
Therefore, the polynomials defined using $t$ are invariant under permutation of variables.
\subsubsection{Canonical dual Grothendieck polynomials.}

In order to formulate the branching formula for $\g$, we need to establish some statistics on partitions.

For a skew-partition $\lambda/\mu$, define 
\begin{align*}
r(\lambda/\mu)&= \text{number of non zero rows,}\\
c(\lambda/\mu)&= \text{number of non zero columns,}\\
b(\lambda/\mu)&= \text{number of connected components.}\\
\end{align*}

Let us now recall the branching formula of $\g$ from \cite[Theorem~8.6]{canon-groth}.
For $\lambda,\mu$, we have
\begin{align}
\label{branchg}
\g(x_1,\dots,x_n,x_{n+1})= \sum_{\mu\subseteq \lambda}  g^{(\alpha,\beta)}_{\mu}(x_1,\dots,x_n) g^{(\alpha,\beta)}_{\lambda/\mu}(x_{n+1}),
\end{align}
where
\begin{multline*}
g_{\lambda/\mu}^{(\alpha,\beta)}(x)=\\
\begin{cases}
\beta^{r(\lambda/\mu)-b(\lambda/\mu)}(\alpha+\beta)^{\lambda/\mu - r(\lambda/\mu)-c(\lambda/\mu)+b(\lambda/\mu)} x^{b(\lambda/\mu)}{(\alpha+x)}^{c(\lambda/\mu)-b(\lambda/\mu)}& \mu \subseteq \lambda\\
0& \text{otherwise.}
\end{cases}
\end{multline*}

Let us observe some examples to understand the above statistics.

\begin{tikzpicture}[scale=0.6]
\node at (2,-1) {$
\tableau{\graycell &\graycell& &\\
\graycell&\graycell&\\
\graycell&\graycell\\
&\\
\\}$};
\node at (11,-1) {$
\tableau{\graycell &\graycell& &\\
\graycell&\graycell&\\
\graycell&\graycell&\\
&\\
\\}$};
\node at (20,-1) {$
\tableau{\graycell &\graycell& &\\
\graycell&\graycell&\\
\graycell&\graycell&\\
& &\\
\\}$};
\node at (2,2) {$\lambda/\mu= (4,3,2,2,1)/(2,2,2)$};
\node at (11,2) {$\lambda/\mu= (4,3,3,2,1)/(2,2,2)$};
\node at (20,2) {$\lambda/\mu= (4,3,3,3,1)/(2,2,2)$};
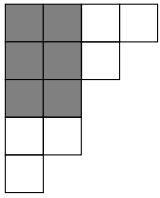
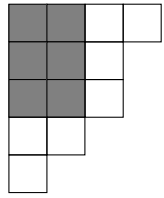
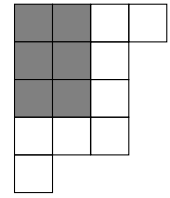
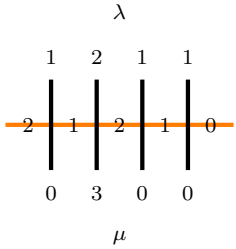
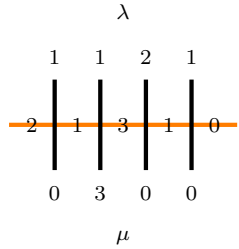
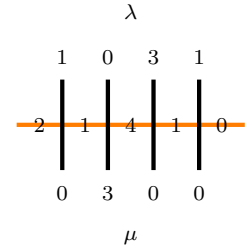
\node at (1.5,-5.5) {
$\begin{aligned}
    r(\lambda/\mu)&=4\\
    c(\lambda/\mu)&=4\\
    b(\lambda/\mu)&=2\\
\end{aligned}$
};
\node at (10.5,-5.5) {
$\begin{aligned}
    r(\lambda/\mu)&=5\\
    c(\lambda/\mu)&=4\\
    b(\lambda/\mu)&=2\\
\end{aligned}$
};
\node at (20.5,-5.5) {
$\begin{aligned}
    r(\lambda/\mu)&=5\\
    c(\lambda/\mu)&=4\\
    b(\lambda/\mu)&=1\\
\end{aligned}$
};
\draw[bosonic,gcol] (-1,-11)-- (4,-11);
\foreach\x in{0,1,2,3} {
\draw[bosonic] (\x,-12)--(\x,-10);
};
\node at (0,-12.5) {$\ss 0$};
\node at (1,-12.5) {$\ss 3$};
\node at (2,-12.5) {$\ss 0$};
\node at (3,-12.5) {$\ss 0$};
\node at (0,-9.5) {$\ss 1$};
\node at (1,-9.5) {$\ss 2$};
\node at (2,-9.5) {$\ss 1$};
\node at (3,-9.5) {$\ss 1$};
\node at (-0.5,-11) {$\ss 2$};
\node at (0.5,-11) {$\ss 1$};
\node at (1.5,-11) {$\ss 2$};
\node at (2.5,-11) {$\ss 1$};
\node at (3.5,-11) {$\ss 0$};
\node at (1.5,-13.5) {$\ss \mu$};
\node at (1.5,-8.5) {$\ss \lambda$};
\draw[bosonic,gcol] (8,-11)-- (13,-11);
\foreach\x in{0,1,2,3} {
\draw[bosonic] (9+\x,-12)--(9+\x,-10);
};
\node at (9,-12.5) {$\ss 0$};
\node at (10,-12.5) {$\ss 3$};
\node at (11,-12.5) {$\ss 0$};
\node at (12,-12.5) {$\ss 0$};
\node at (9,-9.5) {$\ss 1$};
\node at (10,-9.5) {$\ss 1$};
\node at (11,-9.5) {$\ss 2$};
\node at (12,-9.5) {$\ss 1$};
\node at (8.5,-11) {$\ss 2$};
\node at (9.5,-11) {$\ss 1$};
\node at (10.5,-11) {$\ss 3$};
\node at (11.5,-11) {$\ss 1$};
\node at (12.5,-11) {$\ss 0$};
\node at (10.5,-13.5) {$\ss \mu$};
\node at (10.5,-8.5) {$\ss \lambda$};
\draw[bosonic,gcol] (17,-11)-- (22,-11);
\foreach\x in{0,1,2,3} {
\draw[bosonic] (18+\x,-12)--(18+\x,-10);
};
\node at (18,-12.5) {$\ss 0$};
\node at (19,-12.5) {$\ss 3$};
\node at (20,-12.5) {$\ss 0$};
\node at (21,-12.5) {$\ss 0$};
\node at (18,-9.5) {$\ss 1$};
\node at (19,-9.5) {$\ss 0$};
\node at (20,-9.5) {$\ss 3$};
\node at (21,-9.5) {$\ss 1$};
\node at (17.5,-11) {$\ss 2$};
\node at (18.5,-11) {$\ss 1$};
\node at (19.5,-11) {$\ss 4$};
\node at (20.5,-11) {$\ss 1$};
\node at (21.5,-11) {$\ss 0$};
\node at (19.5,-13.5) {$\ss \mu$};
\node at (19.5,-8.5) {$\ss \lambda$};
\end{tikzpicture}

We now unpack the information contained at a vertex. Consider a vertex $\left(
\begin{tikzpicture}[scale=0.4,baseline=-2pt]
\draw[bosonic,gcol,arrow=0.25] (-1,0) node[left,black] {$a$} -- (1,0) node[right,black] {$c$};
\draw[bosonic,arrow=0.25] (0,-1) node[below] {$b$} -- (0,1) node[above] {$d$};
\end{tikzpicture}\right)$ at site $i$. The left node $a$, corresponds to adding $a$ boxes to $i^{th}$ column of $\mu$. The node $d$, is the number rows of $\lambda$ with size $i$. We now want to understand number of row of size $i$ in $\lambda/\mu$. There are three types of vertices, $b<c$, $b>d$, and $b=c$. Let us look at the nodes on $i^{th}$ column of $\lambda/\mu$ interms of the Young diagram.
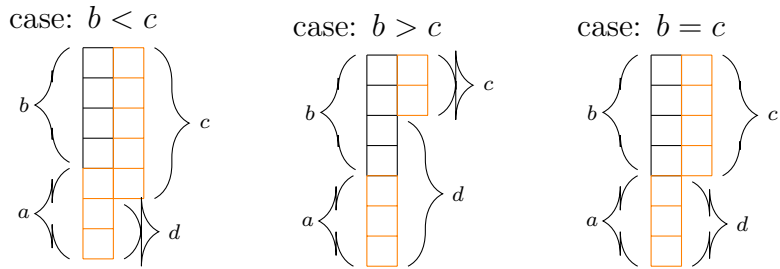
\[
\begin{tikzpicture}[scale=0.4]
\node at (0,5) {$\text{case: } b<c$};
\draw (0,0)--(0,4);
\draw (1,0)--(1,4);
\foreach\x in {0,1,2,3,4}{
\draw (0,\x)--(1,\x);}
\draw [decorate,decoration={brace,amplitude=13pt},xshift=-4pt,yshift=0pt]
(-0.3,0.2) -- (-0.3,4) node [black,midway,xshift=-0.6cm] 
{ $\ss b$};
\draw[orange] (1,4) grid (2,-1);
\draw [decorate,decoration={brace,amplitude=13pt,mirror,raise=4pt},yshift=0pt]
(2,-1) -- (2,4) node [black,midway,xshift=0.8cm] 
{ $\ss c$};
\draw[orange] (0,0) grid (1,-3);
\draw [decorate,decoration={brace,amplitude=13pt},xshift=-4pt,yshift=0pt]
(-0.3,-3) -- (-0.3,0) node [black,midway,xshift=-0.6cm] 
{ $\ss a$};
\draw [decorate,decoration={brace,amplitude=13pt,mirror,raise=4pt},yshift=0pt]
(1,-3) -- (1,-1.2) node [black,midway,xshift=0.8cm] 
{ $\ss d$};
\end{tikzpicture}
\qquad
\begin{tikzpicture}[scale=0.4]
\node at (0,5) {$\text{case: } b>c$};
\draw (0,0)--(0,4);
\draw (1,0)--(1,4);
\foreach\x in {0,1,2,3,4}{
\draw (0,\x)--(1,\x);}
\draw [decorate,decoration={brace,amplitude=13pt},xshift=-2pt,yshift=0pt]
(-0.3,0.2) -- (-0.3,4) node [black,midway,xshift=-0.6cm] 
{ $\ss b$};
\draw[orange] (1,4) grid (2,2);
\draw [decorate,decoration={brace,amplitude=13pt,mirror,raise=4pt},yshift=0pt]
(2.,2) -- (2.,4) node [black,midway,xshift=0.8cm] 
{ $\ss c$};
\draw[orange] (0,0) grid (1,-3);
\draw [decorate,decoration={brace,amplitude=13pt},xshift=-4pt,yshift=0pt]
(-0.3,-3) -- (-0.3,0) node [black,midway,xshift=-0.6cm] 
{ $\ss a$};
\draw [decorate,decoration={brace,amplitude=13pt,mirror,raise=4pt},yshift=0pt]
(1,-3) -- (1,1.8) node [black,midway,xshift=0.8cm] 
{ $\ss d$};
\end{tikzpicture}
\qquad
\begin{tikzpicture}[scale=0.4]
\node at (0,5) {$\text{case: } b=c$};
\draw (0,0)--(0,4);
\draw (1,0)--(1,4);
\foreach\x in {0,1,2,3,4}{
\draw (0,\x)--(1,\x);}
\draw [decorate,decoration={brace,amplitude=13pt},xshift=-4pt,yshift=0pt]
(-0.3,0.2) -- (-0.3,4) node [black,midway,xshift=-0.6cm] 
{ $\ss b$};
\draw[orange] (1,4) grid (2,0);
\draw [decorate,decoration={brace,amplitude=13pt,mirror,raise=4pt},yshift=0pt]
(2,0) -- (2,4) node [black,midway,xshift=0.8cm] 
{ $\ss c$};
\draw[orange] (0,0) grid (1,-3);
\draw [decorate,decoration={brace,amplitude=13pt},xshift=-4pt,yshift=0pt]
(-0.3,-3) -- (-0.3,0) node [black,midway,xshift=-0.6cm] 
{ $\ss a$};
\draw [decorate,decoration={brace,amplitude=13pt,mirror,raise=4pt},yshift=0pt]
(1,-3) -- (1,-0.2) node [black,midway,xshift=0.8cm] 
{ $\ss d$};
\end{tikzpicture}
\]

From the above pictures, it is evident that the number of non zero rows of size $i$ in $\lambda/\mu$ is $\min(a,d)$. Also observe that, in the last two cases, the skew diagram is disjoint. Therefore, the number of connected components is the number of vertices where $b>c$ or $b=c$. Finally, the $i^{th}$ column of $\lambda/\mu$ is non zero if and only if some boxes are added to it i.e.,  when $a\neq 0$. 

\begin{thm}
\label{thm:rowg}
The dual Canonical Grothendieck polynomials $\g(x)$  are given by
\begin{align}
\label{rowg}
\g(x_1,\dots,x_n)
&=
\bra{0}
t(x_1)
\dots
t(x_n)
\ket{\lambda}
\end{align}
where $\ket{\lambda} = \bigotimes_{i=1}^{\infty} \ket{m_i(\lambda)}$.
\end{thm}

\begin{proof}
Following the similar reasoning as in \cref{thm:rowG}, it enough to show that for $\mu\subseteq\lambda$,

\[
g^{(\alpha,\beta)}_{\lambda/\mu}(x)=\bra{\mu}t(x)\ket{\lambda}.
\]
where
\begin{multline*}
g_{\lambda/\mu}^{(\alpha,\beta)}(x)=\\
\begin{cases}
\beta^{r(\lambda/\mu)-b(\lambda/\mu)}(\alpha+\beta)^{\lambda/\mu - r(\lambda/\mu)-c(\lambda/\mu)+b(\lambda/\mu)} x^{b(\lambda/\mu)}{(\alpha+x)}^{c(\lambda/\mu)-b(\lambda/\mu)}& \mu \subseteq \lambda\\
0& \text{otherwise.}
\end{cases}
\end{multline*}

Let us study the exponent of $\beta$. Recall that $r(\lambda/\mu)$ of size $i$, is $\min(a,d)$. The connected components are recorded by vertices where $b\geq c$. When $b<c$, assign $\beta^{d}$ and otherwise $\beta^{a-1}$. Then we get that the overall weight as $\beta^{r(\lambda/\mu)-b(\lambda/\mu)}$. Observe that this is precisely the $\beta$ factor of the Boltzmann weights. Similarly, by doing the same to each factor, one recovers the Boltzmann weights.
\end{proof}

\begin{ex*} For partition $\lambda=(1,0)$, we have
\[
\begin{tikzpicture}[scale=1]
\draw[bosonic,gcol] (0.5,0) node[left,black]{$\ss 1$}--(1.5,0) node[above,black]{$\ss 0$} --(2.5,0) node[above,black]{$\ss 0$}--(3.5,0) node[right,black]{$\ss 0$};
\draw[bosonic,gcol] (0.5,1) node[left,black]{$\ss 0$}--(1.5,1) node[above,black]{$\ss 0$}--(2.5,1)node[above,black]{$\ss 0$}--(3.5,1)node[right,black]{$\ss 0$};
\draw[bosonic]  (1,-0.5) node[below] {$\ss 0$}--(1,0.5) node [right] {$\ss 1$}--(1,1.5) node[above] {$\ss 1$};
\draw[bosonic]  (2,-0.5)node[below] {$\ss 0$}--(2,0.5) node[right]{$\ss 0$}--(2,1.5) node[above] {$\ss 0$};
\draw[bosonic]  (3,-0.5)node[below] {$\ss 0$}--(3,0.5) node[right]{$\ss 0$}--(3,1.5)node[above] {$\ss 0$};
\node at (-0.5,0) {$x_1$};
\node at (-0.5,1) {$x_2$};
\end{tikzpicture}
\qquad
\begin{tikzpicture}[scale=1]
\draw[bosonic,gcol] (0.5,0) node[left,black]{$\ss 0$}--(1.5,0) node[above,black]{$\ss 0$} --(2.5,0) node[above,black]{$\ss 0$}--(3.5,0) node[right,black]{$\ss 0$};
\draw[bosonic,gcol] (0.5,1) node[left,black]{$\ss 1$}--(1.5,1) node[above,black]{$\ss 0$}--(2.5,1)node[above,black]{$\ss 0$}--(3.5,1)node[right,black]{$\ss 0$};
\draw[bosonic]  (1,-0.5) node[below] {$\ss 0$}--(1,0.5) node [right] {$\ss 0$}--(1,1.5) node[above] {$\ss 1$};
\draw[bosonic]  (2,-0.5)node[below] {$\ss 0$}--(2,0.5) node[right]{$\ss 0$}--(2,1.5) node[above] {$\ss 0$};
\draw[bosonic]  (3,-0.5)node[below] {$\ss 0$}--(3,0.5) node[right]{$\ss 0$}--(3,1.5)node[above] {$\ss 0$};
\end{tikzpicture}
\]
\[
g^{(\alpha,\beta)}_{\lambda}(x_1,x_2)=x_1+x_2
\]
\end{ex*}
\begin{ex*} For partition $\lambda=(2,0)$, we have
\[
\begin{tikzpicture}[scale=1]
\draw[bosonic,gcol] (0.5,0) node[left,black]{$\ss 1$}--(1.5,0) node[above,black]{$\ss 1$} --(2.5,0) node[above,black]{$\ss 0$}--(3.5,0) node[right,black]{$\ss 0$};
\draw[bosonic,gcol] (0.5,1) node[left,black]{$\ss 0$}--(1.5,1) node[above,black]{$\ss 0$}--(2.5,1)node[above,black]{$\ss 0$}--(3.5,1)node[right,black]{$\ss 0$};
\draw[bosonic]  (1,-0.5) node[below] {$\ss 0$}--(1,0.5) node [right] {$\ss 0$}--(1,1.5) node[above] {$\ss 0$};
\draw[bosonic]  (2,-0.5)node[below] {$\ss 0$}--(2,0.5) node[right]{$\ss 1$}--(2,1.5) node[above] {$\ss 1$};
\draw[bosonic]  (3,-0.5)node[below] {$\ss 0$}--(3,0.5) node[right]{$\ss 0$}--(3,1.5)node[above] {$\ss 0$};
\node at (-0.5,0) {$x_1$};
\node at (-0.5,1) {$x_2$};
\end{tikzpicture}
\qquad
\begin{tikzpicture}[scale=1]
\draw[bosonic,gcol] (0.5,0) node[left,black]{$\ss 1$}--(1.5,0) node[above,black]{$\ss 0$} --(2.5,0) node[above,black]{$\ss 0$}--(3.5,0) node[right,black]{$\ss 0$};
\draw[bosonic,gcol] (0.5,1) node[left,black]{$\ss 0$}--(1.5,1) node[above,black]{$\ss 1$}--(2.5,1)node[above,black]{$\ss 0$}--(3.5,1)node[right,black]{$\ss 0$};
\draw[bosonic]  (1,-0.5) node[below] {$\ss 0$}--(1,0.5) node [right] {$\ss 1$}--(1,1.5) node[above] {$\ss 0$};
\draw[bosonic]  (2,-0.5)node[below] {$\ss 0$}--(2,0.5) node[right]{$\ss 0$}--(2,1.5) node[above] {$\ss 1$};
\draw[bosonic]  (3,-0.5)node[below] {$\ss 0$}--(3,0.5) node[right]{$\ss 0$}--(3,1.5)node[above] {$\ss 0$};
\end{tikzpicture}
\qquad
\begin{tikzpicture}[scale=1]
\draw[bosonic,gcol] (0.5,0) node[left,black]{$\ss 0$}--(1.5,0) node[above,black]{$\ss 0$} --(2.5,0) node[above,black]{$\ss 0$}--(3.5,0) node[right,black]{$\ss 0$};
\draw[bosonic,gcol] (0.5,1) node[left,black]{$\ss 1$}--(1.5,1) node[above,black]{$\ss 1$}--(2.5,1)node[above,black]{$\ss 0$}--(3.5,1)node[right,black]{$\ss 0$};
\draw[bosonic]  (1,-0.5) node[below] {$\ss 0$}--(1,0.5) node [right] {$\ss 0$}--(1,1.5) node[above] {$\ss 0$};
\draw[bosonic]  (2,-0.5)node[below] {$\ss 0$}--(2,0.5) node[right]{$\ss 0$}--(2,1.5) node[above] {$\ss 1$};
\draw[bosonic]  (3,-0.5)node[below] {$\ss 0$}--(3,0.5) node[right]{$\ss 0$}--(3,1.5)node[above] {$\ss 0$};
\end{tikzpicture}
\]
\[
g^{(\alpha,\beta)}_{\lambda}(x_1,x_2)=(x_1+\alpha)x_1+x_1x_2+(x_2+\alpha)x_2=x^{2}_1+x_1x_2+x^{2}_2+\alpha(x_1+x_2)
\]
\end{ex*}
\begin{ex*} For partition $\lambda=(1,1)$, we have
\[
\begin{tikzpicture}[scale=1]
\draw[bosonic,gcol] (0.5,0) node[left,black]{$\ss 2$}--(1.5,0) node[above,black]{$\ss 0$} --(2.5,0) node[above,black]{$\ss 0$}--(3.5,0) node[right,black]{$\ss 0$};
\draw[bosonic,gcol] (0.5,1) node[left,black]{$\ss 0$}--(1.5,1) node[above,black]{$\ss 0$}--(2.5,1)node[above,black]{$\ss 0$}--(3.5,1)node[right,black]{$\ss 0$};
\draw[bosonic]  (1,-0.5) node[below] {$\ss 0$}--(1,0.5) node [right] {$\ss 2$}--(1,1.5) node[above] {$\ss 2$};
\draw[bosonic]  (2,-0.5)node[below] {$\ss 0$}--(2,0.5) node[right]{$\ss 0$}--(2,1.5) node[above] {$\ss 0$};
\draw[bosonic]  (3,-0.5)node[below] {$\ss 0$}--(3,0.5) node[right]{$\ss 0$}--(3,1.5)node[above] {$\ss 0$};
\node at (-0.5,0) {$x_1$};
\node at (-0.5,1) {$x_2$};
\end{tikzpicture}
\qquad
\begin{tikzpicture}[scale=1]
\draw[bosonic,gcol] (0.5,0) node[left,black]{$\ss 1$}--(1.5,0) node[above,black]{$\ss 0$} --(2.5,0) node[above,black]{$\ss 0$}--(3.5,0) node[right,black]{$\ss 0$};
\draw[bosonic,gcol] (0.5,1) node[left,black]{$\ss 1$}--(1.5,1) node[above,black]{$\ss 0$}--(2.5,1)node[above,black]{$\ss 0$}--(3.5,1)node[right,black]{$\ss 0$};
\draw[bosonic]  (1,-0.5) node[below] {$\ss 0$}--(1,0.5) node [right] {$\ss 1$}--(1,1.5) node[above] {$\ss 2$};
\draw[bosonic]  (2,-0.5)node[below] {$\ss 0$}--(2,0.5) node[right]{$\ss 0$}--(2,1.5) node[above] {$\ss 0$};
\draw[bosonic]  (3,-0.5)node[below] {$\ss 0$}--(3,0.5) node[right]{$\ss 0$}--(3,1.5)node[above] {$\ss 0$};
\end{tikzpicture}
\qquad
\begin{tikzpicture}[scale=1]
\draw[bosonic,gcol] (0.5,0) node[left,black]{$\ss 0$}--(1.5,0) node[above,black]{$\ss 0$} --(2.5,0) node[above,black]{$\ss 0$}--(3.5,0) node[right,black]{$\ss 0$};
\draw[bosonic,gcol] (0.5,1) node[left,black]{$\ss 2$}--(1.5,1) node[above,black]{$\ss 0$}--(2.5,1)node[above,black]{$\ss 0$}--(3.5,1)node[right,black]{$\ss 0$};
\draw[bosonic]  (1,-0.5) node[below] {$\ss 0$}--(1,0.5) node [right] {$\ss 0$}--(1,1.5) node[above] {$\ss 2$};
\draw[bosonic]  (2,-0.5)node[below] {$\ss 0$}--(2,0.5) node[right]{$\ss 0$}--(2,1.5) node[above] {$\ss 0$};
\draw[bosonic]  (3,-0.5)node[below] {$\ss 0$}--(3,0.5) node[right]{$\ss 0$}--(3,1.5)node[above] {$\ss 0$};
\end{tikzpicture}
\]
\[
g^{(\alpha,\beta)}_{\lambda}(x_1,x_2)=\beta x_1+x_1x_2+\beta x_2=x_1x_2+\beta(x_1+x_2)
\]
\end{ex*}

\section{Column Vertex Models}
\label{columnmodels}

\subsection{Definition of Physical space.}
Recall that we identify partitions with basis elements of $V^{r}$ by recording row multiplicities. In this section, for the physical space, we use the same vector space $(V^{r})$ that we used in the earlier section. We shall denote it by $V^{c}$. Even though $V^{c}$ and $V^{r}$ are identical, we distinguish them by the way we identify the partitions with the basis elements.
\begin{align}
    \label{colPhysicalspacedefn}
    V^{c}=\text{Span}
    \left\{
    \ket{m^c_1} \otimes \ket{m^c_2} \otimes \ket{m^c_3} \cdots
    \right\}
    \qquad
    m^c_i \geq0,\ i\geq 1. 
    \end{align}
    
 Given a partition, which we view as Young diagram, let $\ket{\lambda^c}$ be the basis vector with integers
\[
m^{c}_i(\lambda) = \text{number of columns of size $i$ of $\lambda$}
\] 
For example, we identify the partition $\lambda=(5,4,4,3)$ with the basis element $\ket{1}\otimes\ket{0}\otimes\ket{1}\otimes\ket{3}\otimes\ket{0}\dots$ of $V^{c}$:

\begin{center}
\label{columnphysicalspacepic}
\begin{tikzpicture}[scale=0.6]
\draw (0,4) -- (6,4);
\draw (0,3) -- (5,3);
\draw (0,2) -- (4,2);
\draw (0,1) -- (4,1);
\draw (0,0) -- (3,0);
\draw (0,-3) -- (0,4);
\draw (1,0) -- (1,4);
\draw (2,0) -- (2,4);
\draw (3,0) -- (3,4);
\draw (4,1) -- (4,4);
\draw (5,3) -- (5,4);
\bbull{0.5}{0}{0.09};
\bbull{1.5}{0}{0.09};
\bbull{2.5}{0}{0.09};
\bbull{3.5}{1}{0.09};
\bbull{4.5}{3}{0.09};
\draw[dotted,arrow=0.5] (5,3) -- (8.5,3);
\draw[dotted,arrow=0.5] (4,1) -- (8.5,1);
\draw[dotted,arrow=0.5] (3,0) -- (8.5,0);
\draw[thick,arrow=1] (9,3.5) -- (9,-3);
\foreach\x in {0,1,...,6}{
\draw[thick] (8.7,3.5-\x) -- (9,3.5-\x);
}
\bbull{9}{3}{0.09};
\bbull{9}{1}{0.09};
\bbull{8.4}{0}{0.09};\bbull{8.7}{0}{0.09};\bbull{9}{0}{0.09};
\node at (9.5,3) {$\sss m_1$};
\node at (9.5,2) {$\sss m_2$};
\node at (9.5,1) {$\sss m_3$};
\node at (9.5,0) {$\sss m_4$};
\node at (9.5,-1) {$\sss m_5$};
\node at (9.5,-2) {$\sss m_6$};
\node at (9.5,-3) {$\vdots$};
\end{tikzpicture}
\end{center}

\subsection{Column vertex model for canonical Grothendieck polynomials.}
\label{Gmodel}

\subsubsection{Definition of $\widetilde{L}$-matrix and $\widetilde{R}$ matrix.}
\label{cGColumnmodel}
The main difference of the model considered in this section from the row model of $\G$ is that the horizontal line can now carry any non negative integer. For every vertex, we assign the Boltzmann weights in the following way:
\begin{equation}
\label{Gboltz}
w^{}_{x}
\left(
\begin{gathered}
\begin{tikzpicture}[scale=0.4,baseline=-2pt]
\draw[bosonic,Gcol,arrow=0.25] (-1,0) node[left,black] {$a$} -- (1,0) node[right,black] {$c$};
\draw[bosonic,arrow=0.25] (0,-1) node[below] {$b$} -- (0,1) node[above] {$d$};
\end{tikzpicture}
\end{gathered}
\right)
\equiv
w^{}_{x}(a,b;c,d)= \delta_{a+b,c+d}
\begin{cases}
\left(\frac{x} {1-\alpha x} \right)^{a}&  b=c\\
\left(\frac{x} {1-\alpha x} \right)^{a}\left(\frac{1+\beta x} {1-\alpha x} \right) & b>c\\
0& b<c\\
\end{cases}
\quad
a,b,c,d \in \mathbb{Z}_{\geq 0},
\quad
\end{equation}

Let $W$ = Span$\{| j \rangle\}_{j \in \mathbb{Z}_{\geq 0}}$ be an infinite dimensional vector space, and for $1\leq i \leq n $, let $W_i$ be a copy of $W$. Let $V_j \cong W$ be another copy. Then we define a $\widetilde{L}$ matrix which acts linearly on $W_i\otimes V_j$ as follows:
\begin{equation}
\label{collmatG}
\widetilde{L}_{i,j}\left(x_i\right):\ket{a}\otimes \ket{b}
\mapsto  \sum_{c,d \hspace{1mm} \text{where} \hspace{1mm} a+b=c+d}
w_{{x_i}}
\Big(a,b;c,d\Big)
\ket{c}\otimes \ket{d}.
\end{equation}

Let $w(\{i_1,i_2,\dots\};\{k_1,k_2,\dots\})$ be the weight of single row of vertices.
\[
w(\{i_1,i_2,\dots\};\{k_1,k_2,\dots\})=
\begin{tikzpicture}[scale=0.5,baseline=-1pt]
\draw[bosonic,Gcol,arrow=0.05] (-1,0) node[left,black]{$*$}  -- (7,0) node[right,black]{$\ss 0$};
\node[left] at (-0.8,0) {\tiny };\node[right] at (6.8,0) {\tiny };
\foreach\x in {0,1,...,6}{
\draw[arrow=0.25,bosonic] (\x,-1) -- (\x,1);
}
\node[below] at (0,-0.8) {\tiny $i_1$};\node[above] at (0,0.8) {\tiny $k_1$};
\node[below] at (1,-0.8) {\tiny $i_2$};\node[above] at (1,0.8) {\tiny $k_2$};
\node[below] at (2,-0.8) {\tiny $i_3$};\node[above] at (2,0.8) {\tiny $k_3$};
\node[below] at (6,-0.8) {$\cdots$};\node[above] at (6,0.8) {$\cdots$};
\end{tikzpicture}
\]

We now define the transfer matrix $\widetilde{T}$ which acts linearly on $V^c$ as follows,

\begin{align}
\label{coltransferG}
\widetilde{T}(x):\ket{i_1}\otimes \ket{i_2}\otimes {\cdots}
\mapsto  \sum_{k_1,k_2\ldots \ge 0}
w(\{i_1,i_2,\dots\};\{k_1,k_2,\dots\})
\ket{k_1}\otimes \ket{k_2}\otimes {\cdots}
\end{align}

Consider the vector spaces $W_i, W_j$ where $i<j$. Then we define a ${R}$-matrix which acts linearly on $W_i\otimes W_j$ as follows,
\begin{align}
\label{BosonicGRmatrix}
\widetilde{R}_{i,j}(x_i,x_j):\ket{a}\otimes \ket{b}
\mapsto  \sum_{c,d \hspace{1mm} \text{where} \hspace{1mm}  a+b=c+d}
\widetilde{R}^{ad}_{bc}(x_i,x_j)
\ket{c}\otimes \ket{d}.
\end{align}
where the entries are:
\begin{align*}
\widetilde{R}^{a\,d}_{b\,c}\left(x,y\right) =\begin{gathered}
\begin{tikzpicture}[baseline=-3pt,scale=0.9]
\draw[bosonic,arrow=0.35,Gcol] (-2,0.5) node[left,black] {$a$} -- (-1,-0.5) node[right,black] {$c$};
\draw[bosonic,arrow=0.35,Gcol] (-2,-0.5) node[left,black] {$b$} -- (-1,0.5) node[right,black] {$d$};
\end{tikzpicture}
\end{gathered}= \left({\dfrac{\dfrac{x}{1-\alpha x}}{\dfrac{y}{1-\alpha y}}}\right)^{a}
  \begin{cases}
    0 & \text{when } b<c\\
   1 & \text{when } b=c \\
  \dfrac{1}{1-\alpha x}-\dfrac{x}{(1-\alpha x)y} & \text{otherwise } 
  \end{cases}
\end{align*}

Together with $\widetilde{L}_i$ and $\widetilde{L}_j$ matrices, $\widetilde{R}_{ij}$ satisfies the $\mathrm {RLL}$ relation \text{in End} $(W_i \otimes W_j \otimes V_n)$:
\begin{equation}
\label{colGrll}
\widetilde{R}_{ij}(x,y)\widetilde{L}_i(x)\widetilde{L}_j(y)=\widetilde{L}_j(y)\widetilde{L}_i(x)\widetilde{R}_{ij}(x,y)
\quad
\left(
\begin{tikzpicture}[scale=0.6,baseline=7pt]
\draw[bosonic](1,-1)--(1,2);
\draw[bosonic,Gcol,rounded corners,arrow=0.15] (-1,1) node[left,black]{$\ss x$}--(0,0)--(2,0);
\draw[bosonic,Gcol,rounded corners,arrow=0.15] (-1,0) node[left,black]{$\ss y$}--(0,1)--(2,1);
\end{tikzpicture}
=
\begin{tikzpicture}[scale=0.6,baseline=7pt]
\draw[bosonic] (1,2)--(1,-1);
\draw[bosonic,Gcol,rounded corners,arrow=0.15] (0,0)node[left,black]{$\ss y$}--(2,0)--(3,1);
\draw[bosonic,Gcol,rounded corners,arrow=0.15] (0,1)node[left,black]{$\ss x$}--(2,1)--(3,0);
\end{tikzpicture}
\right)
\end{equation}

\subsubsection{Eigenvector property of the $\widetilde{R}$ matrix.}
We proceed as in the previous section, showing an eigenvector property for the  $\widetilde{R}$ matrix in order to prove that the commutation of transfer matrices. Recall that an edge with $*$ is a free boundary. As in previous section, we show that the partition function function with single vertex and fixed right boundary
\[
\begin{tikzpicture}[baseline=-3pt,rounded corners]
\node at (-3.5,0) {$ \mathcal{Z}(c,d) =$};
\draw[bosonic,arrow=0.35,Gcol] (-2,0.5) node[left,black] {$*$} -- (-1,-0.5) node[below] {};
\node[text width=1cm] at (-0.4,-0.55){$c$};
\draw[bosonic,arrow=0.35,Gcol] (-2,-0.5) node[left,black] {$*$} -- (-1,0.5) node[above] {};
\node[text width=1cm] at (-0.4,0.5){$d$};
\end{tikzpicture}.
\]

(where $c,d$ are nonnegative integers) is constant and equal to 1.
We compute:
\begin{align*}
    \mathcal{Z}(c,d)&= \sum^{d}_{i=0} \begin{tikzpicture}[baseline=-3pt,rounded corners]
\draw[bosonic,arrow=0.35,Gcol] (-2,0.5) node[left,black] {$i$} -- (-1,-0.5) node[below] {};
\node[text width=1cm] at (-0.4,-0.55){$c$};
\draw[bosonic,arrow=0.35,Gcol] (-2,-0.5) node[left,black] {$c+d-i$} -- (-1,0.5) node[above] {};
\node[text width=1cm] at (-0.4,0.5){$d$};
\end{tikzpicture}\\
&= \left(\dfrac{1}{1-\alpha x}-\dfrac{x}{(1-\alpha x)y}\right) \sum^{d-1}_{i=0} \left( \dfrac{x (1-\alpha y)}{y(1-\alpha x)} \right)^{i} +\left( \dfrac{x (1-\alpha y)}{y(1-\alpha x)} \right)^{d}\\
&= \left(\dfrac{y-x}{y(1-\alpha x)}\right) \left( \dfrac{1-\left( \dfrac{x (1-\alpha y)}{y(1-\alpha x)} \right)^{d}}{1-\left( \dfrac{x (1-\alpha y)}{y(1-\alpha x)} \right)} \right)+\left( \dfrac{x (1-\alpha y)}{y(1-\alpha x)} \right)^{d}\\
&=1.
\end{align*}

The $\widetilde{R}$-matrix has the eigenvector property implies the commutation relation,
\[
\widetilde{T}(x)\widetilde{T}(y)=\widetilde{T}(y)\widetilde{T}(x).
\]

\subsubsection{Canonical Grothendieck polynomials.}
Given that the transfer matrices commute, the polynomials defined using $\widetilde{T}$ are invariant under permutation of the variables. We now prove that the polynomials defined using $T$ are \emph{canonical Grothendieck polynomials}.

\begin{thm}
\label{thm:colG}
The canonical Grothendieck polynomials $\G(x)$  are given by
\begin{align}
\label{colG}
\G(x_1,\dots,x_n)
&=
\bra{0}
\widetilde{T}(x_1)
\dots
\widetilde{T}(x_n)
\ket{\lambda^c}
\end{align}
where $\ket{\lambda^c} = \bigotimes_{i=1}^{\infty} \ket{m^c_i(\lambda)}$.
\end{thm}

\begin{proof}

Let us now understand the local configuration of vertices of this model.

\begin{tikzpicture}[scale=0.6]
\node at (2,-1) {$
\tableau{\graycell &\graycell&\graycell &\graycell&\\
\graycell&\graycell & \\
\\}$};
\node at (11,-1) {$
\tableau{\graycell &\graycell&\graycell &\graycell&\\
\graycell&\graycell & &\\
\\}$};
\node at (20,-1) {$
\tableau{\graycell &\graycell&\graycell &\graycell&\\
\graycell&\graycell & &\\
&\\}$};
\node at (2,2) {$\lambda/\mu= (5,3,1)/(4,2)$};
\node at (11,2) {$\lambda/\mu= (5,4,1)/(4,2)$};
\node at (20,2) {$\lambda/\mu= (5,4,2)/(4,2)$};
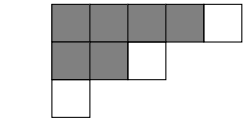
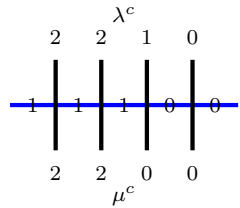
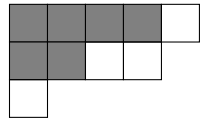
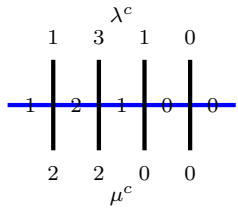
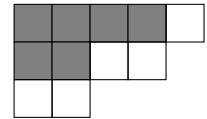
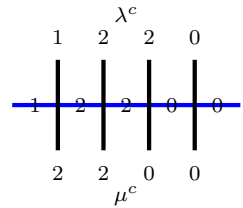
\node at (1.5,-3) {
$r(\mu/\Bar{\lambda})=2$};
\node at (10.5,-3) {
$r(\mu/\Bar{\lambda})=1$};
\node at (20.5,-3) {
$r(\mu/\Bar{\lambda})=0$};
\draw[bosonic,Gcol] (-1,-6)-- (4,-6);
\foreach\x in{0,1,2,3} {
\draw[bosonic] (\x,-7)--(\x,-5);
};
\node at (0,-7.5) {$\ss 2$};
\node at (1,-7.5) {$\ss 2$};
\node at (2,-7.5) {$\ss 0$};
\node at (3,-7.5) {$\ss 0$};
\node at (0,-4.5) {$\ss 2$};
\node at (1,-4.5) {$\ss 2$};
\node at (2,-4.5) {$\ss 1$};
\node at (3,-4.5) {$\ss 0$};
\node at (-0.5,-6) {$\ss 1$};
\node at (0.5,-6) {$\ss 1$};
\node at (1.5,-6) {$\ss 1$};
\node at (2.5,-6) {$\ss 0$};
\node at (3.5,-6) {$\ss 0$};
\node at (1.5,-8) {$\ss \mu^c$};
\node at (1.5,-4) {$\ss \lambda^c$};
\draw[bosonic,Gcol] (8,-6)-- (13,-6);
\foreach\x in{0,1,2,3} {
\draw[bosonic] (9+\x,-7)--(9+\x,-5);
};
\node at (9,-7.5) {$\ss 2$};
\node at (10,-7.5) {$\ss 2$};
\node at (11,-7.5) {$\ss 0$};
\node at (12,-7.5) {$\ss 0$};
\node at (9,-4.5) {$\ss 1$};
\node at (10,-4.5) {$\ss 3$};
\node at (11,-4.5) {$\ss 1$};
\node at (12,-4.5) {$\ss 0$};
\node at (8.5,-6) {$\ss 1$};
\node at (9.5,-6) {$\ss 2$};
\node at (10.5,-6) {$\ss 1$};
\node at (11.5,-6) {$\ss 0$};
\node at (12.5,-6) {$\ss 0$};
\node at (10.5,-8) {$\ss \mu^c$};
\node at (10.5,-4) {$\ss \lambda^c$};
\draw[bosonic,Gcol] (17,-6)-- (22,-6);
\foreach\x in{0,1,2,3} {
\draw[bosonic] (18+\x,-7)--(18+\x,-5);
};
\node at (18,-7.5) {$\ss 2$};
\node at (19,-7.5) {$\ss 2$};
\node at (20,-7.5) {$\ss 0$};
\node at (21,-7.5) {$\ss 0$};
\node at (18,-4.5) {$\ss 1$};
\node at (19,-4.5) {$\ss 2$};
\node at (20,-4.5) {$\ss 2$};
\node at (21,-4.5) {$\ss 0$};
\node at (17.5,-6) {$\ss 1$};
\node at (18.5,-6) {$\ss 2$};
\node at (19.5,-6) {$\ss 2$};
\node at (20.5,-6) {$\ss 0$};
\node at (21.5,-6) {$\ss 0$};
\node at (19.5,-8) {$\ss \mu^c$};
\node at (19.5,-4) {$\ss \lambda^c$};
\end{tikzpicture}

Consider a vertex, $\left(  \begin{tikzpicture}[scale=0.4,baseline=-2pt]
\draw[bosonic,Gcol,arrow=0.25] (-1,0) node[left,black] {$a$} -- (1,0) node[right,black] {$c$};
\draw[bosonic,arrow=0.25] (0,-1) node[below] {$b$} -- (0,1) node[above] {$d$};
\end{tikzpicture}\right)$,
at site $i$. The label $a$ corresponds to adding $a$ boxes to $i^{th}$ row of $\mu$. By recording the left nodes, we get $\lambda/\mu$. Then, in-order to get a horizontal strip, the number of boxes that can be added to $i^{th}$ row should be at most $b$. When $c<b$, we have a removable box in $i^{th}$ row that is not in the same column with any box of $\lambda/\mu$. Therefore, $r(\mu/\tilde{\lambda})$ is precisely the number of vertices where $c<b$.

Following the reasoning in \cref{thm:rowG}, it is enough to show that  $\bra{u^c}\widetilde{T}(x)\ket{\lambda^c}=G^{(\alpha,\beta)}_{\lambda\sslash\mu}$
for a horizontal strip $\lambda/\mu$. Recall that for a horizontal strip, we have
\begin{align*}
G^{(\alpha,\beta)}_{\lambda\sslash\mu}(x)= 
\left(
\frac{x}{1-\alpha x}
\right)^{|\lambda/\mu|}
\left(
\frac{1+\beta x}{1-\alpha x}
\right)^{r(\mu/\Bar{\lambda})}.
\end{align*}
Observe that $\bra{u^c}\widetilde{T}(x)\ket{\lambda^c}\neq 0$ if and only if $\lambda/\mu$ is a horizontal strip. From the above analysis and the way Boltzmann weights are defined, the proof is now immediate.
\end{proof}

\begin{ex*} For partition $\lambda=(2,0)$, we have
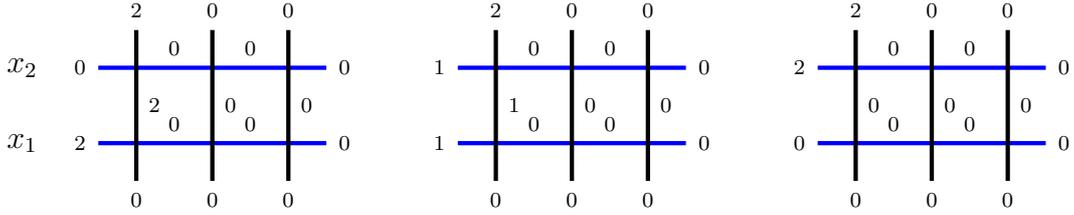
\[
\begin{tikzpicture}[scale=1]
\draw[bosonic,Gcol] (0.5,0) node[left,black]{$\ss 2$}--(1.5,0) node[above,black]{$\ss 0$} --(2.5,0) node[above,black]{$\ss 0$}--(3.5,0) node[right,black]{$\ss 0$};
\draw[bosonic,Gcol] (0.5,1) node[left,black]{$\ss 0$}--(1.5,1) node[above,black]{$\ss 0$}--(2.5,1)node[above,black]{$\ss 0$}--(3.5,1)node[right,black]{$\ss 0$};
\draw[bosonic]  (1,-0.5) node[below] {$\ss 0$}--(1,0.5) node [right] {$\ss 2$}--(1,1.5) node[above] {$\ss 2$};
\draw[bosonic]  (2,-0.5)node[below] {$\ss 0$}--(2,0.5) node[right]{$\ss 0$}--(2,1.5) node[above] {$\ss 0$};
\draw[bosonic]  (3,-0.5)node[below] {$\ss 0$}--(3,0.5) node[right]{$\ss 0$}--(3,1.5)node[above] {$\ss 0$};
\node at (-0.5,0) {$x_1$};
\node at (-0.5,1) {$x_2$};
\end{tikzpicture}
\qquad
\begin{tikzpicture}[scale=1]
\draw[bosonic,Gcol] (0.5,0) node[left,black]{$\ss 1$}--(1.5,0) node[above,black]{$\ss 0$} --(2.5,0) node[above,black]{$\ss 0$}--(3.5,0) node[right,black]{$\ss 0$};
\draw[bosonic,Gcol] (0.5,1) node[left,black]{$\ss 1$}--(1.5,1) node[above,black]{$\ss 0$}--(2.5,1)node[above,black]{$\ss 0$}--(3.5,1)node[right,black]{$\ss 0$};
\draw[bosonic]  (1,-0.5) node[below] {$\ss 0$}--(1,0.5) node [right] {$\ss 1$}--(1,1.5) node[above] {$\ss 2$};
\draw[bosonic]  (2,-0.5)node[below] {$\ss 0$}--(2,0.5) node[right]{$\ss 0$}--(2,1.5) node[above] {$\ss 0$};
\draw[bosonic]  (3,-0.5)node[below] {$\ss 0$}--(3,0.5) node[right]{$\ss 0$}--(3,1.5)node[above] {$\ss 0$};
\end{tikzpicture}
\qquad
\begin{tikzpicture}[scale=1]
\draw[bosonic,Gcol] (0.5,0) node[left,black]{$\ss 0$}--(1.5,0) node[above,black]{$\ss 0$} --(2.5,0) node[above,black]{$\ss 0$}--(3.5,0) node[right,black]{$\ss 0$};
\draw[bosonic,Gcol] (0.5,1) node[left,black]{$\ss 2$}--(1.5,1) node[above,black]{$\ss 0$}--(2.5,1)node[above,black]{$\ss 0$}--(3.5,1)node[right,black]{$\ss 0$};
\draw[bosonic]  (1,-0.5) node[below] {$\ss 0$}--(1,0.5) node [right] {$\ss 0$}--(1,1.5) node[above] {$\ss 2$};
\draw[bosonic]  (2,-0.5)node[below] {$\ss 0$}--(2,0.5) node[right]{$\ss 0$}--(2,1.5) node[above] {$\ss 0$};
\draw[bosonic]  (3,-0.5)node[below] {$\ss 0$}--(3,0.5) node[right]{$\ss 0$}--(3,1.5)node[above] {$\ss 0$};
\end{tikzpicture}
\]
\[
G^{(\alpha,\beta)}_{\lambda}(x_1,x_2)=\br{\frac{x_1}{1-\alpha x_1}}^{2}\br{\frac{1+\beta x_2}{1-\alpha x_2}}+\br{\frac{x_1}{1-\alpha x_1}}\br{\frac{ x_2}{1-\alpha x_2}}\br{\frac{1+\beta x_2}{1-\alpha x_2}}+\br{\frac{ x_2}{1-\alpha x_2}}^{2}
\]
\end{ex*}
\begin{ex*} For partition $\lambda=(1,1)$, we have
\[
\begin{tikzpicture}[scale=1]
\draw[bosonic,Gcol] (0.5,0) node[left,black]{$\ss 1$}--(1.5,0) node[above,black]{$\ss 0$} --(2.5,0) node[above,black]{$\ss 0$}--(3.5,0) node[right,black]{$\ss 0$};
\draw[bosonic,Gcol] (0.5,1) node[left,black]{$\ss 0$}--(1.5,1) node[above,black]{$\ss 1$}--(2.5,1)node[above,black]{$\ss 0$}--(3.5,1)node[right,black]{$\ss 0$};
\draw[bosonic]  (1,-0.5) node[below] {$\ss 0$}--(1,0.5) node [right] {$\ss 1$}--(1,1.5) node[above] {$\ss 0$};
\draw[bosonic]  (2,-0.5)node[below] {$\ss 0$}--(2,0.5) node[right]{$\ss 0$}--(2,1.5) node[above] {$\ss 1$};
\draw[bosonic]  (3,-0.5)node[below] {$\ss 0$}--(3,0.5) node[right]{$\ss 0$}--(3,1.5)node[above] {$\ss 0$};
\node at (-0.5,0) {$x_1$};
\node at (-0.5,1) {$x_2$};
\end{tikzpicture}
\]
\[
G^{(\alpha,\beta)}_{\lambda}(x_1,x_2)=\br{\frac{x_1}{1-\alpha x_1}}\br{\frac{x_2}{1-\alpha x_2}}
\]
\end{ex*}
\subsection{Column vertex model dual canonical Grothendieck polynomials.}

\subsubsection{Definition of $\widetilde{l}$-matrix and $\widetilde{r}$ matrix.}
\label{dualcolumnmodel}
We consider the same vertex model as row model of $\g$, but with different Boltzmann weights. For every vertex, we assign the Boltzmann weights in the following way:

\begin{align}
\label{columdGboltz}
w^{}_{x}
\left(
\begin{gathered}
\begin{tikzpicture}[scale=0.4,baseline=-2pt]
\draw[bosonic,gccol,arrow=0.25] (-1,0) node[left,black] {$a$} -- (1,0) node[right,black] {$c$};
\draw[bosonic,arrow=0.25] (0,-1) node[below] {$b$} -- (0,1) node[above] {$d$};
\end{tikzpicture}
\end{gathered}
\right)
\equiv
w^{}_{x}(a,b;c,d)
= \delta_{a+b,c+d} 
\begin{cases}
(\alpha + \beta)^{a-d-1}\beta (x+\alpha)^{d} & 0<a>d\\
 x (x+\alpha)^{a-1} & 0<a\leq d\\
1& a=0,
\end{cases}
\end{align}
where 
$a,b,c,d$ $\in \mathbb{Z}_{\geq 0}$.
Let $W$ = Span$\{| j \rangle\}_{j \in \mathbb{Z}_{\geq 0}}$ be an infinite dimensional vector space, and for $1\leq i \leq n $, let $W_i$ be a copy of $W$. Let $V_j \cong W_i$ be a vector space. Then we define a $\widetilde{l}$ matrix which acts linearly on $W_i\otimes V_j$ as follows,
\begin{align}
\label{collmatg}
\widetilde{l}_{i,j}\left(x_i\right):\ket{a}\otimes \ket{b}
\mapsto  \sum_{c,d \hspace{1mm} \text{where} \hspace{1mm} a+b=c+d}
w_{{x_i}}
\Big(a,b;c,d\Big)
\ket{c}\otimes \ket{d}.
\end{align}

As usual, let $w(\{i_1,i_2,\dots\};\{k_1,k_2,\dots\})$ be the weight of single row of vertices.
\[
w(\{i_1,i_2,\dots\};\{k_1,k_2,\dots\})=
\begin{tikzpicture}[scale=0.5,baseline=-1pt]
\draw[bosonic,gccol,arrow=0.05] (-1,0) node[left,black]{$*$}  -- (7,0) node[right,black]{$\ss 0$};
\node[left] at (-0.8,0) {\tiny };\node[right] at (6.8,0) {\tiny };
\foreach\x in {0,1,...,6}{
\draw[arrow=0.25,bosonic] (\x,-1) -- (\x,1);
}
\node[below] at (0,-0.8) {\tiny $i_1$};\node[above] at (0,0.8) {\tiny $k_1$};
\node[below] at (1,-0.8) {\tiny $i_2$};\node[above] at (1,0.8) {\tiny $k_2$};
\node[below] at (2,-0.8) {\tiny $i_3$};\node[above] at (2,0.8) {\tiny $k_3$};
\node[below] at (6,-0.8) {$\cdots$};\node[above] at (6,0.8) {$\cdots$};
\end{tikzpicture}
\]

We now define the transfer matrix $t$ which acts linearly on $V^c$ as follows,

\begin{align}
\label{coltransferg}
\tilde{t}(x):\ket{i_1}\otimes \ket{i_2}\otimes {\cdots}
\mapsto  \sum_{k_1,k_2\ldots \ge 0}
w(\{i_1,i_2,\dots\};\{k_1,k_2,\dots\})
\ket{k_1}\otimes \ket{k_2}\otimes {\cdots}
\end{align}

Consider the vector spaces $W_i, W_j$ where $i<j$. Then we define a $\widetilde{r}$-matrix which acts linearly on $W_i\otimes W_j$ as follows,
\begin{align}
\label{colrmatdg}
\widetilde{r}_{i,j}(x_i,x_j):\ket{a}\otimes \ket{b}
\mapsto  \sum_{c,d \hspace{1mm} \text{where} \hspace{1mm}  a+b=c+d}
\widetilde{r}^{a,d}_{b,c}(x_i,x_j)
\ket{c}\otimes \ket{d}.
\end{align}
where the entries are the following:

\begin{align}
\widetilde{r}^{k,l}_{i,j} (x,y)=\begin{gathered}
\begin{tikzpicture}[baseline=-3pt,scale=0.9]
\draw[bosonic,arrow=0.35,gccol] (-2,0.5) node[left,black] {$k$} -- (-1,-0.5) node[below] {};
\node[text width=1cm] at (-0.4,-0.55){$j$};
\draw[bosonic,arrow=0.35,gccol] (-2,-0.5) node[left,black] {$i$} -- (-1,0.5) node[above] {};
\node[text width=1cm] at (-0.4,0.5){$l$};
\end{tikzpicture}
\end{gathered}=
\begin{cases}
0 & i<j\\
1& k=l=0\\
\dfrac{x}{y} \left(\dfrac{y+\alpha}{x+\alpha}\right)^{1-k} & k=l>0\\
 \left(1-\dfrac{x}{y}\right)  & k=0\\
\dfrac{x}{y}  \left(\dfrac{y+\alpha}{x+\alpha} -1\right) \left(\dfrac{y+\alpha}{x+\alpha}\right)^{-k} & k>0\\
\end{cases}
\end{align}

Together with $\widetilde{l}_i$ and $\widetilde{l}_j$ matrices, $\widetilde{r}_{ij}$ satisfies the $\mathrm {RLL}$ relation\text{in End} $(W_i \otimes W_j \otimes V_n)$.
\begin{equation}
\label{colgrll}
\widetilde{r}_{ij}(x,y)\widetilde{l}_i(x)\widetilde{l}_j(y)=\widetilde{l}_j(y)\widetilde{l}_i(x)\widetilde{r}_{ij}(x,y)
\quad
\left(
\begin{tikzpicture}[scale=0.6,baseline=7pt]
\draw[bosonic](1,-1)--(1,2);
\draw[bosonic,arrow=0.15,gccol,rounded corners] (-1,1) node[left,black]{$\ss x$}--(0,0)--(2,0);
\draw[bosonic,arrow=0.15,gccol,rounded corners] (-1,0)node[left,black]{$\ss y$}--(0,1)--(2,1);
\end{tikzpicture}
=
\begin{tikzpicture}[scale=0.6,baseline=7pt]
\draw[bosonic] (1,2)--(1,-1);
\draw[bosonic,arrow=0.15,gccol,rounded corners] (0,0)node[left,black]{$\ss y$}--(2,0)--(3,1);
\draw[bosonic,arrow=0.15,gccol,rounded corners] (0,1)node[left,black]{$\ss x$}--(2,1)--(3,0);
\end{tikzpicture}
\right)
\end{equation}

\subsubsection{Eigenvector property of the $\widetilde{r}$ matrix.}
We now discuss an eigenvector property for the $\widetilde{r}$ matrix. We proceed as in previous sections, computing the partition function of a single vertex with fixed right boundary to show an eigenvector property of the $\widetilde{r}$ matrix. We claim that for any nonnegative integers $c,d$,

\[
\begin{tikzpicture}[baseline=-3pt,rounded corners]
\node at (-3.5,0) {$ \mathcal{Z}(c,d) =$};
\draw[bosonic,arrow=0.35,gccol] (-2,0.5) node[left,black] {$*$} -- (-1,-0.5) node[below] {};
\node[text width=1cm] at (-0.4,-0.55){$c$};
\draw[bosonic,arrow=0.35,gccol] (-2,-0.5) node[left,black] {$*$} -- (-1,0.5) node[above] {};
\node[text width=1cm] at (-0.4,0.5){$d$};
\end{tikzpicture}
\]
the partition function is constant and is equal to $1$. Let us first consider the case where $d=0$. Then there is a unique vertex as the bottom left entry should be greater than or equal to $c$ and also should satisfy the conservation. The weight of the unique configuration is $1$.
\[
\begin{tikzpicture}[baseline=-3pt,rounded corners]
\node at (-3.5,0) {$ \mathcal{Z}(c,0) =$};
\draw[bosonic,arrow=0.35,gccol] (-2,0.5) node[left,black] {$0$} -- (-1,-0.5) node[below] {};
\node[text width=1cm] at (-0.4,-0.55){$c$};
\draw[bosonic,arrow=0.35,gccol] (-2,-0.5) node[left,black] {$c$} -- (-1,0.5) node[above] {};
\node[text width=1cm] at (-0.4,0.5){$0$};
\end{tikzpicture}
\]

We now compute for the case where $d>0$: 
\begin{align*}
    \mathcal{Z}(c,d)&= \sum^{d}_{i=0} \begin{tikzpicture}[baseline=-3pt,rounded corners]
\draw[bosonic,arrow=0.35,gccol] (-2,0.5) node[left,black] {$i$} -- (-1,-0.5) node[below] {};
\node[text width=1cm] at (-0.4,-0.55){$c$};
\draw[bosonic,arrow=0.35,gccol] (-2,-0.5) node[left,black] {$c+d-i$} -- (-1,0.5) node[above] {};
\node[text width=1cm] at (-0.4,0.5){$d$};
\end{tikzpicture}\\
&= \begin{tikzpicture}[baseline=-3pt,rounded corners]
\draw[bosonic,arrow=0.35,gccol] (-2,0.5) node[left,black] {$0$} -- (-1,-0.5) node[below] {};
\node[text width=1cm] at (-0.4,-0.55){$c$};
\draw[bosonic,arrow=0.35,gccol] (-2,-0.5) node[left,black] {$c+d$} -- (-1,0.5) node[above] {};
\node[text width=1cm] at (-0.4,0.5){$d$};
\end{tikzpicture}+
\sum^{d-1}_{i=1}
 \begin{tikzpicture}[baseline=-3pt,rounded corners]
\draw[bosonic,arrow=0.35,gccol] (-2,0.5) node[left,black] {$i$} -- (-1,-0.5) node[below] {};
\node[text width=1cm] at (-0.4,-0.55){$c$};
\draw[bosonic,arrow=0.35,gccol] (-2,-0.5) node[left,black] {$c+d-i$} -- (-1,0.5) node[above] {};
\node[text width=1cm] at (-0.4,0.5){$d$};
\end{tikzpicture}+
\begin{tikzpicture}[baseline=-3pt,rounded corners]
\draw[bosonic,arrow=0.35,gccol] (-2,0.5) node[left,black] {$d$} -- (-1,-0.5) node[below] {};
\node[text width=1cm] at (-0.4,-0.55){$c$};
\draw[bosonic,arrow=0.35,gccol] (-2,-0.5) node[left,black] {$c$} -- (-1,0.5) node[above] {};
\node[text width=1cm] at (-0.4,0.5){$d$};
\end{tikzpicture}\\
&= \left(1-\dfrac{x}{y}\right)+ \dfrac{x}{y}\left(\dfrac{y+\alpha}{x+\alpha} -1 \right) \sum^{d-1}_{i=1}\left(\dfrac{x+\alpha}{y+\alpha}\right)^{i}+ \dfrac{x}{y} \left( \dfrac{x+\alpha}{y+\alpha}\right)^{d-1}\\
&= \left( 1-\dfrac{x}{y}\right)+ \dfrac{x}{y}\left(\dfrac{y-x}{x+\alpha}\right) \left(\dfrac{x+\alpha}{y+\alpha} \right) \left( \dfrac{\left(1-\left(\dfrac{x+\alpha}{y+\alpha}\right)^{d-1}\right)}{1-\dfrac{x+\alpha}{y+\alpha}}\right)+ \dfrac{x}{y} \left( \dfrac{x+\alpha}{y+\alpha}\right)^{d-1}\\
&= \left( 1-\dfrac{x}{y}\right)+ \dfrac{x}{y} \left(1-\left(\dfrac{x+\alpha}{y+\alpha}\right)^{d-1}\right)+ \dfrac{x}{y} \left( \dfrac{x+\alpha}{y+\alpha}\right)^{d-1}\\
&=1.
\end{align*}

The $\widetilde{r}$ matrix has the eigenvector property implies the commutation relation,
\[
\widetilde{t}(x)\widetilde{t}(y)=\widetilde{t}(y)\widetilde{t}(x).
\]
Therefore, the polynomials defined using $\widetilde{t}$ are invariant under permutation of variables.

\subsubsection{Dual canonical Grothendieck polynomials.}
Recall that for a skew-partition $\lambda/\mu$, we have

\begin{align*}
r(\lambda/\mu)&= \text{number of non zero rows,}\\
c(\lambda/\mu)&= \text{number of non zero columns,}\\
b(\lambda/\mu)&= \text{number of connected components.}\\
\end{align*}

We shall now unpack the information contained at a vertex like we did in the case of row model of $\g$. Consider a vertex
$\left(\begin{tikzpicture}[scale=0.4,baseline=-2pt]
\draw[bosonic,gccol,arrow=0.25] (-1,0) node[left,black] {$a$} -- (1,0) node[right,black] {$c$};
\draw[bosonic,arrow=0.25] (0,-1) node[below] {$b$} -- (0,1) node[above] {$d$};
\end{tikzpicture}
\right)$
at site $i$. The left node $a$, corresponds to adding $a$ boxes to $i^{th}$ row of $\mu$. The node $d$, is the number columns of $\lambda$ with size $i$. We now want to understand number of columns of size $i$ in $\lambda/\mu$. There are three types of vertices, $b<c$, $b>d$, and $b=c$. Let us look at the nodes on $i^{th}$ row of $\lambda/\mu$.
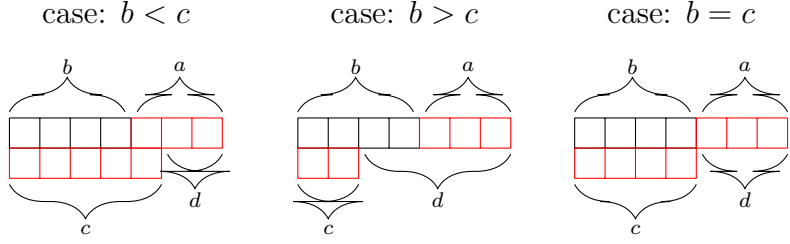
\[
\begin{tikzpicture}[scale=0.4]
\node at (3.5,4.5) {$\text{case: } b<c$};
\draw (0,0) grid (4,1);
\draw[gccol] (4,0) grid (7,1);
\draw[gccol] (0,0) grid (5,-1);
\draw [decorate,decoration={brace,amplitude=13pt}]
(0,1.2) -- (3.8,1.2) node [black,midway,yshift=0.6cm] 
{ $\ss b$};
\draw [decorate,decoration={brace,amplitude=13pt}]
(4.2,1.2) -- (7,1.2) node [black,midway,yshift=0.6cm] 
{ $\ss a$};
\draw [decorate,decoration={brace,amplitude=13pt,mirror}]
(0,-1.2) -- (5,-1.2) node [black,midway,yshift=-0.6cm] 
{ $\ss c$};
\draw [decorate,decoration={brace,amplitude=13pt,mirror}]
(5.2,-0.2) -- (7,-0.2) node [black,midway,yshift=-0.6cm] 
{ $\ss d$};
\end{tikzpicture}
\qquad
\begin{tikzpicture}[scale=0.4]
\node at (3.5,4.5) {$\text{case: } b>c$};
\draw (0,0) grid (4,1);
\draw[red] (4,0) grid (7,1);
\draw[red] (0,0) grid (2,-1);
\draw [decorate,decoration={brace,amplitude=13pt}]
(0,1.2) -- (3.8,1.2) node [black,midway,yshift=0.6cm] 
{ $\ss b$};
\draw [decorate,decoration={brace,amplitude=13pt}]
(4.2,1.2) -- (7,1.2) node [black,midway,yshift=0.6cm] 
{ $\ss a$};
\draw [decorate,decoration={brace,amplitude=13pt,mirror}]
(0,-1.2) -- (2,-1.2) node [black,midway,yshift=-0.6cm] 
{ $\ss c$};
\draw [decorate,decoration={brace,amplitude=13pt,mirror}]
(2.2,-0.2) -- (7,-0.2) node [black,midway,yshift=-0.6cm] 
{ $\ss d$};
\end{tikzpicture}
\qquad
\begin{tikzpicture}[scale=0.4]
\node at (3.5,4.5) {$\text{case: } b=c$};
\draw (0,0) grid (4,1);
\draw[red] (4,0) grid (7,1);
\draw[red] (0,0) grid (4,-1);
\draw [decorate,decoration={brace,amplitude=13pt}]
(0,1.2) -- (3.8,1.2) node [black,midway,yshift=0.6cm] 
{ $\ss b$};
\draw [decorate,decoration={brace,amplitude=13pt}]
(4.2,1.2) -- (7,1.2) node [black,midway,yshift=0.6cm] 
{ $\ss a$};
\draw [decorate,decoration={brace,amplitude=13pt,mirror}]
(0,-1.2) -- (4,-1.2) node [black,midway,yshift=-0.6cm] 
{ $\ss c$};
\draw [decorate,decoration={brace,amplitude=13pt,mirror}]
(4.2,-0.2) -- (7,-0.2) node [black,midway,yshift=-0.6cm] 
{ $\ss d$};
\end{tikzpicture}
\]

It is evident from the pictures that the number of non zero columns of size $i$in $\lambda/\mu$ is $\min(a,d)$. Also, observe that the vertices where $b\leq c$ detect the number of connected components. Finally, the node on the left edge correspond to the number of boxes added in $i^th$ row. So, the number of rows is equal to the number of vertices where $a\neq 0$.

\begin{thm}
\label{thm:colg}
The dual Canonical Grothendieck polynomials $\g(x)$  are given by
\begin{align}
\label{colg}
\g(x_1,\dots,x_n)
&=
\bra{0}
\tilde{t}(x_1)
\dots
\tilde{t}(x_n)
\ket{\lambda^c}
\end{align}
where $\ket{\lambda^c} = \bigotimes_{i=1}^{\infty} \ket{m^c_i(\lambda)}$.
\end{thm}

\begin{proof}
Following the reasoning in \cref{thm:rowG}, it enough to show that for $\mu\subseteq\lambda$,

\[
g^{(\alpha,\beta)}_{\lambda/\mu}(x)=\bra{\mu^c}t(x)\ket{\lambda^c}.
\]

Recall that for $\mu\subseteq\lambda$, we have
\begin{multline*}
g_{\lambda/\mu}^{(\alpha,\beta)}(x)=\\
\begin{cases}
\beta^{r(\lambda/\mu)-b(\lambda/\mu)}(\alpha+\beta)^{\lambda/\mu - r(\lambda/\mu)-c(\lambda/\mu)+b(\lambda/\mu)} x^{b(\lambda/\mu)}{(\alpha+x)}^{c(\lambda/\mu)-b(\lambda/\mu)}& \mu \subseteq \lambda\\
0& \text{otherwise.}
\end{cases}
\end{multline*}

Let us deal the $\beta$ factor in the branching formula. Observe that the $\beta$ factor appears in a Boltzmann weight of a vertex only when $a\neq0$ and $b<c$. From our previous analysis, we see that such vertices precisely count $r(\lambda/\mu)-b(\lambda/\nu)$. Similarly, one can check for all the other factors in the branching formula.
\end{proof}

\begin{ex*} For partition $\lambda=(2,0)$, we have
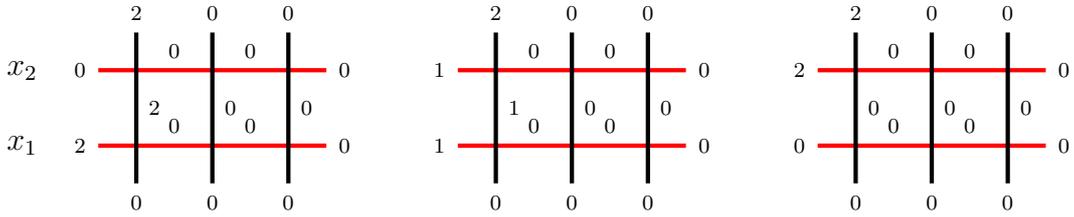
\[
\begin{tikzpicture}[scale=1]
\draw[bosonic,gccol] (0.5,0) node[left,black]{$\ss 2$}--(1.5,0) node[above,black]{$\ss 0$} --(2.5,0) node[above,black]{$\ss 0$}--(3.5,0) node[right,black]{$\ss 0$};
\draw[bosonic,gccol] (0.5,1) node[left,black]{$\ss 0$}--(1.5,1) node[above,black]{$\ss 0$}--(2.5,1)node[above,black]{$\ss 0$}--(3.5,1)node[right,black]{$\ss 0$};
\draw[bosonic]  (1,-0.5) node[below] {$\ss 0$}--(1,0.5) node [right] {$\ss 2$}--(1,1.5) node[above] {$\ss 2$};
\draw[bosonic]  (2,-0.5)node[below] {$\ss 0$}--(2,0.5) node[right]{$\ss 0$}--(2,1.5) node[above] {$\ss 0$};
\draw[bosonic]  (3,-0.5)node[below] {$\ss 0$}--(3,0.5) node[right]{$\ss 0$}--(3,1.5)node[above] {$\ss 0$};
\node at (-0.5,0) {$x_1$};
\node at (-0.5,1) {$x_2$};
\end{tikzpicture}
\qquad
\begin{tikzpicture}[scale=1]
\draw[bosonic,gccol] (0.5,0) node[left,black]{$\ss 1$}--(1.5,0) node[above,black]{$\ss 0$} --(2.5,0) node[above,black]{$\ss 0$}--(3.5,0) node[right,black]{$\ss 0$};
\draw[bosonic,gccol] (0.5,1) node[left,black]{$\ss 1$}--(1.5,1) node[above,black]{$\ss 0$}--(2.5,1)node[above,black]{$\ss 0$}--(3.5,1)node[right,black]{$\ss 0$};
\draw[bosonic]  (1,-0.5) node[below] {$\ss 0$}--(1,0.5) node [right] {$\ss 1$}--(1,1.5) node[above] {$\ss 2$};
\draw[bosonic]  (2,-0.5)node[below] {$\ss 0$}--(2,0.5) node[right]{$\ss 0$}--(2,1.5) node[above] {$\ss 0$};
\draw[bosonic]  (3,-0.5)node[below] {$\ss 0$}--(3,0.5) node[right]{$\ss 0$}--(3,1.5)node[above] {$\ss 0$};
\end{tikzpicture}
\qquad
\begin{tikzpicture}[scale=1]
\draw[bosonic,gccol] (0.5,0) node[left,black]{$\ss 0$}--(1.5,0) node[above,black]{$\ss 0$} --(2.5,0) node[above,black]{$\ss 0$}--(3.5,0) node[right,black]{$\ss 0$};
\draw[bosonic,gccol] (0.5,1) node[left,black]{$\ss 2$}--(1.5,1) node[above,black]{$\ss 0$}--(2.5,1)node[above,black]{$\ss 0$}--(3.5,1)node[right,black]{$\ss 0$};
\draw[bosonic]  (1,-0.5) node[below] {$\ss 0$}--(1,0.5) node [right] {$\ss 0$}--(1,1.5) node[above] {$\ss 2$};
\draw[bosonic]  (2,-0.5)node[below] {$\ss 0$}--(2,0.5) node[right]{$\ss 0$}--(2,1.5) node[above] {$\ss 0$};
\draw[bosonic]  (3,-0.5)node[below] {$\ss 0$}--(3,0.5) node[right]{$\ss 0$}--(3,1.5)node[above] {$\ss 0$};
\end{tikzpicture}
\]
\[
g^{(\alpha,\beta)}_{\lambda}(x_1,x_2)=x_1(x+\alpha)+x_1x_2+x_2(x_2+\alpha)
\]
\end{ex*}

\begin{ex*} For partition $\lambda=(1,1)$, we have
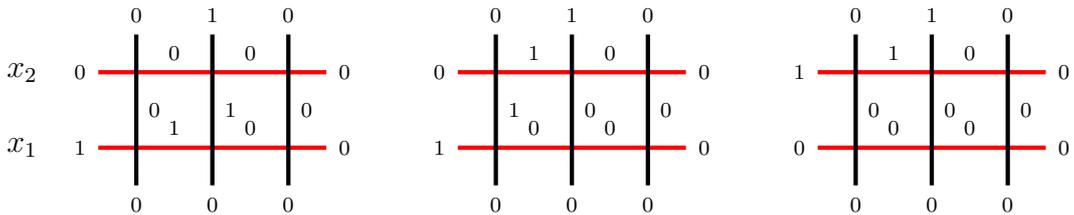
\[
\begin{tikzpicture}[scale=1]
\draw[bosonic,gccol] (0.5,0) node[left,black]{$\ss 1$}--(1.5,0) node[above,black]{$\ss 1$} --(2.5,0) node[above,black]{$\ss 0$}--(3.5,0) node[right,black]{$\ss 0$};
\draw[bosonic,gccol] (0.5,1) node[left,black]{$\ss 0$}--(1.5,1) node[above,black]{$\ss 0$}--(2.5,1)node[above,black]{$\ss 0$}--(3.5,1)node[right,black]{$\ss 0$};
\draw[bosonic]  (1,-0.5) node[below] {$\ss 0$}--(1,0.5) node [right] {$\ss 0$}--(1,1.5) node[above] {$\ss 0$};
\draw[bosonic]  (2,-0.5)node[below] {$\ss 0$}--(2,0.5) node[right]{$\ss 1$}--(2,1.5) node[above] {$\ss 1$};
\draw[bosonic]  (3,-0.5)node[below] {$\ss 0$}--(3,0.5) node[right]{$\ss 0$}--(3,1.5)node[above] {$\ss 0$};
\node at (-0.5,0) {$x_1$};
\node at (-0.5,1) {$x_2$};
\end{tikzpicture}
\qquad
\begin{tikzpicture}[scale=1]
\draw[bosonic,gccol] (0.5,0) node[left,black]{$\ss 1$}--(1.5,0) node[above,black]{$\ss 0$} --(2.5,0) node[above,black]{$\ss 0$}--(3.5,0) node[right,black]{$\ss 0$};
\draw[bosonic,gccol] (0.5,1) node[left,black]{$\ss 0$}--(1.5,1) node[above,black]{$\ss 1$}--(2.5,1)node[above,black]{$\ss 0$}--(3.5,1)node[right,black]{$\ss 0$};
\draw[bosonic]  (1,-0.5) node[below] {$\ss 0$}--(1,0.5) node [right] {$\ss 1$}--(1,1.5) node[above] {$\ss 0$};
\draw[bosonic]  (2,-0.5)node[below] {$\ss 0$}--(2,0.5) node[right]{$\ss 0$}--(2,1.5) node[above] {$\ss 1$};
\draw[bosonic]  (3,-0.5)node[below] {$\ss 0$}--(3,0.5) node[right]{$\ss 0$}--(3,1.5)node[above] {$\ss 0$};
\end{tikzpicture}
\qquad
\begin{tikzpicture}[scale=1]
\draw[bosonic,gccol] (0.5,0) node[left,black]{$\ss 0$}--(1.5,0) node[above,black]{$\ss 0$} --(2.5,0) node[above,black]{$\ss 0$}--(3.5,0) node[right,black]{$\ss 0$};
\draw[bosonic,gccol] (0.5,1) node[left,black]{$\ss 1$}--(1.5,1) node[above,black]{$\ss 1$}--(2.5,1)node[above,black]{$\ss 0$}--(3.5,1)node[right,black]{$\ss 0$};
\draw[bosonic]  (1,-0.5) node[below] {$\ss 0$}--(1,0.5) node [right] {$\ss 0$}--(1,1.5) node[above] {$\ss 0$};
\draw[bosonic]  (2,-0.5)node[below] {$\ss 0$}--(2,0.5) node[right]{$\ss 0$}--(2,1.5) node[above] {$\ss 1$};
\draw[bosonic]  (3,-0.5)node[below] {$\ss 0$}--(3,0.5) node[right]{$\ss 0$}--(3,1.5)node[above] {$\ss 0$};
\end{tikzpicture}
\]
\[
g^{(\alpha,\beta)}_{\lambda}(x_1,x_2)=\beta x_1+x_1 x_2+\beta x_2
\]
\end{ex*}

\subsection{Vertex model for \texorpdfstring{$j$}{j} polynomials.}
\label{vmodelforj}
\subsubsection{Definition of \texorpdfstring{$\jl$}{j} matrix.}
In this subsection, the auxiliary line is fermionic. 
Let
\begin{align}
\label{jLmat}
{{\jl}}_{i,j}(x)
=
\begin{pmatrix}
1 & \phi_i \\
x \phid_i & x+\delta_{0,m}
\end{pmatrix}_{i,j}
\end{align}
be the $\jl$-matrix acting on $F_i \otimes V_j$. Below we represent the entries of $\jl$ graphically:
\begin{equation}
\label{jtiles}
\begin{tabular}{ccccc}
\qquad
\begin{tikzpicture}[scale=0.4,baseline=-2pt]
\draw[fermionic,gcol,arrow=0.25] (-1,0) node[left,black] {$\ss {0}$} -- (1,0) node[right,black] {$\ss {0}$};
\draw[bosonic,arrow=0.25] (0,-1) node[below] {$\ss {m}$} -- (0,1) node[above] {$\ss {m}$};
\node[text width=1cm] at (0.4,-3.5) {$1$};
\end{tikzpicture}
\qquad
\begin{tikzpicture}[scale=0.4,baseline=-2pt]
\draw[fermionic,gcol,arrow=0.25] (-1,0) node[left,black] {$\ss {0}$} -- (1,0) node[right,black] {$\ss {1}$};
\draw[bosonic,arrow=0.25] (0,-1) node[below] {$\ss {m}$} -- (0,1) node[above] {$\ss{m-1}$};
\node[text width=1cm] at (0.4,-3.5){$1$};
\end{tikzpicture}
\qquad
\begin{tikzpicture}[scale=0.4,baseline=-2pt]
\draw[fermionic,gcol,arrow=0.25] (-1,0) node[left,black] {$\ss {1}$} -- (1,0) node[right,black] {$\ss {0}$};
\draw[bosonic,arrow=0.25] (0,-1) node[below] {$\ss {m}$} -- (0,1) node[above] {$\ss {m
+1}$};
\node[text width=1cm] at (0.4,-3.5){$x$};
\end{tikzpicture}
\qquad
\begin{tikzpicture}[scale=0.4,baseline=-2pt]
\draw[fermionic,gcol,arrow=0.25] (-1,0) node[left,black] {$\ss {1}$} -- (1,0) node[right,black] {$\ss {1}$};
\draw[bosonic,arrow=0.25] (0,-1) node[below] {$\ss {m}$} -- (0,1) node[above] {$\ss {m}$};
\node[text width=1cm] at (0.4,-3.5){$x$};
\end{tikzpicture}
\qquad
\begin{tikzpicture}[scale=0.4,baseline=-2pt]
\draw[fermionic,gcol,arrow=0.25] (-1,0) node[left,black] {$\ss {1}$} -- (1,0) node[right,black] {$\ss {1}$};
\draw[bosonic,arrow=0.25] (0,-1) node[below] {$\ss {0}$} -- (0,1) node[above] {$\ss {0}$};
\node[text width=1cm] at (1,-3.5){ ${x+1}$};
\end{tikzpicture}\\
\end{tabular}
\end{equation}
Similarly, we have the dual $\jl^*$ matrices,

 \begin{align}
\label{jdLmat}
\jl^{*}_{i,j}(x)
=
\begin{pmatrix}
x+\delta_{0,m} & x \phi \\
 \phid_i & 1
\end{pmatrix}_{i,j}
\end{align}

\begin{equation}
\label{dGtiles}
\begin{tabular}{ccccc}

\begin{tikzpicture}[scale=0.4,baseline=-2pt]
\draw[fermionic,gcol,arrow=0.25]  (-1,0) node[left,black] {$\ss {1}$}--(1,0) node[right,black] {$\ss {1}$};
\draw[bosonic,arrow=0.25]  (0,-1) node[below] {$\ss {m}$}--(0,1) node[above] {$\ss {m}$};
\node[text width=1cm] at (1,-3.5){1};
\end{tikzpicture}
\qquad
\begin{tikzpicture}[scale=0.4,baseline=-2pt]
\draw[fermionic,gcol,arrow=0.25] (-1,0) node[left,black] {$\ss {1}$}-- (1,0) node[right,black] {$\ss {0}$};
\draw[bosonic,arrow=0.25]  (0,-1) node[below] {$\ss {m}$}--(0,1) node[above] {$\ss {m+1}$} ;
\node[text width=1cm] at (1,-3.5){$ 1$};
\end{tikzpicture}
\qquad
\begin{tikzpicture}[scale=0.4,baseline=-2pt]
\draw[fermionic,gcol,arrow=0.25]  (-1,0) node[left,black] {$\ss {0}$}--(1,0) node[right,black] {$\ss {1}$} ;
\draw[bosonic,arrow=0.25]  (0,-1) node[below] {$\ss {m}$}--(0,1) node[above] {$\ss {m-1}$};
\node[text width=1cm] at (1,-3.5) {$x$};
\end{tikzpicture}
\qquad
\begin{tikzpicture}[scale=0.4,baseline=-2pt]
\draw[fermionic,gcol,arrow=0.25] (-1,0) node[left,black] {$\ss {0}$}--(1,0) node[right,black] {$\ss {0}$};
\draw[bosonic,arrow=0.25]  (0,-1) node[below] {$\ss {m}$}--(0,1) node[above] {$\ss {m}$};
\node[text width=1cm] at (1,-3.5){$x$};
\end{tikzpicture}
\qquad
\begin{tikzpicture}[scale=0.4,baseline=-2pt]
\draw[fermionic,gcol,arrow=0.25]  (-1,0) node[left,black] {$\ss {0}$}--(1,0) node[right,black] {$\ss {0}$};
\draw[bosonic,arrow=0.25]  (0,-1) node[below] {$\ss {0}$}--(0,1) node[above] {$\ss {0}$};
\node[text width=1cm] at (1,-3.5){$x+1$};
\end{tikzpicture}\\
\end{tabular}
\end{equation}

The $\rm R$ matrix that makes this model integrable is same as \cref{GRmatrix} with $\beta=0$. When $\beta=0$ we denote this matrix as $R(y/x)$.

\begin{align}
\label{rmatj}
R_{ij}(y/x)
=
\begin{pmatrix}
1 & 0 & 0 & 0 \\
0 & 0 &\dfrac{y}{x} & 0 \\
0 & 1 & 1-\dfrac{y}{x} & 0 \\
0 & 0 & 0 & 1 
\end{pmatrix}_{ij}
\in 
{\rm End}(F_i \otimes F_j).
\end{align}
The $R$ matrix together with $\jl_i$ and $ \jl_j$ matrices satisfy the $\mathrm{RLL}$ relation with \text{in End} $(F_i \otimes F_j \otimes V_n)$:
\begin{equation}
\label{rllj}
R^{}_{ij}(y/x)\jl_i(x)\jl_j(y)=\jl_j(y)\jl_i(x)R^{}_{ij}(y/x)
\quad
\left(
\begin{tikzpicture}[scale=0.6,baseline=7pt]
\draw[bosonic](1,-1)--(1,2);
\draw[fermionic,gcol,rounded corners,arrow=0.15] (-1,1) node[left,black]{$\ss x$}--(0,0)--(2,0);
\draw[fermionic,gcol,rounded corners,arrow=0.15] (-1,0)  node[left,black]{$\ss y$}--(0,1)--(2,1);
\end{tikzpicture}
=
\begin{tikzpicture}[scale=0.6,baseline=7pt]
\draw[bosonic] (1,2)--(1,-1);
\draw[fermionic,gcol,arrow=0.15,rounded corners] (0,0) node[left,black]{$\ss y$}--(2,0)--(3,1);
\draw[fermionic,gcol,arrow=0.15,rounded corners] (0,1)node[left,black]{$\ss x$}--(2,1)--(3,0);
\end{tikzpicture}
\right)
\end{equation}

\subsubsection{ Row-row transfer matrices.}
 We now define the transfer matrix $\jt$ which acts linearly on $V^{r}$ as follows,
\begin{align}
\label{transferj}
\jt(x):\ket{i_1}\otimes \ket{i_2}\otimes {\cdots}
\mapsto  \sum_{k_1,k_2\ldots \ge 0}
w(\{i_1,i_2,\dots\};\{k_1,k_2,\dots\})
\ket{k_1}\otimes \ket{k_2}\otimes {\cdots}
\end{align}

where $w(\{i_1,i_2,\dots\};\{k_1,k_2,\dots\})$ is the weight of the single of row of vertices.
\[
w(\{i_1,i_2,\dots\};\{k_1,k_2,\dots\})=
\begin{tikzpicture}[scale=0.5,baseline=-1pt]
\draw[fermionic,gcol,arrow=0.05] (-1,0) node[left,black]{$*$} -- (7,0) node[right,black]{$\ss 0$};
\node[left] at (-0.8,0) {\tiny };\node[right] at (6.8,0) {\tiny };
\foreach\x in {0,1,...,6}{
\draw[arrow=0.25,bosonic] (\x,-1) -- (\x,1);
}
\node[below] at (0,-0.8) {\tiny $i_1$};\node[above] at (0,0.8) {\tiny $k_1$};
\node[below] at (1,-0.8) {\tiny $i_2$};\node[above] at (1,0.8) {\tiny $k_2$};
\node[below] at (2,-0.8) {\tiny $i_3$};\node[above] at (2,0.8) {\tiny $k_3$};
\node[below] at (6,-0.8) {$\cdots$};\node[above] at (6,0.8) {$\cdots$};
\end{tikzpicture}
\]

\begin{rmk*}
Observe that the transfer matrix $t$ from row model of $g^{(1,0)}_{\lambda}$, and $\mathfrak{t}$ are the same.
\end{rmk*}

Similarly, we define the dual transfer matrix $\jt^{*}$ which acts linearly on $V^{r}$ as follows,
\begin{align}
\label{dualtransferj}
\jt^{*}(x):\ket{i_1}\otimes \ket{i_2}\otimes {\cdots}
\mapsto  \sum_{k_1,k_2\ldots \ge 0}
w^{*}(\{i_1,i_2,\dots\};\{k_1,k_2,\dots\})
\ket{k_1}\otimes \ket{k_2}\otimes {\cdots}
\end{align}

where $w^{*}(\{i_1,i_2,\dots\};\{k_1,k_2,\dots\})$ is ths weight of the single row of vertices made of the vertices of $\jl^{*}$

\[
w^{*}(\{i_1,i_2,\dots\};\{k_1,k_2,\dots\})=
\begin{tikzpicture}[scale=0.5,,baseline=-1pt]
\draw[fermionic,gcol,arrow=0.05] (-1,0) node[left,black]{$*$}-- (7,0) node[right,black]{$\ss 1$};
\node[left] at (-0.8,0) {\tiny };\node[right] at (6.8,0) {\tiny };
\foreach\x in {0,1,...,6}{
\draw[arrow=0.25,bosonic] (\x,-1) -- (\x,1);
}
\node[below] at (0,-0.8) {\tiny $k_1$};\node[above] at (0,0.8) {\tiny $i_1$};
\node[below] at (1,-0.8) {\tiny $k_2$};\node[above] at (1,0.8) {\tiny $i_2$};
\node[below] at (2,-0.8) {\tiny $k_3$};\node[above] at (2,0.8) {\tiny $i_3$};
\node[below] at (6,-0.8) {$\cdots$};\node[above] at (6,0.8) {$\cdots$};
\end{tikzpicture}
\]

\subsubsection{\texorpdfstring{$j$}{j} polynomials.}
Recall that we denote dual Grothendieck polynomials by $g_{\lambda}$, which is the $\alpha=0$ and $\beta=1$ specialization of $\g$. Then the $\omega(g_{\lambda})$ polynomials are called \emph{weak dual Grothendieck polynomials} and we shall denote them by $j_{\lambda}$: \[
j_{\lambda}=\omega(g_{\lambda})=\omega(g^{(0,1)}_{\lambda})=g^{(1,0)}_{\lambda'}
\]

When $\beta=0$, the branching formula of $\g$ reduces to the following \cite{cauchyid-groth}:
For a partition $\lambda$, we have
\[
\label{singlej}
j_{\lambda}(x_1,\dots,x_n,x_{n+1})= \sum_{\mu } j_{\lambda/\mu}(x_{n+1}) j_{\mu}(x_1,\dots,x_n),
\]
where $j_{\lambda/\mu}(x)$ is defined as follows,
\[
j_{\lambda/\mu}(x)= 
\begin{cases}
x^{c(\lambda/\mu)}(1+x)^{|\lambda/\mu|-c(\lambda/\mu)}&  \lambda/\mu \text{ vert. strip,}\\
0& \text{otherwise.}
\end{cases}
\]

\begin{thm}
\label{thm:j}
The dual weak Grothendieck polynomials $j_{\lambda}(x)$  are given by
\begin{align}
\label{j}
j_{\lambda}(x_1,\dots,x_n)
&=
\bra{0}
\jt(x_1)
\dots
\jt(x_n)
\ket{\lambda^c}
\\
\label{dj}
j_{\lambda}(x_1,\dots,x_n)
&=
\bra{\lambda^c}
\jt^{*}(x_n)
\dots
\jt^{*}(x_1)
\ket{0}
\end{align}
where $\ket{\lambda^c} = \bigotimes_{i=1}^{\infty} \ket{m^c_i(\lambda)}$, and similarly for the dual state $\bra{\lambda^c}$.
\end{thm}

\begin{proof}
Before we prove, let us observe the tiles in an example.

For $\mu=(5,4,4,3)$ and $\lambda=(5,5,5,3,1,1)$,

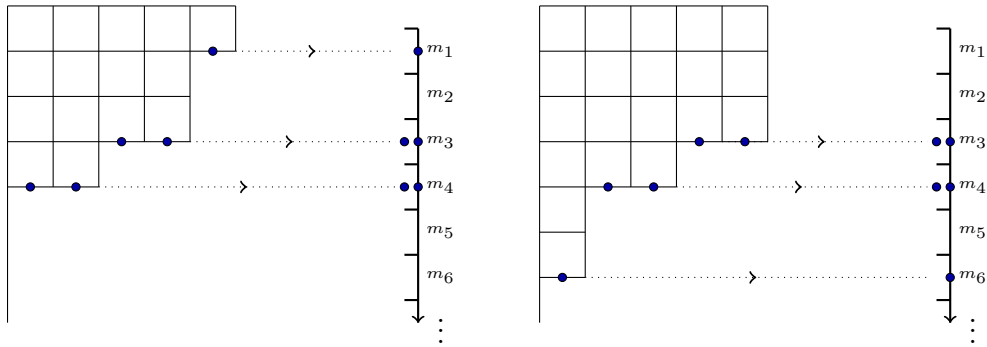
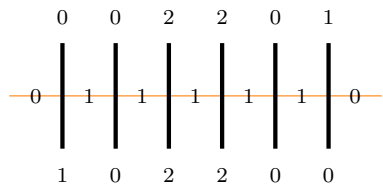
\begin{center}
\label{Exampleforjproof}
\begin{tikzpicture}[scale=0.6]
\draw (0,4) -- (5,4);
\draw (0,3) -- (5,3);
\draw (0,2) -- (4,2);
\draw (0,1) -- (4,1);
\draw (0,0) -- (2,0);
\draw (0,-3) -- (0,4);
\draw (1,0) -- (1,4);
\draw (2,0) -- (2,4);
\draw (3,1) -- (3,4);
\draw (4,1) -- (4,4);
\draw (5,3) -- (5,4);
\bbull{0.5}{0}{0.09};
\bbull{1.5}{0}{0.09};
\bbull{3.5}{1}{0.09};
\bbull{2.5}{1}{0.09};
\bbull{4.5}{3}{0.09};
\draw[dotted,arrow=0.5] (5,3) -- (8.5,3);
\draw[dotted,arrow=0.5] (4,1) -- (8.5,1);
\draw[dotted,arrow=0.5] (2,0) -- (8.5,0);
\draw[thick,arrow=1] (9,3.5) -- (9,-3);
\foreach\x in {1,...,7}{
\draw[thick] (8.7,4.5-\x) -- (9,4.5-\x);
}
\bbull{9}{3}{0.09};
\bbull{9}{1}{0.09};\bbull{8.7}{1}{0.09};
\bbull{8.7}{0}{0.09};\bbull{9}{0}{0.09};
\node at (9.5,3) {$\sss m_1$};
\node at (9.5,2) {$\sss m_2$};
\node at (9.5,1) {$\sss m_3$};
\node at (9.5,0) {$\sss m_4$};
\node at (9.5,-1) {$\sss m_5$};
\node at (9.5,-2) {$\sss m_6$};
\node at (9.5,-3) {$\vdots$};
\end{tikzpicture}
\qquad
\begin{tikzpicture}[scale=0.6]
\draw (0,4) -- (5,4);
\draw (0,3) -- (5,3);
\draw (0,2) -- (5,2);
\draw (0,1) -- (5,1);
\draw (0,0) -- (3,0);
\draw (0,-1) -- (1,-1);
\draw (0,-2) -- (1,-2);
\draw (0,-3) -- (0,4);
\draw (1,-2) -- (1,4);
\draw (2,0) -- (2,4);
\draw (3,0) -- (3,4);
\draw (4,1) -- (4,4);
\draw (5,1) -- (5,4);
\bbull{0.5}{-2}{0.09};
\bbull{1.5}{0}{0.09};
\bbull{2.5}{0}{0.09};
\bbull{3.5}{1}{0.09};
\bbull{4.5}{1}{0.09};
\draw[dotted,arrow=0.5] (4,1) -- (8.5,1);
\draw[dotted,arrow=0.5] (3,0) -- (8.5,0);
\draw[dotted,arrow=0.5] (1,-2) -- (8.5,-2);
\draw[thick,arrow=1] (9,3.5) -- (9,-3);
\foreach\x in {1,...,7}{
\draw[thick] (8.7,4.5-\x) -- (9,4.5-\x);
}
\bbull{8.7}{1}{0.09};\bbull{9}{1}{0.09};
\bbull{8.7}{0}{0.09};\bbull{9}{0}{0.09};
\bbull{9}{-2}{0.09};
\node at (9.5,3) {$\sss m_1$};
\node at (9.5,2) {$\sss m_2$};
\node at (9.5,1) {$\sss m_3$};
\node at (9.5,0) {$\sss m_4$};
\node at (9.5,-1) {$\sss m_5$};
\node at (9.5,-2) {$\sss m_6$};
\node at (9.5,-3) {$\vdots$};
\end{tikzpicture}
\end{center}

\[
\begin{tikzpicture}[scale=0.7]
\draw[fermionic,gcol] (0,0) -- (7,0);
\foreach\x in {1,2,3,4,5,6}{
\draw[bosonic] (\x,-1) -- (\x,1);
}
\node at (1,-1.5) {$\ss 1$};
\node at (2,-1.5) {$\ss 0$};
\node at (3,-1.5) {$\ss 2$};
\node at (4,-1.5) {$\ss 2$};
\node at (5,-1.5) {$\ss 0$};
\node at (6,-1.5) {$\ss 0$};
\node at (1,1.5) {$\ss 0$};
\node at (2,1.5) {$\ss 0$};
\node at (3,1.5) {$\ss 2$};
\node at (4,1.5) {$\ss 2$};
\node at (5,1.5) {$\ss 0$};
\node at (6,1.5) {$\ss 1$};
\node at (0.5,0) {$\ss 0$};
\node at (1.5,0) {$\ss 1$};
\node at (2.5,0) {$\ss 1$};
\node at (3.5,0) {$\ss 1$};
\node at (4.5,0) {$\ss 1$};
\node at (5.5,0) {$\ss 1$};
\node at (6.5,0) {$\ss 0$};

\end{tikzpicture}
\]

From the example above, observe that having $\ss 1$ on the left horizontal edge at site $i$ amounts to adding a box in row $i$. The number of such vertices amounts to the number of boxes added. The vertex \begin{tikzpicture}[scale=0.4,baseline=-2pt]
\draw[fermionic,gcol,arrow=0.25] (-1,0) node[left,black] {$\ss {1}$} -- (1,0) node[right,black] {$\ss {1}$};
\draw[bosonic,arrow=0.25] (0,-1) node[below] {$\ss {0}$} -- (0,1) node[above] {$\ss {0}$};

\end{tikzpicture} at site $i$ can be read as adding a box in two successive rows in the same column. So every such vertex amounts to $r(\lambda/\mu)-c(\lambda/\mu)$. Using similar reasoning in (\ref{thm:rowG}) and with the above analysis, the proof is immediate.
\end{proof}

\section{Generalised polynomials}
\label{generalisedpolynomials}
In this section, we shall generalise the polynomials by introducing additional variables which are attached to the vertical lines of the underlying lattice model. In order to do that, we need the $\rm R$ matrix that underpins the integrability of the lattice model to satisfy the so-called {\em difference property}. Usually the difference property refers to entries of $\rm R$ matrix being invariant under translation of the spectral variables. In this paper, we say that a $\rm R$ matrix satisfies the difference property when the entries are invariant under scaling of the spectral parameters i.e., the non constant entries are polynomials in ratio of the spectral variables.

\subsection{Difference property of the \texorpdfstring{$\mathrm R$}{R} matrices.}
In this subsection, we study the difference property of the various $\mathrm R$ matrices studied in this paper. 

\subsubsection{\texorpdfstring{$\mathrm R$}{R} matrices of Row models.}
Consider the $R$ matrix for canonical Grothendieck polynomials. Observe that when $\beta=0$, the $R$ matrix satisfies the difference property. More generally, it satisfies the difference property when we consider the spectral variables to be $\dfrac{x}{1-\beta x}$ instead of $x$.

\begin{align}
R(y/x)=R(x,y)=
\begin{pmatrix}
1 & 0 & 0 & 0 \\[1ex]
0 & 0 &  \dfrac{y}{x} & 0 \\[1ex]
0 & 1 & {1}- \dfrac{y}{x} & 0 \\[1ex]
0 & 0 & 0 & 1 \\
\end{pmatrix}
\end{align}

In the case of $\g$, the $r$-matrix does not satisfy the difference property. It is also not defined for $\beta=0$. Hence, we studied a different model for the case where $\beta=0$. Recall that the $R$ matrix for vertex model of $j_{\lambda}=g^{(1,0)}_{\lambda'}$, is same as the $R$.

\subsubsection{\texorpdfstring{$\mathrm R$}{R} matrices of Column models.}

In the case of column models, the $\widetilde{R}$ matrix does satisfy the difference property in general, when we consider the spectral variables to be $\frac{x}{1+\alpha x}$ and $\frac{y}{1+\alpha y}$ instead of $x$ and $y$.
For $G^{(\alpha,\beta)}_{\lambda}$, the $\widetilde{R}$ matrix is given below:
\begin{align*}
\widetilde{R}^{a\,d}_{b\,c}\left(y,x\right) = \left({\dfrac{\dfrac{x}{1-\alpha x}}{\dfrac{y}{1-\alpha y}}}\right)^{a}
  \begin{cases}
    0 & \text{when } b<c,\\
   1 & \text{when } b=c ,\\
  \dfrac{1}{1-\alpha x}-\dfrac{x}{(1-\alpha x)y} & \text{when } b>c .
  \end{cases}
\end{align*}

For the polynomials $G^{(0,\beta)}$, the $\widetilde{R}$ matrix satisfies the difference property.

The $\widetilde{r}$-matrix of column vertex model of $\g$,
\begin{align}
\widetilde{r}^{k,l}_{i,j} (x,y)=
\begin{cases}
0 & i<j\\
1& k=l=0\\
\dfrac{x}{y} \left(\dfrac{y+\alpha}{x+\alpha}\right)^{1-k} & k=l>0\\
 \left(1-\dfrac{x}{y}\right)  & k=0\\
\dfrac{x}{y}  \left(\dfrac{y+\alpha}{x+\alpha} -1\right) \left(\dfrac{y+\alpha}{x+\alpha}\right)^{-k} & k>0\\
\end{cases}
\end{align}
satisfies the difference property when $\alpha=0$.

\subsection{Generalised polynomials.}
To summarize, in the case of row models we can only generalise $G^{(\alpha,0)}_{\lambda}$ and $g^{(\alpha,0)}_{\lambda}$. Similarly, in the case of column models, we can generalise $G^{(0,\beta)}_{\lambda}$ and $g^{(0,\beta)}_{\lambda}$.

We generalise the polynomials by assigning a variable to the vertical lines. Let us assign the variable $z_i$ to the $i^{th}$ vertical line from the left. In any model, the weight of a vertex formed with the intersection of  $i^{th}$ horizontal line and $j^{th}$ vertical line is defined as follows,

\begin{align*}
\label{}
\widetilde{w}^{}_{(x_i,z_j)}
\left(
\begin{gathered}
\begin{tikzpicture}[scale=0.4,baseline=-2pt]
\draw[arrow=0.25,bosonic] (-1,0) node[left] {$a$} -- (1,0) node[right] {$c$};
\draw[arrow=0.25,bosonic] (0,-1) node[below] {$b$} -- (0,1) node[above] {$d$};
\end{tikzpicture}
\end{gathered}
\right)
=w_{\left(\frac{x_i}{z_j}\right)}(a,b;c,d)
\end{align*}

Let us name these polynomials.
\begin{itemize}
\item[(i)] We call $G^{\alpha}_{\lambda}=G^{(0,-\alpha)}_{\lambda}$
\emph{generalised Grothendieck polynomials}
\item[(ii)] We call $g^{\alpha}_{\lambda}=g^{(0,\alpha)}_{\lambda}$ 
\emph{generalised dual Grothendieck polynomials }
\item[(iii)] We call $J^{\alpha}_{\lambda}=G^{(-\alpha,0)}_{\lambda'}$  \emph{generalised weak Grothendieck polynomials }
\item[(iv)] We call $j^{\alpha}_{\lambda}=g^{(\alpha,0)}_{\lambda'}$  \emph{generalised weak dual Grothendieck polynomials }
\end{itemize}

When $\alpha=1$, we shall drop the superscript.
\begin{ex*}
For the partition $(1)$, the \emph{generalised Grothendieck polynomial} $G_{\lambda}(x_1,z_1)$ is $\dfrac{x_1}{z_1}$.
For comparison, the \emph{Double Grothendieck polynomial} is $x_1+y_1 (1-x_1)$ \cite{McN-Groth}. We observe that these two generalisations of Grothendieck polynomials are not the same.
\end{ex*}

\begin{ex*} Let us look at a non trivial example.
\[G^{(0,-1)}_{(3,1)}(x_1,x_2)=\br{\frac{x^{2}_1}{z_1 z_2}}\br{1-\frac{x_1}{z_1}}\br{\frac{x_2}{z_1}}^{2}+ \br{\frac{x^{3}_1}{z^{2}_1 z_2}} \br{\frac{x^{}_2}{z^{}_1}}+ \br{1-\frac{x_1}{z_1}}\br{\frac{x_1}{z_2}}\br{\frac{x_2}{z_1}}^{3}
\]
\[
\begin{tikzpicture}[scale=0.5]
\draw[bosonic,Gcol,rounded corners] (0,0) node[left,black]{$\ss 2$}--(1,0)--(1,1) node[right,black]{$\ss 2$}--(1,3) node[above,black]{$\ss 2$};
\draw[bosonic,Gcol,rounded corners] (0,2) node[left,black]{$\ss 1$}--(2,2)  node[above,black]{$\ss 1$}--(3,2)--(3,3)node[above,black]{$\ss 1$};
\node at (-1.5,0) {$ x_2$};
\node at (-1.5,2) {$ x_1$};
\end{tikzpicture}
\qquad
\begin{tikzpicture}[scale=0.5]
\draw[bosonic,Gcol,rounded corners] (0,0) node[left,black]{$\ss 1$}--(1,0)--(1,1) node[right,black]{$\ss 1$}--(1,2)--(2,2)node[above,black]{$\ss 1$}--(3,2)--(3,3)node[above,black]{$\ss 1$};
\draw[bosonic,Gcol,rounded corners] (0,2) node[left,black]{$\ss 2$}--(1,2)  --(1,3)node[above,black]{$\ss 2$};
\end{tikzpicture}
\qquad
\begin{tikzpicture}[scale=0.5]
\draw[bosonic,Gcol,rounded corners] (0,0) node[left,black]{$\ss 3$}--(1,0)--(1,1) node[right,black]{$\ss 3$}--(1,3) node[above,black]{$\ss 2$};
\draw[bosonic,Gcol,rounded corners] (0,2) node[left,black]{$\ss 0$}--(2,2)  node[above,black]{$\ss 1$}--(3,2)--(3,3)node[above,black]{$\ss 1$};
\end{tikzpicture}
\]
\end{ex*}

\begin{rmk*}
Observe that by setting both $\alpha$ and $\beta$ to $0$, we get a generalised version of \emph{Schur polynomials}. \emph{Generalised Schur polynomials} from row model and column model are not the same. Let us denote the generalised Schur from row (column) model with $s^{r}$ $(s^{c})$.
\begin{ex*}
For partition $\lambda=(3,1)$, we have
\[
s^{r}_{(3,1)}(x_1,x_2;z_1,z_2,\dots)= \br{\frac{x^{3}_1}{z_1 z_2 z_3}}\br{\frac{x_2}{z_1}}+\br{\frac{x^{2}_1}{z_1 z_2}}\br{\frac{x^{2}_2}{z_1 z_3}}+ \br{\frac{x_1}{z_1}}\br{\frac{x^{3}_2}{z_1 z_2 z_3}}
\]
\[s^{c}_{(3,1)}(x_1,x_2;z_1,z_2,\dots)=\br{\frac{x^{2}_1}{z_1 z_2}}\br{\frac{x_2}{z_1}}^{2}+ \br{\frac{x^{3}_1}{z^{2}_1 z_2}} \br{\frac{x^{}_2}{z^{}_1}}+\br{\frac{x_1}{z_2}}\br{\frac{x_2}{z_1}}^{3}
\]
\end{ex*}
\end{rmk*}
$s^{r}_{\lambda}$ is a monomial multiple of $s_{\lambda}$, where the monomial is obtained by recording the columns of the Young diagram. Similarly, in the case of $s^{c}_{\lambda}$ the monomial is obtained by recording the rows of the Young diagram.

\section{Duality between Column and Row models}
\label{dualitybetweenmodels}

In this section, we shall study a relation between the transfer matrix of the row and column model of $\G$. Let us recall the necessary notation from various sections. The transfer matrices of the row model for $\G$ are denoted by $T$, and of the column model are denoted by $\widetilde{T}$.

\begin{prop}[Inversion relation]
\label{inversionGrothprop}
The transfer matrices $\widetilde{T}$ and $T$ satisfy the following identity:
\begin{align}
\label{inversionGroth}
{T}\br{-x}\widetilde{{T}}\left(\frac{x}{1+(\alpha-\beta) x} \right) =1
\qquad
\left(
\begin{tikzpicture}[scale=0.6,baseline=-2pt]
\draw[fermionic,Gcol] (0,0) node[left,black]{$*$} --(7,0)node[right,black]{$\ss 0$};
\draw[bosonic,Gcol]  (0,1)node[left,black]{$*$}--(7,1)node[right,black]{$\ss 0$};
\draw[bosonic] (6,-1)--(6,2);
\foreach\x in {1,2,3,4}{
\draw[bosonic] (\x,-1)--(\x,2);
\node at (\x,2.2) {$\ss u_{\x}$};
\node at (\x,-1.3) {$\ss v_{\x}$};
};
\node at (-2.5,0) {$\ss {T\br{\frac{-x}{z}}}$};
\node at (-2.5,1) {$\ss \widetilde{T}\left(\frac{\frac{x}{z}}{1+(\alpha-\beta) \frac{x}{z}}\right)$};
\node at (5,2.2) {$\ss \dots \dots$};
\node at (5,-1.2) {$\ss \dots  \dots$};
\end{tikzpicture}
\right)
\end{align}
\end{prop}
\begin{proof}
We shall prove the proposition for transfer matrices of size $1$ and then apply induction to the size of the transfer matrices. Assign weights of $T\br{-x}$  for the vertex at the bottom and the weights of $\widetilde{T}\left( \frac{x}{1+(\alpha-\beta)x}\right)$ for the other vertex. Write $z$ for the vertical spectral parameter.

Observe that when $b=d$, there is a unique configuration with total weight $1$. Now assume $b\neq d$

\begin{align*}
\begin{tikzpicture}[scale=0.6,baseline=7pt]
\draw[bosonic,Gcol] (0,1) node[left,black]{$\ss d-b$}--(2,1) node[right,black] {$0$};
\draw[fermionic,Gcol] (0,0) node[left,black]{$ 0$}--(2,0) node[right,black] {$0$};
\draw[bosonic] (1,-1) node[below,black] {$b$}--(1,0.5) node[right,black]{$\ss b$}-- (1,2) node[above,black]{$d$};
\end{tikzpicture}+
\begin{tikzpicture}[scale=0.6,baseline=7pt]
\draw[bosonic,Gcol] (0,1) node[left,black]{$\ss d-b-1$}--(2,1) node[right,black] {$0$};
\draw[fermionic,Gcol] (0,0) node[left,black]{$ 1$}--(2,0) node[right,black] {$0$};
\draw[bosonic] (1,-1) node[below,black] {$b$}--(1,0.5) node[right,black]{$\ss b+1$}-- (1,2) node[above,black]{$d$};
\end{tikzpicture}
\end{align*}
When $b>0$,
\begin{align*}
=&\text{ }\br{\frac{1-\beta \br{\frac{x}{z}}}{1+\alpha \br{\frac{x}{z}}}}\br{\frac{{\frac{x}{z}}}{1-\beta \br{\frac{x}{z}}}}^{d-b}\br{\frac{1+\alpha \br{\frac{x}{z}}}{1-\beta \br{\frac{x}{z}}}}+\\
&\text{ }\br{\frac{-{\frac{x}{z}}}{1+\alpha \br{\frac{x}{z}}}}\br{\frac{{\frac{x}{z}}}{1-\beta \br{\frac{x}{z}}}}^{d-b-1}\br{\frac{1+\alpha \br{\frac{x}{z}}}{1-\beta \br{\frac{x}{z}}}}\\
=&\text{ }0
\end{align*}

When $b=0$,
\begin{align*}
=&\text{ }\br{\frac{{\frac{x}{z}}}{1-\beta \br{\frac{x}{z}}}}^{d}+ \br{\frac{-{\frac{x}{z}}}{1+\alpha \br{\frac{x}{z}}}}\br{\frac{{\frac{x}{z}}}{1-\beta \br{\frac{x}{z}}}}^{d-1}\br{\frac{1+\alpha \br{\frac{x}{z}}}{1-\beta \br{\frac{x}{z}}}}\\
=&\text{ }0
\end{align*}

In order to apply the induction argument, we need to show that the transfer matrix of size $n+1$ can be written as a multiple of the transfer matrix of size $n$. Consider a transfer matrix of size $n+1$ where top and bottom labels are fixed.

\[
\begin{tikzpicture}[scale=0.6,baseline=-2pt]
\draw[fermionic,Gcol] (0,0) node[left,black]{$*$} --(5,0)node[right,black]{$\ss 0$};
\draw[bosonic,Gcol]  (0,1)node[left,black]{$*$}--(5,1)node[right,black]{$\ss 0$};
\foreach\x in {1,2,3,4}{
\draw[bosonic] (\x,-1)--(\x,2);
\node at (\x,2.2) {$\ss u_{\x}$};
\node at (\x,-1.3) {$\ss v_{\x}$};
};
\end{tikzpicture}
\]
\[
\begin{tikzpicture}[scale=0.6,baseline=-2pt]
\draw[bosonic](-1,-1)node[below]{$\ss v_1$}--(-1,2)node[above]{$\ss u_1$};
\draw[fermionic,Gcol] (-2,0)node[left,black]{$*$}--(0,0) ;
\draw[bosonic,Gcol] (-2,1)node[left,black]{$*$}--(0,1);
\draw[fermionic,Gcol] (1,0)--(5,0)node[right,black]{$\ss 0$};
\draw[bosonic,Gcol]  (1,1)--(5,1)node[right,black]{$\ss 0$};
\foreach\x in {2,3,4}{
\draw[bosonic] (\x,-1) node[below]{$\ss u_{\x}$}--(\x,2)node[above]{$\ss v_{\x}$};
};
\node at (7,0.5) {$=$};
\end{tikzpicture}
\begin{tikzpicture}[scale=0.6,baseline=-2pt]
\draw[bosonic](-1,-1)node[below]{$\ss v_1$}--(-1,2)node[above]{$\ss u_1$};
\draw[fermionic,Gcol] (-2,0)--(0,0) ;
\draw[bosonic,Gcol] (-2,1)--(0,1);
\draw[fermionic,Gcol] (1,0)node[left,black]{$*$}--(5,0)node[right,black]{$\ss 0$};
\draw[bosonic,Gcol]  (1,1)node[left,black]{$*$}--(5,1)node[right,black]{$\ss 0$};
\foreach\x in {2,3,4}{
\draw[bosonic] (\x,-1) node[below]{$\ss u_{\x}$}--(\x,2)node[above]{$\ss v_{\x}$};
};
\end{tikzpicture}
\]

Observe that when the left most boundary is fixed, then the contribution from the first site is fixed. Therefore, we can move the free boundary condition across the physical line at site $1$. We then apply induction.
\end{proof}

Define $\widetilde{T}^{*}(x)\in \mathrm {End}(V^{c})$ as the adjoint of $\widetilde{T}$(x):
\[
\bra{\lambda^{c}}\widetilde{T}^{*}(x)\ket{\mu^{c}}=\bra{\mu^{c}}\widetilde{T}(x)\ket{\lambda^{c}}.
\]

Then the inversion relation between $T^{*}$ and $\widetilde{T}^{*}$ follows from the definition of $\widetilde{T}^*$ and the inversion relation of $T$ and $\widetilde{T}$.
\begin{align}
\label{inversionGrothT*}
T^{*}\br{-x}\widetilde{T}^{*}\br{\frac{x}{1+(\alpha-\beta)x}}=1
\end{align}

 In the case of $\g$, such an inversion relation does not exist for general $\alpha$ and $\beta$. But there is an inversion relation in the case where $\alpha=1$ and $\beta=0$. Since we are concerned with $g^{(0,1)}_{\lambda}$, it is convenient to specialize the Boltzmann weights from the column model of $\g$.
\begin{align*}
\label{}
{w}^{}_{(x,z)}
\left(
\begin{gathered}
\begin{tikzpicture}[scale=0.4,baseline=-2pt]
\draw[arrow=0.25,bosonic,gccol] (-1,0) node[left,black] {$a$} -- (1,0) node[right,black] {$c$};
\draw[arrow=0.25,bosonic] (0,-1) node[below,black] {$b$} -- (0,1) node[above,black] {$d$};
\end{tikzpicture}
\end{gathered}
\right)
=w_{\left(\frac{x}{z}\right)}(a,b;c,d)=\br{\frac{x}{z}}^{\min(a,d)}
\end{align*}

\begin{prop}
\label{inversionDualGroth}
The transfer matrices of $g^{(1,0)}_{\lambda}$ from the row model and transfer matrices of $g^{(0,1)}_{\lambda}$ from the column model satisfy the following relation:
\[
  \jt\br{-x} \tilde{t}\br{x}=1
\qquad
\left(
\begin{tikzpicture}[scale=0.6,baseline=-2pt]
\draw[fermionic,gcol] (0,0) node[left,black]{$*$} --(7,0)node[right,black]{$\ss 0$};
\draw[bosonic,gccol]  (0,1) node[left,black]{$*$}--(7,1)node[right,black]{$\ss 0$};
\draw[bosonic] (6,-1)--(6,2);
\foreach\x in {1,2,3,4}{
\draw[bosonic] (\x,-1)--(\x,2);
\node at (\x,2.2) {$\ss u_{\x}$};
\node at (\x,-1.3) {$\ss v_{\x}$};
};
\node at (-2.5,1) {$\ss {\tilde{t}\br{x}} $};
\node at (-2.5,0) {$\ss {\jt\br{-x}} $};
\node at (5,2.2) {$\ss \dots \dots$};
\node at (5,-1.2) {$\ss \dots  \dots$};
\end{tikzpicture}
\right)
\]
\end{prop}
\begin{proof}
The proof is similar to \cref{inversionGrothprop}. We shall prove the statement for transfer matrices of size $1$. When $b=d$, there is a unique configuration with weight $1$ and when $b\neq d$, we have two configurations which add upto $0$.
\[
\begin{tikzpicture}[scale=0.6,baseline=7pt]
\draw[fermionic,gcol] (0,0) node[left,black]{$\ss 0$}--(2,0) node[right,black] {$\ss 0$};
\draw[bosonic,gccol] (0,1) node[left,black]{$\ss d-b$}--(2,1) node[right,black] {$\ss 0$};
\draw[bosonic] (1,-1) node[below,black] {$\ss b$}--(1,0.5) node[right,black]{$\ss b$}-- (1,2) node[above,black]{$\ss d$};
\end{tikzpicture}+
\begin{tikzpicture}[scale=0.6,baseline=7pt]
\draw[fermionic,gcol] (0,0) node[left,black]{$\ss 1$}--(2,0) node[right,black] {$\ss 0$};
\draw[bosonic,gccol] (0,1) node[left,black]{$\ss d-1-b$}--(2,1) node[right,black] {$\ss 0$};
\draw[bosonic] (1,-1) node[below,black] {$\ss b$}--(1,0.5) node[right,black]{$\ss b+1$}-- (1,2) node[above,black]{$\ss d$};
\end{tikzpicture}
\qquad
\begin{aligned}
\text{ }\br{\dfrac{x}{z}}^{d-b}+\br{\dfrac{-x}{z}}\br{\dfrac{x}{z}}^{d-1-b}=0
\end{aligned}
\]
Assume $b>0$, then for any fixed $v,u,b,d$ the Boltzmann weights are fixed irrespective of the right boundary. 
\begin{align*}
\begin{tikzpicture}[scale=0.6,baseline=7pt]
\draw[fermionic,gcol] (0,0) node[left,black]{$\ss v$}--(2,0) node[right,black] {$*$};
\draw[bosonic,gccol] (0,1) node[left,black]{$\ss u$}--(2,1) node[right,black] {$*$};
\draw[bosonic] (1,-1) node[below,black] {$\ss b$}--(1,0.5) node[right,black]{}-- (1,2) node[above,black]{$\ss d$};
\end{tikzpicture}
\end{align*}

We can then simply slide the free boundary condition. Special care needs to be taken when $b=0$. When $b=0$, we have the following configurations:
\begin{align*}
\begin{tikzpicture}[scale=0.6,baseline=7pt]
\draw[fermionic,gcol] (0,0) node[left,black]{$\ss 0$}--(2,0) node[right,black] {$\ss 0$};
\draw[bosonic,gccol] (0,1) node[left,black]{$\ss  d+c+1$}--(2,1) node[right,black] {$\ss c+1$};
\draw[bosonic] (1,-1) node[below,black] {$\ss 0$}--(1,0.5) node[right,black]{$\ss 0$}-- (1,2) node[above,black]{$d$};
\end{tikzpicture}= \br{\frac{x}{z}}^{d}
\quad
\begin{tikzpicture}[scale=0.6,baseline=7pt]
\draw[fermionic,gcol] (0,0) node[left,black]{$\ss 1$}--(2,0) node[right,black] {$\ss 0$};
\draw[bosonic,gccol] (0,1) node[left,black]{$\ss  d+c$}--(2,1) node[right,black] {$\ss c+1$};
\draw[bosonic] (1,-1) node[below,black] {$\ss 0$}--(1,0.5) node[right,black]{$\ss 1$}-- (1,2) node[above,black]{$\ss d$};
\end{tikzpicture}=\br{-\frac{x}{z}}^{d+1}
\quad
\begin{tikzpicture}[scale=0.6,baseline=7pt]
\draw[fermionic,gcol] (0,0) node[left,black]{$\ss 1$}--(2,0) node[right,black] {$\ss 1$};
\draw[bosonic,gccol] (0,1) node[left,black]{$\ss  d+c$}--(2,1) node[right,black] {$\ss c$};
\draw[bosonic] (1,-1) node[below,black] {$\ss 0$}--(1,0.5) node[right,black]{$\ss 0$}-- (1,2) node[above,black]{$\ss d$};
\end{tikzpicture}={ \br{1-\frac{x}{z}}\br{\frac{x}{z}}^{d}}
\end{align*}

Observe that in the case of the first configuration, there is a unique configuration suggesting that the right boundary is not free. But we can get away with it by adding the weight of first configuration with the weight of the second configuration. Then we obtain $\br{1-\frac{x}{z}}\br{\frac{x}{z}}^{d}$ times the transfer matrix of size $n$.
\end{proof}
\section{Cauchy identities}
\label{cauchyidentities}

In this section, we shall prove Cauchy identities involving $G^{(-\alpha,-\beta)}_{\lambda}$ and $\g$. We shall prove these identities by using the commutation relations between various combinations of the transfer matrices.

Before we prove the \emph{Cauchy identity}, we shall derive the commutation relation between the appropriate transfer matrices:
\begin{prop}\label{prop:commcauchy}
Let $T^*$ be the transfer matrix for the row model of 
$G^{(-\alpha,-\beta)}_{\lambda}$,
and $t$ the transfer matrix for the row model of $\g$.
Then
\begin{equation}
\label{cauchytransferrel}
{t}(y) {T^{*}}(x)= \frac{1}{1-xy} {T^{*}}(x) {t}(y)
\end{equation}
\end{prop}
\begin{proof}
The proof is similar to the way we proved that the transfer matrices commute. 
The $\mathfrak{R}(x,y)$ matrix
\begin{align}
\mathfrak{R}(x,y)\in 
\text{End} (F\otimes W),\quad
\mathfrak{R}^{k,l}_{i,j}=
\begin{gathered}
\begin{tikzpicture}[baseline=-3pt,scale=0.9]
\draw[fermionic,arrow=0.35,Gcol] (-2,0.5) node[left,black] {$k$} -- (-1,-0.5) node[below] {};
\node[text width=1cm] at (-0.4,-0.55){$j$};
\draw[bosonic,arrow=0.35,gcol] (-2,-0.5) node[left,black] {$i$} -- (-1,0.5) node[above] {};
\node[text width=1cm] at (-0.4,0.5){$l$};
\end{tikzpicture}
\end{gathered}=
\begin{cases}
1-x y & j=k=1,i=l=0\\
xy & k=l=0,i=j=1\\
1-x \beta & k=1\\
x\beta &k=0\\
1& i=k=l=j=0\\
\end{cases}
\end{align}
where $k,j\in \{ 0,1\}$ and $i,l \in \ZZ_{\geq 0}$, together with the $L^{*}(x)$ matrix of $G^{(-\alpha,-\beta)}_{\lambda}(x)$ and the $l(y)$ matrix of $\g(y)$ satisfy the \rm{RLL} relation:
(see appendix~\ref{rllforcauchy})
\begin{equation}
\label{rLl}
\mathfrak{R}^{}_{ij}(x,y)L^{*}(x)l(y)=l(y)L^{*}(x)\mathfrak{R}(x,y) \in \text{End}(F\otimes W \otimes V)
\quad
\left(
\begin{tikzpicture}[scale=0.6,baseline=7pt]
\draw[bosonic](1,-1)--(1,2);
\draw[fermionic,Gcol,rounded corners,arrow=0.15] (-1,1) node[left,black]{$\ss x$}--(0,0)--(2,0);
\draw[bosonic,gcol,rounded corners,arrow=0.15] (-1,0)  node[left,black]{$\ss y$}--(0,1)--(2,1);
\end{tikzpicture}
=
\begin{tikzpicture}[scale=0.6,baseline=7pt]
\draw[bosonic] (1,2)--(1,-1);
\draw[bosonic,gcol,arrow=0.15,rounded corners] (0,0) node[left,black]{$\ss y$}--(2,0)--(3,1);
\draw[fermionic,Gcol,arrow=0.15,rounded corners] (0,1)node[left,black]{$\ss x$}--(2,1)--(3,0);
\end{tikzpicture}
\right)
\end{equation}

Observe that $\mathfrak{R}$ satisfies the eigenvector property.
\[
\begin{tikzpicture}[scale=0.8,baseline=5pt]
\draw[fermionic,Gcol] (0,1) node[left,black]{$*$}--(1,0);
\draw[bosonic,gcol] (0,0) node[left,black]{$*$}--(1,1);
\node at (-0.8,0) {$\ss y$};
\node at (-0.8,1) {$\ss x$};
\end{tikzpicture}
=
\begin{tikzpicture}[scale=0.8,baseline=5pt]
\draw[fermionic,Gcol] (0,1) node[left,black]{$*$}--(1,1);
\draw[bosonic,gcol] (0,0) node[left,black]{$*$}--(1,0);
\node at (-0.8,0) {$\ss y$};
\node at (-0.8,1) {$\ss x$};
\end{tikzpicture}
\]
Then after multiplying the $\mathfrak{R}$-matrix to $T^*(x)t(y)$ and repeated application of the $\rm {RLL}$ relation we get the following equation:
\[
\begin{tikzpicture}[scale=0.6,baseline=3pt]
\draw[fermionic,Gcol,rounded corners] (-1,1)node[black,left]{$*$}--(0,0) node[below,black] {}--(7,0) node[black,right]{$\ss 1$};
\draw[bosonic,gcol,rounded corners](-1,0)node[black,left]{$*$}-- (0,1)node[above,black] {}--(7,1)  node[black,right]{$\ss 0$};
\foreach \x in {1,2,3,4,5}{
\draw[bosonic] (\x,-1)--(\x,2);
\node at (\x,2.2) {$\ss i_{\x}$};
\node at (\x,-1.2) {$\ss k_{\x}$};
}
\draw[bosonic] (6,-1)--(6,2);
\node at (6,2.2) {$\ss \dots$};
\node at (6,-1.2) {$\ss \dots$};
\node at (-2,1) {$\ss x$};
\node at (-2,0) {$\ss y$};
\end{tikzpicture}
=
\begin{tikzpicture}[scale=0.6,baseline=3pt]
\draw[bosonic,gcol,rounded corners] (0,0) node[left,black]{$*$}--(7,0) node[below,black]{$\ss 0$}--(8,1) node[right,black]{$\ss 0$};
\draw[fermionic,Gcol,rounded corners] (0,1)node[left,black]{$*$}--(7,1)node[above,black]{$\ss 1$}--(8,0) node[right,black]{$\ss 1$};
\foreach \x in {1,2,3,4,5}{
\draw[bosonic] (\x,-1)--(\x,2);
\node at (\x,2.2)  {$\ss i_{\x}$};
\node at (\x,-1.2) {$\ss k_{\x}$};
}
\draw[bosonic] (6,-1)--(6,2);
\node at (6,2.2)  {$\ss \dots$};
\node at (6,-1.2) {$\ss \dots$};
\node at (-1,0) {$\ss y$};
\node at (-1,1) {$\ss x$};
\end{tikzpicture}
\]
Finally, the entry corresponding to the cross at the end of the right hand  side is $(1-xy)$.
\[
\begin{tikzpicture}[scale=0.8,baseline=5pt]
\draw[fermionic,Gcol] (0,1) node[left,black]{$\ss 1$} --(1,0)node[right,black]{$\ss 1$};
\draw[bosonic,gcol] (0,0) node[left,black]{$\ss 0$} --(1,1)node[right,black]{$\ss 0$};
\node at (-0.8,0) {$\ss y$};
\node at (-0.8,1) {$\ss x$};
\end{tikzpicture}
=(1-xy)
\begin{tikzpicture}[scale=0.8,baseline=5pt]
\draw[fermionic,Gcol] (0,1) node[left,black]{$\ss 1$}  --(1,1) node[right,black]{$\ss 1$};
\draw[bosonic,gcol] (0,0) node[left,black]{$\ss 0$} --(1,0)node[right,black]{$\ss 0$};
\node at (-0.8,0) {$\ss y$};
\node at (-0.8,1) {$\ss x$};
\end{tikzpicture}
\]

This implies the desired commutation relation:
\[
T^{*}(x)t(y)=(1-xy)t(y)T^*(x)
\]
\end{proof}

\begin{thm}
\label{cauchyidentity}
Canonical Grothendieck polynomials and their duals  satisfy the following Cauchy identity:
\begin{align}
\label{cauchyidentityeq}
    \sum_{\lambda}G^{(-\alpha,-\beta)}_{\lambda}(x_1,x_2,\dots,x_m) g^{(\alpha,\beta)}_{\lambda}(y_1,y_2,\dots,y_n)=\prod_{1\le i\le m,1\le j\le n} \dfrac{1}{1-x_i y_j}.
 \end{align}
\end{thm}

\begin{proof}
Let 
\[
\mathcal{G}(x_1,x_2,\dots, x_m,y_1,y_2,\dots,y_n)= \bra{0}{t}(y_1){t}(y_2)\cdots \cdots {T^{*}}(x_2) {T^{*}}(x_1)\ket{0}.
\] Then by \cref{thm:rowG} and \cref{thm:rowg} we get that,
\[
\mathcal{G}(x_1,x_2,\dots,x_m,y_1,y_2,\dots,y_n)=  \sum_{\lambda}G^{(-\alpha,-\beta)}_{\lambda}(x_1,x_2,\dots,x_m) g^{(\alpha,\beta)}_{\lambda}(y_1,y_2,\dots,y_n).
\]

By repeatedly applying the commutation relation of proposition~\ref{prop:commcauchy}, we obtain
\begin{align*}
    \mathcal{G}(x_1,x_2,\dots,x_m,y_1,y_2,\dots,y_n)&= \prod_{1\le i\le n,1\le j\le m} \dfrac{1}{1-x_i y_j} \bra{0}{T^{*}}(x_1) {T^{*}}(x_2)\cdots \cdots {t}(y_2){t}(y_1)\ket{0}\\
&=  \prod_{1\le i\le m,1\le j\le n}\dfrac{1}{1-x_i y_j}.
\end{align*}
\end{proof}

We can derive a skew version of the identity if we choose a different vector and covector. Let
\[
\mathcal{G}(x_1,x_2,\dots, x_m,y_1,y_2,\dots,y_n)= \bra{\mu}{t}(y_1){t}(y_2)\cdots \cdots {T^{*}}(x_2) {T^{*}}(x_1)\ket{\lambda},
\]
then using the same reasoning as in the above theorem we get the following identity:
\begin{multline*}
\sum_{\nu}G_{\nu\sslash \lambda}(x_1,x_2,\dots, x_m)g_{\nu/\lambda}(y_1,y_2,\dots,y_n)=\\
\prod_{1\le i\le n,1\le j\le m} \dfrac{1}{1-x_i y_j}
\sum_{\nu}G_{\mu\sslash \nu}(x_1,x_2,\dots, x_m)g_{\lambda/\nu}(y_1,y_2,\dots,y_n)
\end{multline*}

We can do the same for all the identities in this section but for simplicity we shall stick to the non-skew identities.

\begin{cor}
Generalised weak Grothendieck polynomials and their dual satisfy the following identity:
\begin{align}
\sum_{\lambda} J^{\alpha}_{\lambda}(x_1,\dots,x_m;z_1,z_2,\dots) j^{\alpha}_{\lambda} (y_1,\dots,y_n;\frac{1}{z_1},\frac{1}{z_2},\dots)= \prod_{1\le i \le m,1\le j\leq n} \dfrac{1}{1-x_iy_j}
\end{align}
\end{cor}
\begin{proof}
Observe that when $\beta=0$, the non constant entries of $\mathfrak{R}$ are $xy$ and $1-xy$. If we introduce the inhomogeneities as $xz$ and $\dfrac{y}{z}$, then the $\mathfrak{R}$ matrix remains the same and thereby giving the same commutation relation. Then the proof of the identity is same as the \cref{cauchyidentity}.
\end{proof}

We now prove the following Cauchy identity.
\begin{align}
\label{involcauchyeq}
    \sum_{\lambda}G^{(-\beta,-\alpha)}_{\lambda'}(x_1,x_2,\dots,x_m) g^{(\alpha,\beta)}_{\lambda}(y_1,y_2,\dots,y_n)=\prod_{1\le i \le m,1\le j\le n} {(1+x_i y_j)}
\end{align}

We can prove this identity by proving a commutation relation between $T^*$ and $\tilde{t}$. But we shall prove it using the inversion relation from \cref{inversionGroth}.

\begin{thm}
\label{involcauchy}
Canonical Grothendieck polynomials and their dual  satisfy the following Cauchy identity:
\begin{align*}
    \sum_{\lambda}G^{(-\beta,-\alpha)}_{\lambda'}(x_1,x_2,\dots,x_m) g^{(\alpha,\beta)}_{\lambda}(y_1,y_2,\dots,y_n)=\prod_{1\le i\le m,1\le j\le n} {(1+x_i y_j)}
\end{align*}
\end{thm}

\begin{proof}
By substituting $-x$ for $x$ in \cref{cauchytransferrel}, we get the following relation.
\begin{align*}
{t}(y) {T^{*}}(-x)&=\text{ } \frac{1}{1+xy} {T^{*}}(-x) {t}(y)\\
\end{align*}
By multiplying $\widetilde{T}^{*}\br{\dfrac{x}{1+x(\beta-\alpha)}}$ on both sides and applying the inversion relation \cref{inversionGrothT*}, we get the following relation:
\begin{align*}
\widetilde{T}^{*} \br{\dfrac{x}{1+x(\beta-\alpha)}}{t}(y)&=\text{ } \frac{1}{1+xy}  {t}(y)\widetilde{T}^{*} \br{\dfrac{x}{1+x(\beta-\alpha)}}\\
\end{align*}

Let 
\begin{multline*}
\mathcal{G}(x_1,x_2,\dots, x_n,y_1,y_2,\dots,y_m)=\\
\bra{0}{t}(y_1){t}(y_2)\cdots {\widetilde{T}^{*}}\br{\dfrac{x_1}{1+(\beta-\alpha)x_1}}\cdots {T^{*}}\br{\dfrac{x_n}{1+(\beta-\alpha)x_n}}\ket{0}.
\end{multline*}

Then by \cref{thm:rowg} and the definition of $\widetilde{T}^{*}$ we obtain

\begin{align*}
\mathcal{G}(x_1,x_2,&\dots, x_n,y_1,y_2,\dots,y_m)
=\text{ }\sum_{\lambda}G_{\lambda'}^{(-\beta,-\alpha)}(x_1,\dots,x_n)\g(y_1,\dots,y_m)
 \end{align*}

On the other hand by repeatedly applying the commutation relation we get that
\begin{align*}
\mathcal{G}(x_1,x_2,&\dots, x_n,y_1,y_2,\dots,y_m)\\
=&\text{ }\prod_{1\leq n,1\le m}(1+x_iy_j)\bra{0} {\widetilde{T}^{*}}\br{\dfrac{x_1}{1+(\beta-\alpha)x_1}}\cdots {T^{*}}\br{\dfrac{x_n}{1+(\beta-\alpha)x_n}} {t}(y_1){t}\cdots(y_m)\ket{0}\\
=&\text{ }\prod_{1\le i \le n,1\le j \le m}(1+x_iy_j)
 \end{align*}
\end{proof}

\begin{cor}
Generalised Grothendieck polynomials and generalised weak dual Grothendieck polynomials satisfy the following identity:
\begin{align}
\sum_{\lambda} G_{\lambda}(x_1,\dots,x_m;z_1,z_2,\dots) j_{\lambda} (y_1,\dots,y_n;\frac{1}{z_1},\frac{1}{z_2},\dots)= \prod_{1\le i\le m,1\le j\leq n} \left({1+x_iy_j} \right)
\end{align}
\end{cor}
\begin{proof}
Plug in $\beta=0$ and $\alpha=1$ in \cref{involcauchy} and use the inhomogeneous transfer matrices.
\end{proof}

\begin{thm}
Generalised Grothendieck polynomials and their dual satisfy the following identity:
\begin{align}
\label{grothvsdual}
\sum_{\lambda} G_{\lambda}(x_1,\dots,x_m;z_1,\dots,z_m) g_{\lambda} (y_1,\dots,y_n;\frac{1}{z_1},\dots,\frac{1}{z_m})= \prod_{1\le i\le m,1\le j\leq n} \dfrac{1}{1-x_iy_j}
\end{align}
\end{thm}
\begin{proof}
Recall the commutation relation from \cref{involcauchy}:
\begin{align*}
\widetilde{T}^{*} \br{\dfrac{x}{1+x(\beta-\alpha)}}{t}(y)&=\text{ } \frac{1}{1+xy}  {t}(y)\widetilde{T}^{*} \br{\dfrac{x}{1+x(\beta-\alpha)}}\\
\end{align*}

In order to apply the inversion relation among the transfer matrices of the dual Grothendieck polynomials, we need to specialize the above commutation relation with $\alpha=1$ and $\beta=0$. We then get the following relation:
\begin{align*}
\widetilde{T}^{*} \br{\dfrac{\frac{x}{z}}{1-{\frac{x}{z}}}}{t}\br{-yz}&=\, \frac{1}{1-xy}  {t}\br{-yz}\widetilde{T}^{*} \br{\dfrac{\frac{x}{z}}{1-\frac{x}{z}}}\\
\end{align*}
 We now multiply the above equation by $\tilde{t}(yz)$ and apply the inversion relation (\cref{inversionDualGroth}).
 \begin{align*}
\tilde{t}(yz)\widetilde{T}^{*} \br{\dfrac{\frac{x}{z}}{1-\frac{x}{z}}}&=\, \frac{1}{1-xy}  \widetilde{T}^{*} \br{\dfrac{\frac{x}{z}}{1-\frac{x}{z}}} \tilde{t}(yz)\\
\end{align*}
Then the result follows immediately from the definition of $\widetilde{T}^{*}$, and \cref{thm:colg}.
\end{proof}

\begin{prop}
Generalised  Grothendieck polynomials satisfy 
 \begin{align*}
 G_{\lambda}(z_1,\dots,z_m;z_1,\dots,z_m)=1
 \end{align*}
\end{prop}

\begin{proof}
When $\alpha=0$ and $\beta=-1$, $\widetilde{R}$ and $\widetilde{L}$ are same. The $\widetilde{R}$ matrix satisfies the unitary relation (for proof refer to \cref{unitary}).
\[
\begin{tikzpicture}[scale=1,baseline=7pt]
\draw[bosonic,Gcol,rounded corners] (0,0) node[left,black]{$\ss y $} --(1,1)--(2,0);
\draw[bosonic,Gcol,rounded corners] (0,1)node[left,black]{$\ss x $} --(1,0)--(2,1);
\end{tikzpicture}=
\begin{tikzpicture}[scale=1,baseline=7pt]
\draw[bosonic,Gcol,rounded corners] (0,0) node[left,black]{$\ss y $} --(2,0);
\draw[bosonic,Gcol,rounded corners] (0,1)node[left,black]{$\ss x $} --(2,1);
\end{tikzpicture}
\]

Recall that there are two types of vertices in the column model for Grothendieck polynomials.
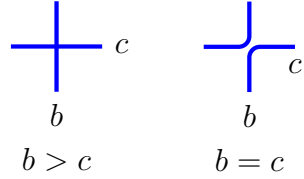
\[
\begin{tikzpicture}[scale=0.6]
\draw[bosonic,Gcol](0,0)--(2,0) node[right,black]{$c$};
\draw[bosonic,Gcol](1,-1)node[below,black]{$b$}--(1,1);
\node at (1,-2.5) {$ b>c$};
\end{tikzpicture}
\qquad
\begin{tikzpicture}[scale=0.6]
\draw[bosonic,Gcol,rounded corners](0,0)--(1,0)--(1,1);
\draw[bosonic,Gcol, rounded corners](1,-1)node[below,black]{$b$}--(1,0)--(2,0) node[below,black]{$c$};
\node at (1,-2.5) {$ b=c$};
\end{tikzpicture}
\]
We shall now argue that the contribution from the vertices that are below the anti-diagonal is trivial. To see that, let us consider the model with $m$ rows i.e., Grothendieck polynomial in $m$ variables. We know that the number of variables restricts us to consider the partitions with highest column being less than or equal to $m$. So, the only  inhomgeneities are $z_1,\dots,z_m$.

Consider the first row from the bottom in the vertex model. At the first site, we can only have an elbow as the bottom label is $0$. Then we get that the nodes around every other vertex are just $0's$. 

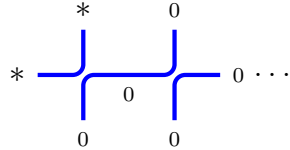
\[
\begin{tikzpicture}[scale=0.6]
\draw[bosonic,Gcol,rounded corners](0,0) node[left,black]{$*$}--(1,0)--(1,1)node[above,black]{$*$};
\draw[bosonic,Gcol, rounded corners](1,-1)node[below,black]{$\ss 0$}--(1,0)--(2,0) node[below,black]{$\ss 0$}--(3,0)--(3,1)node[above,black]{$\ss 0$};
\draw[bosonic,Gcol,rounded corners](3,-1) node[below,black]{$\ss 0$}--(3,0)--(4,0)node[right,black]{$\ss 0$};
\node at (5.2,0) {$\dots$};
\end{tikzpicture}
\]

Observe that the top labels of the are all $0's$ except the first label. By the same argument, we get that all the nodes of the vertices below the anti-diagonal are just $0's$, and the contribution of such vertices is $1$.

 Recall that the Botlzmann weight of a crossing has a $\br{1-\dfrac{x}{z}}$ factor. When we plug in $x_i=z_i$, then the contribution of a configuration to partition function is non zero only when the diagonals in a configuration are elbows.

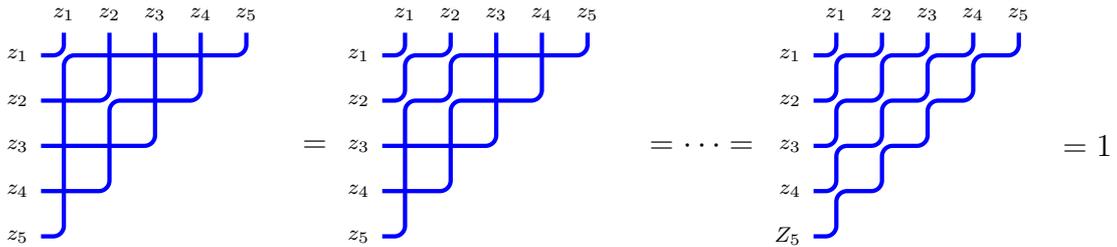
\[
\begin{tikzpicture}[scale=0.6,baseline=10pt]
\foreach\x in {0}{
\draw[bosonic,Gcol,rounded corners](0,4.5+\x) node[left,black]{$\ss z_1$}--(0.5,4.5+\x)--(0.5,5) node[above,black]{$\ss z_1$};
\draw[bosonic,Gcol,rounded corners](0,0.5) node[left,black]{$\ss z_5$}--(0.5,0.5)--(0.5,4.5)--(4.5,4.5)--(4.5,5) node[above,black]{$\ss z_5$};
\draw[bosonic,Gcol,rounded corners](0,3.5+\x) node[left,black]{$\ss z_2$}--(1.5,3.5+\x)--(1.5,5) node[above,black]{$\ss z_2$};
\draw[bosonic,Gcol,rounded corners](0,2.5) node[left,black]{$\ss z_3$}--(2.5,2.5)--(2.5,5) node[above,black]{$\ss z_3$};
\draw[bosonic,Gcol,rounded corners](0,1.5) node[left,black] {$\ss z_4$}--(1.5,1.5)--(1.5,3.5)--(3.5,3.5)--(3.5,5) node[above,black]{$\ss z_4$};
};
\node at (6,2.5) {$=$};
\end{tikzpicture}
\begin{tikzpicture}[scale=0.6,baseline=10pt]
\foreach\x in {0}{
\draw[bosonic,Gcol,rounded corners](0,4.5+\x) node[left,black]{$\ss z_1$}--(0.5,4.5+\x)--(0.5,5) node[above,black]{$\ss z_1$};
\draw[bosonic,Gcol,rounded corners](0,0.5) node[left,black]{$\ss z_5$}--(0.5,0.5)--(0.5,3.5)--(1.5,3.5)--(1.5,4.5)--(4.5,4.5)--(4.5,5) node[above,black]{$\ss z_5$};
\draw[bosonic,Gcol,rounded corners](0,3.5+\x) node[left,black]{$\ss z_2$}--(0.5,3.5+\x)--(0.5,4.5)--(1.5,4.5)--(1.5,5) node[above,black]{$\ss z_2$};
\draw[bosonic,Gcol,rounded corners](0,2.5) node[left,black]{$\ss z_3$}--(2.5,2.5)--(2.5,5) node[above,black]{$\ss z_3$};
\draw[bosonic,Gcol,rounded corners](0,1.5) node[left,black] {$\ss z_4$}--(1.5,1.5)--(1.5,3.5)--(3.5,3.5)--(3.5,5) node[above,black]{$\ss z_4$};
};
\node at (7,2.5){$= \cdots=$};
\end{tikzpicture}
\begin{tikzpicture}[scale=0.6,baseline=10pt]
\draw[bosonic,Gcol,rounded corners](0,4.5) node[left,black]{$\ss z_1$}--(0.5,4.5)--(0.5,5) node[above,black]{$\ss z_1$};
\draw[bosonic,Gcol,rounded corners](0,3.5) node[left,black]{$\ss z_2$}--(0.5,3.5)--(0.5,4.5)--(1.5,4.5)--(1.5,5) node[above,black]{$\ss z_2$};
\draw[bosonic,Gcol,rounded corners] (0,2.5) node[left,black]{$\ss z_3$}--(0.5,2.5)--(0.5,3.5)--(1.5,3.5)--(1.5,4.5)--(2.5,4.5)--(2.5,5) node[above,black]{$\ss z_3$};
\draw[bosonic,Gcol,rounded corners] (0,1.5) node[left,black]{$\ss z_4$}--(0.5,1.5)--(0.5,2.5)--(1.5,2.5)--(1.5,3.5)--(2.5,3.5)--(2.5,4.5)--(3.5,4.5)--(3.5,5) node[above,black]{$\ss z_4$};
\draw[bosonic,Gcol,rounded corners] (0,0.5) node[left,black]{$\ss Z_5$}--(0.5,0.5)--(0.5,1.5)--(0.5+1,1.5)--(0.5+1,2.5)--(1.5+1,2.5)--(1.5+1,3.5)--(2.5+1,3.5)--(2.5+1,4.5)--(3.5+1,4.5)--(3.5+1,5) node[above,black]{$\ss z_5$};
\node at (6,2.5) {$=1$};
\end{tikzpicture}
\]
Then we get the desired result by repeatedly applying the unitary relation.
\end{proof}

As a consequence of the above proposition, we recover an identity for \emph{generalised dual Grothendieck polynomials}, which is proved by Yeliussizov in \cite{DY-dualcauchy}.
\begin{cor}
Dual Grothendieck polynomials satisfy the following identity.
\begin{align*}
\sum_{l(\lambda)\leq m} g_{\lambda}(y_1,\dots,y_n;\frac{1}{z_1},\dots,\frac{1}{z_m})=\prod_{1\le i\le m,1\le j\le n}\dfrac{1}{1-z_i y_j}
\end{align*}
\end{cor}
\begin{proof}
Set $x_i=z_i$ in \cref{grothvsdual}.
\end{proof}

\newcommand\subsubsubsection[1]{\par\noindent\textit{#1.}\par}

\appendix
\section{\texorpdfstring{$\mathrm{RLL}$}{RLL} relations}
\subsection{\texorpdfstring{$\mathrm{RLL}$}{RLL} for column model of  \texorpdfstring{$\G$}{G}.}
For convenience, let us recall the Boltzmann weights of the model and the entries of the $\widetilde{R}$ matrix.
\begin{align}
w^{}_{x}
\left(
\begin{gathered}
\begin{tikzpicture}[scale=0.4,baseline=-2pt]
\draw[bosonic,Gcol,arrow=0.25] (-1,0) node[left,black] {$a$} -- (1,0) node[right,black] {$c$};
\draw[bosonic,arrow=0.25] (0,-1) node[below] {$b$} -- (0,1) node[above] {$d$};
\end{tikzpicture}
\end{gathered}
\right)
\equiv
w^{}_{x}(a,b;c,d)= \delta_{a+b,c+d}
\begin{cases}
\left(\frac{x} {1-\alpha x} \right)^{a}&  b=c\\
\left(\frac{x} {1-\alpha x} \right)^{a}\left(\frac{1+\beta x} {1-\alpha x} \right) & b>c\\
0& b<c\\
\end{cases}
\quad
a,b,c,d \in \mathbb{Z}_{\geq 0},
\quad
\end{align}

\begin{align*}
\widetilde{R}^{a\,d}_{b\,c}\left(y,x\right) =\begin{gathered}
\begin{tikzpicture}[baseline=-3pt,scale=0.9]
\draw[bosonic,arrow=0.35,Gcol] (-2,0.5) node[left,black] {$a$} -- (-1,-0.5) node[below] {};
\node[text width=1cm] at (-0.4,-0.55){$c$};
\draw[bosonic,arrow=0.35,Gcol] (-2,-0.5) node[left,black] {$b$} -- (-1,0.5) node[above] {};
\node[text width=1cm] at (-0.4,0.5){$d$};
\end{tikzpicture}
\end{gathered}= \left({\dfrac{\dfrac{x}{1-\alpha x}}{\dfrac{y}{1-\alpha y}}}\right)^{a}
  \begin{cases}
    0 & \text{when } b<c\\
   1 & \text{when } b=c \\
  \dfrac{1}{1-\alpha x}-\dfrac{x}{(1-\alpha x)y} & \text{otherwise } 
  \end{cases}
\end{align*}

Let us try to understand the range of $g$. First observe that whenever $a'<g$, the summation is $0$ because of the $\widetilde{R}$-matrix. Based on the Boltzmann weights, the contribution of the top vertex is non zero if and only if $g+b-c\geq c'$. Therefore, $g$ on $LHS$ can at most be $a'$, and it has to be at least $c'+c-b$. Similarly, on  the $RHS$ we have $a+a'-d\leq c'$.
\[
\sum^{a'}_{g={c+c'-b}}
\begin{tikzpicture}[scale=1,baseline=7pt]
\draw[bosonic](1,-1)--(1,2);
\draw[bosonic,arrow=0.35,rounded corners,Gcol] (-1,1)--(0,0)--(2,0);
\draw[bosonic,arrow=0.35,rounded corners,Gcol] (-1,0)--(0,1)--(2,1);
\node at (-1.2,0) {$ a'$};
\node at (-1.2,1) {$ a$};
\node at (1,-1.3) {$b$};
\node at (1,2.3) {$d$};
\node at (2.2,1) {$c'$};
\node at (2.2,0) {$c$};
\node at (1.6,0.5) {$\ss g+b-c$};
\node at (0,-0.3) {$ g$};
\node at (0,1.3) {$ \ss a+a'-g$};
\end{tikzpicture}
=
\sum^{c'}_{g=a+a'-d}
\begin{tikzpicture}[scale=1,baseline=7pt]
\draw[bosonic] (1,2)--(1,-1);
\draw[bosonic,arrow=0.7,Gcol,rounded corners] (0,0)--(2,0)--(3,1);
\draw[bosonic,arrow=0.7,Gcol,rounded corners] (0,1)--(2,1)--(3,0);
\node at (-0.2,0) {$a'$};
\node at (-0.2,1) {$a$};
\node at (1,-1.3) {$b$};
\node at (1,2.3) {$d$};
\node at (3.2,0) {$c$};
\node at (3.2,1) {$c'$};
\node at (1.6,0.5) {$\ss g+d-a$};
\node at (2,-0.3) {$\ss c+c'-g$};
\node at (2,1.3) {$g $};
\end{tikzpicture}
\]

\subsubsection{Assume \texorpdfstring{$b>c$ and $d>a$}{b>c and d>a}.}
\label{colGcase1}
Let us now compute the $LHS$.
{\allowdisplaybreaks
\begin{align*}
   LHS=&\text{ } \br{\dfrac{y-x}{y(1-\alpha x)}} \br{\dfrac{x(1-\alpha y)}{y(1-\alpha x)}}^{a} \br{\dfrac{y}{1-\alpha y}}^{d} \br{\dfrac{x}{1-\alpha x}}^{c+c'-b}\br{\dfrac{1+\beta x}{1-\alpha x}}+\\
    &\br{\dfrac{y-x}{y(1-\alpha x)}} \br{\dfrac{x(1-\alpha y)}{y(1-\alpha x)}}^{a} \br{\dfrac{y}{1-\alpha y}}^{a+a'}\br{\dfrac{1+\beta y}{1-\alpha y}}\br{\dfrac{1+\beta x}{1-\alpha x}}\\
   & \br{\sum^{a'-1}_{g=c+c'-b+1} \br{\dfrac{x(1-\alpha y)}{x(1-\alpha x)}}^{g}}+\\
   &\br{\dfrac{x(1-\alpha y)}{y(1-\alpha x)}}^{a} \br{\dfrac{x}{1-\alpha x}}^{a'}\br{\dfrac{1+\beta x}{1-\alpha x}} \br{\dfrac{y}{1-\alpha y}}^{a}\br{\dfrac{1+\beta y}{1-\alpha y}}\\
   =& \text{ }\br{\dfrac{1+\beta x}{1-\alpha x}}^2 \br{\dfrac{x}{1-\alpha x}}^{a+c+c'-b}\br{\dfrac{y}{1-\alpha y}}^{d-a} \\
\end{align*}
}

We compute the right hand side of the equation:
{
\allowdisplaybreaks
\begin{align*}
RHS=    &\text{ } \br{\dfrac{y-x}{y(1-\alpha x)}} \br{\dfrac{x(1-\alpha   y)}{y(1-\alpha x)}}^{a+a'-d} \br{\dfrac{x}{1-\alpha x}}^{a}\br{\dfrac{1+\beta x}{1-\alpha x}} \br{\dfrac{y}{1-\alpha y}}^{a'}+\\
        &\left( \sum^{c'-1}_{g=a+a'-d+1}  \br{\dfrac{y-x}{y(1-\alpha x)}} \br{\dfrac{x(1-\alpha y)}{y(1-\alpha x)}}^{g} \br{\dfrac{x}{1-\alpha x}}^{a}\right.\\
        &\left.\br{\dfrac{1+\beta x}{1-\alpha x}} \br{\dfrac{y}{1-\alpha y}}^{a'} \br{\dfrac{1+\beta y}{1-\beta y}}\right)+\\
        &    \br{\dfrac{x(1-\alpha y)}{y(1-\alpha x)}}^{c'} \br{\dfrac{x}{1-\alpha x}}^{a}\br{\dfrac{1+\beta x}{1-\alpha x}} \br{\dfrac{y}{1-\alpha y}}^{a'} \br{\dfrac{1+\beta y}{1-\beta y}}\\
       =& \text{ } \br{\dfrac{1+\beta x}{1-\alpha x}}^2 \br{\dfrac{x}{1-\alpha x}}^{a+c+c'-b} \br{\dfrac{y}{1-\alpha y}}^{d-a}\\
\end{align*}}
\subsubsection{Assume \texorpdfstring{$a<d$ and $b=c$}{a<d and b=c}.}
\label{colGcase2}

From the computation of the previous case, we can get the $LHS$ by multiplying $\dfrac{1-\alpha x}{1+\beta x}$ to the $LHS$ of previous computation.
\[
LHS= \text{ }\br{\dfrac{1+\beta x}{1-\alpha x}} \br{\dfrac{x}{1-\alpha x}}^{a+c'}\br{\dfrac{y}{1-\alpha y}}^{d-a} \\
\]

On the right hand side, there is only one case because of the global condition, $a+a'+b=c+c'+d$.
\begin{align*}
RHS=\text{ }&\br{\dfrac{x}{1-\alpha x}}^{a+c'}\br{\dfrac{1+\beta x}{1-\alpha x}}\br{\dfrac{y}{1-\alpha y}}^{a'-c'}\\
    =\text{ }&\br{\dfrac{x}{1-\alpha x}}^{a+c'}\br{\dfrac{1+\beta x}{1-\alpha x}}\br{\dfrac{y}{1-\alpha y}}^{d-a}\\
\end{align*}

\subsubsection{Assume \texorpdfstring{$a=d$ and $b>c$}{a=d and b>c}.}
\label{colGcase3}
 $RHS$ of the present case is a $\dfrac{1-\alpha x}{1+\beta x}$ multiple of the $RHS$ of the \cref{colGcase1}.

\begin{align*}
RHS=&  \text{ } \br{\dfrac{1+\beta x}{1-\alpha x}} \br{\dfrac{x}{1-\alpha x}}^{a+c+c'-b} \br{\dfrac{y}{1-\alpha y}}^{d-a}\\
 &  \br{\dfrac{1+\beta x}{1-\alpha x}} \br{\dfrac{x}{1-\alpha x}}^{a+c+c'-b}\\
\end{align*}

For the $LHS$, there is just one valid configuration:

\begin{align*}
LHS=\text{ }& \br{\dfrac{x}{1-\alpha x}}^{a} \br{\dfrac{y}{1-\alpha y}}^{-a} \br{\dfrac{y}{1-\alpha y}}^{a} \br{\dfrac{x}{1-\alpha x}}^{a'}\br{\dfrac{1+\beta x}{1-\alpha x}}\\
 =\text{ }& \br{\dfrac{x}{1-\alpha x}}^{a+a'} \br{\dfrac{1+\beta x}{1-\alpha x}}\\
 =\text{ }& \br{\dfrac{x}{1-\alpha x}}^{a+c+c'-b} \br{\dfrac{1+\beta x}{1-\alpha x}}\\
\end{align*}
In the final step, we substitute $a'=c+c'-b$, which follows from the global condition.

\subsubsection{Assume \texorpdfstring{$a=d$ and $b=c$}{a=d and b=c}.}

Recall from \cref{colGcase3} that when $a=d$, there is a unique configuration on $LHS$. Similarly, recall from \cref{colGcase2} that there is a unique configuration on $RHS$ when $b=c$. We have already computed these two configurations and they are equal.

\subsubsection{Unitary relation for the $\widetilde{R}$ matrix.}
\label{unitary}
The $\widetilde{R}$ stisfies the unitary relation.
\[
\begin{tikzpicture}[scale=1]
\draw[bosonic,Gcol,rounded corners] (0,0) node[left,black]{$y$ \text{ }}--(1,1)--(2,0)node[right,black]{};
\draw[bosonic,Gcol,rounded corners] (0,1)node[left,black]{$x$ \text{ }}--(1,0)--(2,1)node[right,black]{};
\node at (2.7,0.5) {$=$};
\draw[bosonic,Gcol,rounded corners] (4,0) node[left,black]{$y$ \text{ }}--(6,0)node[right,black]{};
\draw[bosonic,Gcol,rounded corners] (4,1)node[left,black]{$x$ \text{ }}--(6,1)node[right,black]{};
\end{tikzpicture}
\]
Consider the following configuration:
\begin{align*}
\begin{tikzpicture}[scale=1]
\draw[bosonic,Gcol,rounded corners] (0,0) node[left,black]{$a'$}--(1,1)--(2,0)node[right,black]{$b'$};
\draw[bosonic,Gcol,rounded corners] (0,1)node[left,black]{$a$}--(1,0)--(2,1)node[right,black]{$b$};
\end{tikzpicture}
\end{align*}

When $a=b$ and $a'=b'$, there is a unique configuration with weight $1$.
Now let us assume $a\neq b$ and $a'\neq b'$.

\begin{align*}
=&\text{ }\br{\dfrac{x}{y}}^{a}\br{1-\dfrac{x}{y}}\br{\dfrac{y}{x}}^{a+a'-b'}+\sum^{a'-1}_{g=b'+1} \br{\dfrac{x}{y}}^{a}\br{1-\dfrac{x}{y}}\br{\dfrac{y}{x}}^{a+a'-g}\br{1-\dfrac{y}{x}}+\\ &\br{\dfrac{x}{y}}^{a}\br{\dfrac{y}{x}}^{a}\br{1-\dfrac{y}{x}}\\
=&\text{ }\br{1-\dfrac{x}{y}}\br{\dfrac{y}{x}}^{a'-b'}+\br{\dfrac{y}{x}}^{a'-b'-1}\br{1-\dfrac{y}{x}}\br{1-\br{\dfrac{x}{y}}^{a'-b'-1}}+\br{1-\dfrac{y}{x}}\\
=&\text{ }0
\end{align*}

\subsection{\texorpdfstring{$\mathrm{RLL}$}{RLL} relation for row model of \texorpdfstring{$\g$}{g}.}

We recall the Boltzmann weights and $r$ matrix of the row model of $\g$.

\begin{align}
w^{}_{x}
\left(
\begin{gathered}
\begin{tikzpicture}[scale=0.4,baseline=-2pt]
\draw[bosonic,gcol,arrow=0.25] (-1,0) node[left,black] {$a$} -- (1,0) node[right,black] {$c$};
\draw[bosonic,arrow=0.25] (0,-1) node[below] {$b$} -- (0,1) node[above] {$d$};
\end{tikzpicture}
\end{gathered}
\right)
\equiv
w^{}_{x}(a,b;c,d)
= \delta_{a+b,c+d} 
\begin{cases}
(\alpha+\beta)^{a-d-1} (x+\alpha) \beta^{d}& a>d\\
\beta^{a-1} x & 0<a\leq d\\
1& a=0,
\end{cases}
\end{align}
where 
$a,b,c,d$ $\in \mathbb{Z}_{\geq 0}$.

The entries of the $r$-matrix are the following:
\begin{align}
r^{a,d}_{b,c} (x,y)=
\begin{gathered}
\begin{tikzpicture}[baseline=-3pt,scale=0.9]
\draw[bosonic,arrow=0.35,gcol] (-2,0.5) node[left,black] {$a$} -- (-1,-0.5) node[below] {};
\node[text width=1cm] at (-0.4,-0.55){$c$};
\draw[bosonic,arrow=0.35,gcol] (-2,-0.5) node[left,black] {$b$} -- (-1,0.5) node[above] {};
\node[text width=1cm] at (-0.4,0.5){$d$};
\end{tikzpicture}
\end{gathered}=
\begin{cases}
0 & b>c\\
1 & b=c=0\\
\dfrac{y}{x} & b=c>0\\
\left(1- \dfrac{y}{x}\right) \left(1-\dfrac{y}{\beta} \right)^{a-d-1} & b=0\\
\left(1- \dfrac{y}{x} \right) \left(1-\dfrac{y}{\beta} \right)^{a-d-1} \left(\dfrac{y}{\beta}\right) & b>0\\
\end{cases}
\end{align}

\[
\sum^{a}_{g={0}}
\begin{tikzpicture}[scale=1,baseline=7pt]
\draw[bosonic](1,-1)--(1,2);
\draw[bosonic,arrow=0.35,gcol,rounded corners] (-1,1)--(0,0)--(2,0);
\draw[bosonic,arrow=0.35,gcol,rounded corners] (-1,0)--(0,1)--(2,1);
\node at (-1.2,0) {$ a'$};
\node at (-1.2,1) {$ a$};
\node at (1,-1.3) {$b$};
\node at (1,2.3) {$d$};
\node at (2.2,1) {$c'$};
\node at (2.2,0) {$c$};
\node at (1.6,0.5) {$\ss d+c'-g$};
\node at (0,-0.3) {$\ss a+a'-g$};
\node at (0,1.3) {$ g$};
\end{tikzpicture}
=
\sum^{c}_{g=0}
\begin{tikzpicture}[scale=1,baseline=7pt]
\draw[bosonic] (1,2)--(1,-1);
\draw[bosonic,arrow=0.7,gcol,rounded corners] (0,0)--(2,0)--(3,1);
\draw[bosonic,arrow=0.7,gcol,rounded corners] (0,1)--(2,1)--(3,0);
\node at (-0.2,0) {$a'$};
\node at (-0.2,1) {$a$};
\node at (1,-1.3) {$b$};
\node at (1,2.3) {$d$};
\node at (3.2,0) {$c$};
\node at (3.2,1) {$c'$};
\node at (1.6,0.5) {$\ss a'+b-g$};
\node at (2,-0.3) {$ g$};
\node at (2,1.3) {$\ss c+c'-g$};
\end{tikzpicture}
\]

Before we start proving the relation, let us analyze the cases we need to consider. Firstly, from the $LHS$, the weight of bottom vertex is fixed based on whether $b\leq c$ or $b>c$. Similarly, on the $RHS$, the weight of the top vertex is fixed based on the relation between $a$ and $d$.

We now look at the cases that arise from considering the entries of the $r$ matrix. Firstly, the entries of the $r$ matrix on $LHS$ depends on whether $a'>0$ or $a'=0$. Similarly, the entry of $r$ matrix on $RHS$ depends on whether $c>0$ or $c=0$.

In total, there are sixteen cases to consider. We shall divide these cases into four categories based on the conditions on $b,c$ and $a,d$. Then for each such case, we shall consider four sub-cases by the conditions on $a'$ and $c$.

\subsubsection{Assume \texorpdfstring{$b<c$ and $d< a$}{b<c and d<a}.}
\label{growcase1}
\subsubsubsection{Assume \texorpdfstring{$a'>0$ and $c>0$}{a'>0 and c>0}.}

 To ease up the computation, we break up the summation into two parts, $0\leq g \leq d$ and then ${d<g\leq a}$.
 
 Assume that $b<c$
\[
\sum^{d}_{g=0}
\begin{tikzpicture}[scale=0.8,baseline=6pt]
\draw[bosonic,gcol] (-1,1)--(0,0);
\draw[bosonic,gcol] (-1,0)--(0,1);
\node at (-1.3,0) {$ a'$};
\node at (-1.3,1) {$ a$};
\node at (0.2,-0.2) {$\ss a'+a-g$};
\node at (0.2,1) {$ g$};
\draw[bosonic] (2,0) node[below] {$\ss d+c'-g$}--(2,1) node[above] {$d$};
\draw[bosonic,gcol] (1.5,0.5) node[left,black] {$g$}--(2.5,0.5) node[right,black] {$c'$};
\draw[bosonic] (5.5,0) node[below] {$\ss b$}--(5.5,1) node[above] {$\ss {d+c'-g}$};
\draw[bosonic,gcol] (5,0.5) node[left,black] {$\ss {a+a'-g}$}--(6,0.5) node[right,black] {$c$};
\end{tikzpicture}+
\sum^{a}_{g=d+1}
\begin{tikzpicture}[scale=0.8,baseline=6pt]
\draw[bosonic,gcol] (-1,1)--(0,0);
\draw[bosonic,gcol] (-1,0)--(0,1);
\node at (-1.3,0) {$ a'$};
\node at (-1.3,1) {$ a$};
\node at (0.2,-0.2) {$\ss a'+a-g$};
\node at (0.2,1) {$ g$};
\draw[bosonic] (2,0) node[below] {$\ss d+c'-g$}--(2,1) node[above] {$d$};
\draw[bosonic,gcol] (1.5,0.5) node[left,black] {$g$}--(2.5,0.5) node[right,black] {$c'$};
\draw[bosonic] (5.5,0) node[below] {$ b$}--(5.5,1) node[above] {$\ss {d+c'-g}$};
\draw[bosonic,gcol] (5,0.5) node[left,black] {$\ss {a+a'-g}$}--(6,0.5) node[right,black] {$c$};
\end{tikzpicture}
\]
{\allowdisplaybreaks
\begin{align*}
  LHS   =\text{ }& \br{1-\dfrac{y}{x}} \br{1-\dfrac {y}{\beta}}^{a-1} \br{\dfrac{y}{\beta}} \br{\alpha+\beta}^{c-b-1}\br{x+\alpha}\beta^{d+c'}+\\
              & \sum^{d}_{g=1} \br{1-\dfrac{y}{x}} \br{1-\dfrac {y}{\beta}}^{a-g-1} \br{\dfrac{y}{\beta}} \br{\beta^{g-1}y} \br{\br{\alpha+\beta}^{c-b-1}\br{x+\alpha}\beta^{d+c'-g}}+\\
              &\sum^{a-1}_{g=d+1} \br{1-\dfrac{y}{x}} \br{1-\dfrac {y}{\beta}}^{a-g-1} \br{\dfrac{y}{\beta}} \br{\br{\alpha+\beta}^{g-d-1}\beta^{d}\br{y+\alpha}}\\
        &\br{\br{\alpha+\beta}^{c-b-1}\br{x+\alpha}\beta^{d+c'-g}}\\
        &+\br{\dfrac{y}{x}}\br{\br{\alpha+\beta}^{a-d-1}\beta^{d}\br{y+\alpha}} \br{\br{\alpha+\beta}^{c-b-1}\br{x+\alpha}\beta^{d+c'-a}}\\
     =\text{ }& (\alpha+\beta)^{c+a-b-d-2} \beta^{2d+c'-a} (x+\alpha) \br{y+ \alpha \dfrac{y}{x}}
\end{align*}
}
   
We compute the right hand side:
 Assume that $b<c$
\[
\sum^{c}_{g=0}
\begin{tikzpicture}[scale=0.8,baseline=3pt,yshift=0.3cm]
\draw[bosonic,gcol] (-1,0) node[left,black] {$a$}--(0,0) node[right,black] {$\ss {c+c'-g}$};
\draw[bosonic] (-0.5,-0.5) node[below] {$\ss {a'+b-g}$}--(-0.5,0.5) node[above] {$d$};
\draw[bosonic,gcol] (-1+3.5,0) node[left,black] {$a'$}--(0+3.5,0) node[right,black] {$g$};
\draw[bosonic] (-0.5+3.5,-0.5) node[below] {$b$}--(-0.5+3.5,0.5) node[above] {$\ss {a'+b-g}$};
\draw[bosonic,gcol] (4+1,0.5) node[above,black]{$\ss {c+c'-g}$} --(5+1,-0.5) node[right,black] {$c$};
\draw[bosonic,gcol] (4+1,-0.5) node[left,black] {$g$}--(5+1,0.5) node[right,black] {$c'$};
\end{tikzpicture}
\]

{\allowdisplaybreaks
\begin{align*}
         \dfrac{RHS}{(\alpha+\beta)^{a-d-1}\beta^{d}(x+\alpha)} =\text{ }&\br{1-\dfrac{y}{x}}
                    \br{1-\dfrac{y}{\beta}}^{c-1} \beta^{a'-1}y+\\
                   &\sum^{b}_{g=1}\br{1-\dfrac{y}{x}}\br{1-\dfrac{y}{\beta}}^{c-g-1}\dfrac{y}{\beta} \beta^{a'-1}y+\\
                   &\sum^{c-1}_{g=b+1} \br{1-\dfrac{y}{x}}\br{1-\dfrac{y}{\beta}}^{c-g-1}\dfrac{y}{\beta} \br{(\alpha+\beta)^{g-b-1}(y+\alpha)\beta^{a'+b-g}}+\\
                   &\dfrac{y}{x} \br{(\alpha+\beta)^{c-b-1}(y+\alpha)\beta^{a'+b-c}}\\
          =\text{ }&  \br{1-\dfrac{y}{x}}(\alpha+\beta)^{c-b-1}\beta^{a'+b-c}y+\\          &\dfrac{y}{x} \br{(\alpha+\beta)^{c-b-1}(y+\alpha)\beta^{a'+b-c}}\\
      RHS=\text{ }& (\alpha+\beta)^{a+c-d-b-2} \beta^{d+a'+b-c}\br{y+ \alpha \dfrac{y}{x}}(x+\alpha)\\ 
      =\text{ }& (\alpha+\beta)^{a+c-d-b-2} \beta^{2d+c'-a}\br{y+ \alpha \dfrac{y}{x}}(x+\alpha)\\ 
\end{align*}
}

\subsubsubsection{Assume \texorpdfstring{$a'=0$ and $c>0$}{a'=0 and c>0}.}
{\allowdisplaybreaks
\begin{align*}
LHS   =\text{ }& \br{1-\dfrac{y}{x}} \br{1-\dfrac {y}{\beta}}^{a-1}  \br{\alpha+\beta}^{c-b-1}\br{x+\alpha}\beta^{d+c'}+\\
              & \sum^{d}_{g=1} \br{1-\dfrac{y}{x}} \br{1-\dfrac {y}{\beta}}^{a-g-1}  \br{\beta^{g-1}y} \br{\br{\alpha+\beta}^{c-b-1}\br{x+\alpha}\beta^{d+c'-g}}+\\
              &\sum^{a+b-c}_{g=d+1} \br{1-\dfrac{y}{x}} \br{1-\dfrac {y}{\beta}}^{a-g-1}  \br{\br{\alpha+\beta}^{g-d-1}\beta^{d}\br{y+\alpha}}\\
              &\br{\br{\alpha+\beta}^{c-b-1}\br{x+\alpha}\beta^{d+c'-g}}\\
        =\text{ }& \br{1-\dfrac{y}{x}}\br{1-\dfrac{y}{\beta}}^{c-b-1}(\alpha+\beta)^{a-d-1}(x+\alpha)\beta^{d}
\end{align*}
}

We compute the $RHS$:
{\allowdisplaybreaks
\begin{align*}
         RHS =\text{ }&(\alpha+\beta)^{a-d-1}\beta^{d}(x+\alpha)\left(\br{1-\dfrac{y}{x}}
                     \br{1-\dfrac{y}{\beta}}^{c-1} +\right.\\ &\left.\br{\sum^{b}_{g=1}\br{1-\dfrac{y}{x}}\br{1-\dfrac{y}{\beta}}^{c-g-1}\dfrac{y}{\beta}} \right)\\
        =\text{ }&(\alpha+\beta)^{a-d-1}\beta^{d}(x+\alpha)\br{1-\dfrac{y}{x}}\br{1-\dfrac{y}{\beta}}^{c-b-1}        \\
\end{align*}
}

Since we assumed $b>c$, we dont need to consider the cases where $c=0$.

\subsubsection{Assume \texorpdfstring{$b\geq c$ and $d<a$}{b>=c and d<a}.}
\label{growcase2}
\subsubsubsection{Assume  \texorpdfstring{$a'>0$ and $c>0$}{a'>0 and c>0}. }
{\allowdisplaybreaks
\begin{align*}
  LHS   =\text{ }& \br{1-\dfrac{y}{x}} \br{1-\dfrac {y}{\beta}}^{a-1} \br{\dfrac{y}{\beta}} x\beta^{a+a'-1}+\\
              & \sum^{d}_{g=1} \br{1-\dfrac{y}{x}} \br{1-\dfrac {y}{\beta}}^{a-g-1} \br{\dfrac{y}{\beta}} \br{\beta^{g-1}y}  x\beta^{a+a'-g-1}+\\
              &\sum^{a-1}_{g=d+1} \br{1-\dfrac{y}{x}} \br{1-\dfrac {y}{\beta}}^{a-g-1} \br{\dfrac{y}{\beta}} \br{\br{\alpha+\beta}^{g-d-1}\beta^{d}\br{y+\alpha}}  x\beta^{a+a'-g-1}\\
              &+\br{\dfrac{y}{x}}\br{\br{\alpha+\beta}^{a-d-1}\beta^{d}\br{y+\alpha}} x\beta^{a'-1}\\
    =\text{ }& \br{\dfrac{y}{\beta}}(\alpha+\beta)^{a-d-1}\beta^{a'+d}(x+\alpha)\\          
\end{align*}
}

Since $a<d$ and $b\geq c$, the Boltzmann weights from the vertices are fixed. When we factor out the contribution from the overall sum, we are left with entries of the $r$ matrix with fixed right boundary. Then from the eigenvector property of the $r$ matrix, we get that the overall sum of $RHS$ is just the product of the fixed Boltzmann weights.
\[
RHS=\br{(\alpha+\beta)^{a-d-1}(x+\alpha)\beta^{d}}\br{\beta^{a'-1}y}
\]

\subsubsubsection{Assume \texorpdfstring{$a'=0$}{a'=0}}

{\allowdisplaybreaks
\begin{align*}
  LHS   =\text{ }& \br{1-\dfrac{y}{x}} \br{1-\dfrac {y}{\beta}}^{a-1}  x\beta^{a-1}+\\
              & \sum^{d}_{g=1} \br{1-\dfrac{y}{x}} \br{1-\dfrac {y}{\beta}}^{a-g-1}  \br{\beta^{g-1}y}  x\beta^{a-g-1}+\\
              &\sum^{a-1}_{g=d+1} \br{1-\dfrac{y}{x}} \br{1-\dfrac {y}{\beta}}^{a-g-1}  \br{\br{\alpha+\beta}^{g-d-1}\beta^{d}\br{y+\alpha}}  x\beta^{a+a'-g-1}\\
              &+\br{\br{\alpha+\beta}^{a-d-1}\beta^{d}\br{y+\alpha}}\\
    =\text{ }& (\alpha+\beta)^{a-d-1}\beta^{d}(x+\alpha)\\          
\end{align*}
}
For $RHS$, we use the eigenvector property of the $r$ matrix to conclude that $RHS$ is just the product the Boltzmann weights. Observe that the contribution from the bottom vertex is $1$. So, the overall $RHS$ is just the Boltzmann weight of the top vertex.
\[
RHS=(\alpha+\beta)^{a-d-1}(x+\alpha)\beta^{d}
\]

\subsubsection{Assume \texorpdfstring{$a\leq d$ and $b<c$}{a<=d and b<c}.}
Since we are assuming $b<c$, its follows that $c>0$.
\subsubsubsection{Assume \texorpdfstring{$a'>0$}{a'>0}}
\label{growcase3}
{\allowdisplaybreaks
\begin{align*}
  LHS   =\text{ }& \br{1-\dfrac{y}{x}} \br{1-\dfrac {y}{\beta}}^{a-1} \br{\dfrac{y}{\beta}} \br{\alpha+\beta}^{c-b-1}\br{x+\alpha}\beta^{d+c'}+\\
              & \sum^{a-1}_{g=1} \br{1-\dfrac{y}{x}} \br{1-\dfrac {y}{\beta}}^{a-g-1} \br{\dfrac{y}{\beta}} \br{\beta^{g-1}y} \br{\br{\alpha+\beta}^{b-c-1}\br{x+\alpha}\beta^{d+c'-g}}+\\
              &+\br{\dfrac{y}{x}}\br{\beta^{a-1}y} \br{\br{\alpha+\beta}^{c-b-1}\br{x+\alpha}\beta^{d+c'-a}}\\
     =\text{ }&(\alpha+\beta)^{c-b-1}(x+\alpha)\beta^{d+c'}\br{\dfrac{y}{\beta}}\\        
\end{align*}
}

Given that $b<c$, we can get the $RHS$ computation from \cref{growcase1}. The only difference being the Boltzmann weight corresponding to the top vertex.
\begin{align*}
RHS =\text{ }& x\beta^{a-1} (\alpha+\beta)^{c-b-1}\beta^{a'+b-c} \br{\br{1-\dfrac{y}{x}}y+ \dfrac{y}{x}(y+\alpha)}\\ 
 =\text{ }& (\alpha+\beta)^{c-b-1}\beta^{a'+b-c+a}\br{\dfrac{y}{\beta}}(x+\alpha)\\
 =\text{ }& (\alpha+\beta)^{c-b-1}\beta^{d+c'}\br{\dfrac{y}{\beta}}(x+\alpha)
\end{align*}

We do not need to consider the case where $a'=0$ as the global condition forces $c'$ to negative.
\begin{align*}
c'=&(a-d)+(b-c)<0\\
\end{align*}

\subsubsection{Assume \texorpdfstring{$a\leq d$ and $b\geq c$}{a<=d and b>=c}.}

\subsubsubsection{Assume \texorpdfstring{$a'>0$}{a'>0}}

{\allowdisplaybreaks
\begin{align*}
LHS=\text{ }& \br{1-\dfrac{y}{x}}\br{1-\dfrac{y}{\beta}}^{a-1}\br{\dfrac{y}{\beta            }}\beta^{a+a'-1}x+\\
           &\sum^{a-1}_{g=1}\br{1-\dfrac{y}{x}}\br{1-\dfrac{y}{\beta}}^{a-g-1}\br{\dfrac{y}{\beta}} \beta^{g-1}y \beta^{a+a'-g-1}x+\\
           &\dfrac{y}{x}\beta^{a-1}y\beta^{a'-1}x\\
           =\text{ }& \br{y\beta^{a'-1}}\br{x\beta^{a-1}}\\
\end{align*}
}

As a result of the assumptions, the Boltzmann weights are fixed. Using the eigenvector property, $RHS$ is just the product of the two fixed Boltzmann weights.
\[
RHS=\br{x\beta^{a-1}} \br{y\beta^{a'-1}}
\]

\subsubsubsection{Assume \texorpdfstring{$a'=0$}{a'=0}}

{
\allowdisplaybreaks
\begin{align*}
LHS=\text{ }& \br{1-\dfrac{y}{x}}\br{1-\dfrac{y}{\beta}}^{a-1}\beta^{a-1}x+\\
            &\sum^{a-1}_{g=1}\br{1-\dfrac{y}{x}}\br{1-\dfrac{y}{\beta}}^{a-g-1} \beta^{g-1}y \beta^{a-g-1}x+\beta^{a-1}x\\
   =\text{ }&  \beta^{a-1}x\\       
\end{align*}
}

On the $RHS$, its just the Boltzmann weight of the top vertex.
\[
RHS=\beta^{a-1}x
\]

\subsection{\texorpdfstring{$\mathrm{RLL}$}{RLL} for column model of \texorpdfstring{$\g$}{g}.}
We recall the Boltzmann weights, and entries of $\tilde{r}$ matrix.
\begin{align}
w^{}_{x}
\left(
\begin{gathered}
\begin{tikzpicture}[scale=0.4,baseline=-2pt]
\draw[bosonic,gccol,arrow=0.25] (-1,0) node[left,black] {$a$} -- (1,0) node[right,black] {$c$};
\draw[bosonic,arrow=0.25] (0,-1) node[below] {$b$} -- (0,1) node[above] {$d$};
\end{tikzpicture}
\end{gathered}
\right)
\equiv
w^{}_{x}(a,b;c,d)
= \delta_{a+b,c+d}
\begin{cases}
\beta (\alpha + \beta)^{a-d-1} (x+\alpha)^{d} & 0<a>d\\
 x (x+\alpha)^{a-1} & 0<a\leq d\\
1& a=0.
\end{cases}
\end{align}

\begin{align}
\widetilde{r}^{k,l}_{i,j} (x,y)=\begin{gathered}
\begin{tikzpicture}[baseline=-3pt,scale=0.9]
\draw[bosonic,arrow=0.35,gccol] (-2,0.5) node[left,black] {$k$} -- (-1,-0.5) node[below] {};
\node[text width=1cm] at (-0.4,-0.55){$j$};
\draw[bosonic,arrow=0.35,gccol] (-2,-0.5) node[left,black] {$i$} -- (-1,0.5) node[above] {};
\node[text width=1cm] at (-0.4,0.5){$l$};
\end{tikzpicture}
\end{gathered}=
\begin{cases}
0 & i<j\\
1& k=l=0\\
\dfrac{x}{y} \left(\dfrac{y+\alpha}{x+\alpha}\right)^{1-k} & k=l>0\\
 \left(1-\dfrac{x}{y}\right)  & k=0\\
\dfrac{x}{y}  \left(\dfrac{y+\alpha}{x+\alpha} -1\right) \left(\dfrac{y+\alpha}{x+\alpha}\right)^{-k} & k>0\\
\end{cases}
\end{align}

\begin{align*}
\sum^{a+a'}_{g={a}}
\begin{tikzpicture}[scale=1,baseline=7pt]
\draw[bosonic](1,-1)--(1,2);
\draw[bosonic,arrow=0.35,gccol,rounded corners] (-1,1)--(0,0)--(2,0);
\draw[bosonic,arrow=0.35,gccol,rounded corners] (-1,0)--(0,1)--(2,1);
\node at (-1.2,0) {$ a'$};
\node at (-1.2,1) {$ a$};
\node at (1,-1.3) {$b$};
\node at (1,2.3) {$d$};
\node at (2.2,1) {$c'$};
\node at (2.2,0) {$c$};
\node at (1.2,0.5) {$\ss d+c'-g$};
\node at (0,-0.3) {$\ss a+a'-g$};
\node at (0,1.3) {$g$};
\end{tikzpicture}
&=
\sum^{c+c'}_{g=c}
\begin{tikzpicture}[scale=1,baseline=7pt]
\draw[bosonic] (1,2)--(1,-1);
\draw[bosonic,arrow=0.7,gccol,rounded corners] (0,0)--(2,0)--(3,1);
\draw[bosonic,arrow=0.7,gccol,rounded corners] (0,1)--(2,1)--(3,0);
\node at (-0.2,0) {$a'$};
\node at (-0.2,1) {$a$};
\node at (1,-1.3) {$b$};
\node at (1,2.3) {$d$};
\node at (3.2,0) {$c$};
\node at (3.2,1) {$c'$};
\node at (1.2,0.5) {$\ss a'+b-g$};
\node at (2,-0.3) {$g$};
\node at (2,1.3) {$\ss c+c'-g$};
\end{tikzpicture}
\end{align*}

 Firstly, on $LHS$ the Boltzmann weight of the bottom vertex is fixed based on the relation between $b$ and $c$. Similarly, from the $RHS$, we see that the weight of the top vertex is determined by the relation between $a$ and $d$. So, we need to assume certain relations between $b$ and $c$, and $a$ and $d$. Furthermore, observe that entries of the $\widetilde{r}$ matrix depends on whether the top left node is equal to or greater than $0$.

Therefore, in each subsection we assume some combination of relations between $b$ and $c$, and $a$ and $d$, and a condition on $c'$.

\subsubsection{Assume \texorpdfstring{$a>d$ and $b\geq c$}{a>d and b>=c}.}

As $a>d$, we get that $a>0$. Similarly, from the bottom node of the top vertex on $LHS$, we get that $c'\neq 0$. Therefore, we only need to consider the case where $a>0$ and $c'>0$.

{\allowdisplaybreaks
\begin{align*}
LHS=\text{ }& \br{\dfrac{x}{y}}\br{\dfrac{y+\alpha}{x+\alpha}}^{1-a}\beta(\alpha+
              \beta)^{a-d-1}(y+\alpha)^d x (x+\alpha)^{a'-1}+\\
            & \sum^{a+a'-1}_{g=a+1} \br{\dfrac{x}{y}}\br{\dfrac{y-x}{y+\alpha}}\br{\dfrac{y+\alpha}{x+\alpha}}^{1-a} \beta(\alpha+\beta)^{g-d-1}(y+\alpha)^d x (x+\alpha)^{a'+a-g-1}\\
            &\br{\dfrac{x}{y}}\br{\dfrac{y-x}{y+\alpha}}\br{\dfrac{y+\alpha}{x+\alpha}}^{1-a} \beta(\alpha+\beta)^{a+a'-d-1}(y+\alpha)^d\\
     =\text{ }& \br{\dfrac{x}{y}} \beta (\alpha+\beta)^{a-d-1} (y+\alpha)^{d-a} x (x+\alpha)^{a'+a-1}\br{ \dfrac{y-\beta}{x-\beta}}-\\
             & \br{\dfrac{x}{y}} \br{\dfrac{y-x}{y+\alpha}}\beta (\alpha+\beta)^{a+a'-d-1}(y+\alpha)^{d-a+1} (x+\alpha)^{a-1} \br{\dfrac{\beta}{x-\beta}}\\         
\end{align*}
}

We compute the $RHS$. Observe that we need assume a condition on $c'$. First, let us assume that $c'>0$.

Observe that the node $a'+b-g$ has to be positive. Based on our assumptions and the global condition, we conclude that,
\[
a'+b-(c+c')=d-a<0.
\]
Therefore, the range of $g$ is from $a$ to $a'+b$.

{\allowdisplaybreaks
\begin{align*}
\dfrac{RHS}{\beta (\alpha+\beta)^{a-d-1}(x+\alpha)^{d}}
           =\text{ }& \br{\dfrac{x}{y}}\br{\dfrac{y+\alpha}{x+\alpha}}^{1-c'}y (y+\alpha)^{a'-1}+\\
           &\sum^{b}_{g=c+1} \br{\dfrac{x}{y}}\br{\dfrac{y+\alpha}{x+\alpha}}^{1+g-c'-c} \br{\dfrac{y-x}{y+\alpha}}y (y+\alpha)^{a'-1}+\\
           &\sum^{a'+b}_{g=b+1} \br{\dfrac{x}{y}}\br{\dfrac{y+\alpha}{x+\alpha}}^{1+g-c'-c} \br{\dfrac{y-x}{y+\alpha}}\beta (\alpha+\beta)^{g-b-1} (y+\alpha)^{a'+b-g}\\
    RHS     =\text{ }&\br{\dfrac{x}{y}} \br{\dfrac{y-\beta}{x-\beta}} x(x+\alpha)^{a+a'-1}(y+\alpha)^{d-a}\beta(\alpha+\beta)^{a-d-1}-\\
           &\br{\dfrac{x}{y}}(y-x)\br{\dfrac{\beta^2}{x-\beta}}
           (y+\alpha)^{d-a}(x+\alpha)^{a-1}(\alpha+\beta)^{a+a'-d-1}\\
\end{align*}
}
\subsubsection{Assume \texorpdfstring{$a\leq d$ and $b\geq c$}{a<=d and b>=c}.}

We now assume that $a>0$ and $c'>0$.
{\allowdisplaybreaks
\begin{align*}
LHS=\text{ }& \br{\dfrac{x}{y}}\br{\dfrac{y+\alpha}{x+\alpha}}^{1-a} y (y+\alpha)^{a-1} x (x+\alpha)^{a'-1}+\\
            & \sum^{d}_{g=a+1} \br{\dfrac{x}{y}}\br{\dfrac{y-x}{y+\alpha}}\br{\dfrac{y+\alpha}{x+\alpha}}^{1-a} y(y+\alpha)^{g-1} x (x+\alpha)^{a'+a-g-1}+\\
            & \sum^{a+a'-1}_{g=d+1} \br{\dfrac{x}{y}}\br{\dfrac{y-x}{y+\alpha}}\br{\dfrac{y+\alpha}{x+\alpha}}^{1-a} \beta(\alpha+\beta)^{g-d-1}(y+\alpha)^d x (x+\alpha)^{a'+a-g-1}+\\
            &\br{\dfrac{x}{y}}\br{\dfrac{y-x}{y+\alpha}}\br{\dfrac{y+\alpha}{x+\alpha}}^{1-a} \beta(\alpha+\beta)^{a+a'-d-1}(y+\alpha)^d\\
            =\text{ }& \br{\dfrac{x^3}{y}}\br{\dfrac{y-\beta}{x-\beta}}(y+\alpha)^{d-a}(x+\alpha)^{2a+a'-d-2}-\\
            & \br{\dfrac{x}{y}} \br{\dfrac{y-x}{x-\beta}}\beta (\alpha+\beta)^{a+a'-d-1}(y+\alpha)^{d-a}(x+\alpha)^{a-1}
\end{align*}            
}

{\allowdisplaybreaks
\begin{align*}
\dfrac{RHS}{x(x+\alpha)^{a-1}}
           =\text{ }& \br{\dfrac{x}{y}}\br{\dfrac{y+\alpha}{x+\alpha}}^{1-c'}y (y+\alpha)^{a'-1}+\\
           &\sum^{b}_{g=c+1} \br{\dfrac{x}{y}}\br{\dfrac{y+\alpha}{x+\alpha}}^{1+g-c'-c} \br{\dfrac{y-x}{y+\alpha}}y (y+\alpha)^{a'-1}+\\
           &\sum^{c+c'-1}_{g=b+1} \br{\dfrac{x}{y}}\br{\dfrac{y+\alpha}{x+\alpha}}^{1+g-c'-c} \br{\dfrac{y-x}{y+\alpha}}\beta (\alpha+\beta)^{g-b-1} (y+\alpha)^{a'+b-g}+\\
           &\br{1-\dfrac{x}{y}}\beta (\alpha+\beta)^{c+c'-b-1} (y+\alpha)^{a'+b-c-c'}\\
    RHS   =\text{ }& \br{\dfrac{x^3}{y}}\br{\dfrac{y-\beta}{x-\beta}}(y+\alpha)^{d-a}(x+\alpha)^{2a+a'-d-2}-\\
          & \br{\dfrac{x}{y}}\br{\dfrac{y-x}{x-\beta}}\beta(\alpha+\beta)^{a+a'-d-1}(y+\alpha)^{d-a}(x+\alpha)^{a-1}\\
\end{align*}
}

\subsubsubsection{Assume \texorpdfstring{$a=0$ and $c'>0$}.}
{\allowdisplaybreaks
\begin{align*}
LHS=\text{ }& x(x+\alpha)^{a'-1}+\sum^{d}_{g=1} \br{1-\dfrac{x}{y}} y(y+\alpha)^{g-1}x(x+\alpha)^{a'-g-1}+\\
            &\sum^{a'-1}_{g=d+1}\br{1-\dfrac{x}{y}}\beta(\alpha+\beta)^{g-d-1}(y+\alpha)^{d}x(x+\alpha)^{a'-g-1}+\\
            &\br{1-\dfrac{x}{y}}\beta(\alpha+\beta)^{a'-d-1}(y+\alpha)^{d}\\
            =\text{ }& x(x+\alpha)^{a'-d-1}(y+\alpha)^{d}\br{\dfrac{x}{y}}\br{\dfrac{y-\beta}{x-\beta}}-\\
            &\br{1-\dfrac{x}{y}}\beta (y+\alpha)^{d}(\alpha+\beta)^{a'-d-1}\br{\dfrac{\beta}{x-\beta}}
\end{align*}
}
 Observe that, while computing $RHS$ in the case where $a\leq d$, its only in the final step we multiply the Boltzmann weight of the top vertex. Here, the Boltzmann weight if the top vertex is $1$. Therefore,
{\allowdisplaybreaks
\begin{align*}
RHS=\text{ }&  \br{\dfrac{x}{y}}\br{\dfrac{y+\alpha}{x+\alpha}}^{1+b-c-c'}(y+\alpha)^{a'-1}\br{\dfrac{(y-\beta)x}{x-\beta}}-\\
           & \beta(\alpha+\beta)^{c+c'-b-1}(y+\alpha)^{d}\br{\dfrac{\beta(y-x)}{y(x-\beta)}}\\
           =\text{ }& \br{\dfrac{x^2}{y}}\br{\dfrac{y-\beta}{x-\beta}}(y+\alpha)^{d}(x+\alpha)^{a'-d-1}-\\
           & \br{1-\dfrac{x}{y}}\br{\dfrac{\beta}{x-\beta}}\beta(y+\alpha)^{d}(\alpha+\beta)^{a'-d-1}
\end{align*}
}

In the above computation, we have assumed $a'>d$. We now consider the case where $a'\leq d$.

{\allowdisplaybreaks
\begin{align*}
LHS=\text{ }& x(x+\alpha)^{a'-1}+\sum^{a'-1}_{g=1} \br{1-\dfrac{x}{y}} y(y+\alpha)^{g-1}x(x+\alpha)^{a'-g-1}+\\
            &\br{1-\dfrac{x}{y}}y(y+\alpha)^{a'-1}\\
            =\text{ }&y(y+\alpha)^{a'-1}\\
\end{align*}
}

On the $RHS$, the weights of the vertices are fixed for all $g$. Since the right boundary of the cross is fixed, because of the eigenvector property, we have  $RHS=y(y+\alpha)^{a'-1}$.

\subsubsection{Assume \texorpdfstring{$a>d$ and $b<c$}{a>d and b<c}.}

As $a>d$, we will have $a>0$. Also $c'>0$, otherwise we get contradiction on the range of $g$.
{\allowdisplaybreaks
\begin{align*}
LHS=\text{ }&\br{\dfrac{x}{y}}\br{\dfrac{y+\alpha}{x+\alpha}}^{1-a}
            \beta(\alpha+\beta)^{a-d-1}(y+\alpha)^{d}\beta(\alpha+
            \beta)^{a'+a-d-c'-1}(x+\alpha)^{d+c'-a}+\\
            &\sum^{d+c'}_{g=a+1}\br{\dfrac{x}{y}}\br{\dfrac{y+\alpha}{x+\alpha}}^{1-a}\br{\dfrac{y-x}{y+\alpha}} \beta (\alpha+\beta)^{g-d-1}(y+\alpha)^{d}\beta (\alpha+\beta)^{a+a'-d-c'-1}(x+\alpha)^{d+c'-g}\\
            =\text{  }&\br{\dfrac{x}{y}}\beta^{2}(\alpha+\beta)^{a+c-b-d-2}(y+\alpha)^{d+1-a}(x+\alpha)^{d+c'-1}+\\
            &\br{\dfrac{x}{y}}\br{\dfrac{y-x}{y+\alpha}}\beta^2 (\alpha+\beta)^{a+b-c-d-1}(y+\alpha)^{d+1-a}(x+\alpha)^{d+c'-1}\br{\dfrac{1-\br{\dfrac{\alpha+\beta}{x+\alpha}}^{d+c'-a}}{x-\beta}}\\
\end{align*}}
We compute $RHS$:
{\allowdisplaybreaks
\begin{align*}
RHS=\text{ }& \br{\dfrac{x}{y}}\br{\dfrac{y+\alpha}{x+\alpha}}^{1-c'} \beta(\alpha+\beta)^{a-d-1}(x+\alpha)^{d}\beta (\alpha+\beta)^{c-b-1}(y+\alpha)^{a'+b-c}+\\
             &\sum^{a'+b}_{g=c+1} \br{\dfrac{x}{y}}\br{\dfrac{y-x}{y+\alpha}}\br{\dfrac{y+\alpha}{x+\alpha}}^{1+g-c-c'}\br{\beta(\alpha+\beta)^{a-d-1}(x+\alpha)^{d}}\\
             &\br{\beta (\alpha+\beta)^{g-b-1}(y+\alpha)^{a'+b-g}}\\
             =\text{ }&\br{\dfrac{x}{y}}\beta^2 (\alpha+\beta)^{a+c-b-d-2}(y+\alpha)^{d-a+1}(x+\alpha)^{d+c'-1}+\\
             & \br{\dfrac{x}{y}}\br{\dfrac{y-x}{y+\alpha}}\beta^2 (\alpha+\beta)^{a+c-b-d-1}(y+\alpha)^{d-a+1}(x+\alpha)^{d+c'-1}
             \br{\dfrac{1-\br{\dfrac{\alpha+\beta}{x+\alpha}}^{d+c'-a}}{x-\beta}}\\
\end{align*}
}

\subsubsection{Assume \texorpdfstring{$a\leq d$ and $b<c$}{a<=d and b<c}.}

\subsubsubsection{Assume \texorpdfstring{$a>0$ and $c'>0$}{a>0 and c'>0}}

 Recall that we have the global condition $a+a'+b=c+c'+d$. Observe that, because of the assumptions, the range of $g$ on the $LHS$ is $a$ to $d+c'$.

On the $RHS$, we have $a'+b-c-c'=d-a\geq 0$. Therefore, the range of $g$ is from $c$ to $c+c'$

{\allowdisplaybreaks
\begin{align*}
LHS=\text{ }& \br{\dfrac{x}{y}}\br{\dfrac{y+\alpha}{x+\alpha}}^{1-a}y(y+\alpha)^{a-1}\beta (\alpha+\beta)^{c-b-1}(x+\alpha)^{d+c'-a}+\\
            & \sum^{d}_{g=a+1}\br{\dfrac{x}{y}}\br{\dfrac{y-x}{y+\alpha}}\br{\dfrac{y+\alpha}{x+\alpha}}^{1-a}y(y+\alpha)^{g-1}\beta (\alpha+\beta)^{c-b-1}(x+\alpha)^{d+c'-g}+\\
            &\sum^{d+c'}_{g=d+1} \br{\dfrac{x}{y}}\br{\dfrac{y-x}{y+\alpha}}\br{\dfrac{y+\alpha}{x+\alpha}}^{1-a} \beta (\alpha+\beta)^{g-d-1}(y+\alpha)^{d}\beta (\alpha+\beta)^{c-b-1}(x+\alpha)^{d+c'-g}\\
             =\text{ }& \br{\dfrac{x^2}{y}}\beta (\alpha+\beta)^{c-b-1}(y+\alpha)^{d-a}(x+\alpha)^{a+c'-1}\br{\dfrac{y-\beta}{x-\beta}}-\\
            &\br{\dfrac{x}{y}}\beta^2 \br{\dfrac{y-x}{x-\beta}} (\alpha+\beta)^{c+c'-b-1}(y+\alpha)^{d-a}(x+\alpha)^{a-1}\\
\end{align*}
}

{\allowdisplaybreaks
\begin{align*}
\dfrac{RHS}{x(x+\alpha)^{a-1}}
           =\text{ }& \br{\dfrac{x}{y}}\br{\dfrac{y+\alpha}{x+\alpha}}^{1-c'} \beta (\alpha+\beta)^{c-b-1} (y+\alpha)^{a'+b-c}+\\
           &\sum^{c+c'-1}_{g=c+1} \br{\dfrac{x}{y}}\br{\dfrac{y+\alpha}{x+\alpha}}^{1+g-c'-c} \br{\dfrac{y-x}{y+\alpha}}\beta (\alpha+\beta)^{g-b-1} (y+\alpha)^{a'+b-g}+\\
           &\br{1-\dfrac{x}{y}}\beta (\alpha+\beta)^{c+c'-b-1} (y+\alpha)^{a'+b-c-c'}\\
         RHS  =\text{ }&\br{\dfrac{x^2}{y}}\beta (\alpha+\beta)^{c-b-1}(x+\alpha)^{a+c'-1}(y+\alpha)^{a'+b-c-c'}\br{\dfrac{y-\beta}{x-\beta}}-\\
           &\br{\dfrac{x}{y}}\beta^2(\alpha+\beta)^{c+c'-b-1}(y+\alpha)^{a'+b-c-c'}(x+\alpha)^{a-1}\br{\dfrac{y-x}{x-\beta}}
\end{align*}
}
           
\subsubsubsection{Assume $a=0$ and $c'>0$}

{\allowdisplaybreaks
\begin{align*}
LHS=\text{ }& \beta (\alpha+\beta)^{c-b-1}(x+\alpha)^{d+c'}+\\
            & \sum^{d}_{g=1} \br{1-\dfrac{x}{y}} y(y+\alpha)^{g-1}\beta (\alpha+\beta)^{c-b-1}(x+\alpha)^{d+c'-g}+\\
            & \sum^{d+c'}_{g=d+1} \br{1-\dfrac{x}{y}} \beta (\alpha+\beta)^{g-d-1}(y+\alpha)^{d}\beta (\alpha+\beta)^{c-b-1}(x+\alpha)^{d+c'-g}\\
             =\text{ }& \br{\dfrac{x}{y}}\beta (\alpha+\beta)^{c-b-1}(y+\alpha)^{d}(x+\alpha)^{c'}\br{\dfrac{y-\beta}{x-\beta}}-\\
             &\br{1-\dfrac{x}{y}}\br{\dfrac{\beta^2}{x-\beta}} (\alpha+\beta)^{c'+c-b-1}(y+\alpha)^{d}\\
\end{align*}}

Observe when $a=0$, the only difference in the $RHS$ from the earlier case is the weight of the Boltzmann weight of the top vertex.
{\allowdisplaybreaks
\begin{align*}
RHS= \text{ }&\br{\dfrac{x}{y}}\beta (\alpha+\beta)^{c-b-1}(y+\alpha)^{d}(x+\alpha)^{c'}\br{\dfrac{y-\beta}{x-\beta}}-\\
             &\br{1-\dfrac{x}{y}}\br{\dfrac{\beta^{2}}{x-\beta}}(\alpha+\beta)^{c+c'-b-1}(y+\alpha)^{d}\\
\end{align*}
}

\subsubsubsection{Assume $a=0$ and $c'=0$}

{\allowdisplaybreaks
\begin{align*}
LHS=\text{ }& \beta (\alpha+\beta)^{c-b-1}(x+\alpha)^{d+c'}+\\
            & \sum^{d}_{g=1} \br{1-\dfrac{x}{y}} y(y+\alpha)^{g-1}\beta (\alpha+\beta)^{c-b-1}(x+\alpha)^{d-g}+\\
            =\text{ }& \beta(\alpha+\beta)^{c-b-1}(y+\alpha)^{d}
\end{align*}
}

On the $RHS$, there is a unique configuration.
\[
\begin{aligned}
RHS=\text{ }& \beta (\alpha+\beta)^{c-b-1}(y+\alpha)^{a'+b-c}\\
=\text{ }& \beta (\alpha+\beta)^{c-b-1}(y+\alpha)^{d}\\
\end{aligned}
\qquad
\br{
\begin{tikzpicture}[scale=0.8,baseline=7pt]
\draw[bosonic] (1,2)--(1,-1);
\draw[bosonic,arrow=0.7,gccol,rounded corners] (0,0)--(2,0)--(3,1);
\draw[bosonic,arrow=0.7,gccol,rounded corners] (0,1)--(2,1)--(3,0);
\node at (-0.2,0) {$\ss a'$};
\node at (-0.2,1) {$\ss 0$};
\node at (1,-1.3) {$\ss b$};
\node at (1,2.3) {$\ss d$};
\node at (3.2,0) {$\ss c$};
\node at (3.2,1) {$\ss 0$};
\node at (1.6,0.5) {$\ss a'+b-c$};
\node at (2,-0.3) {$\ss c$};
\node at (2,1.3) {$\ss 0$};
\end{tikzpicture}}\\
\]

\subsubsubsection{Assume $a>0$ and $c'=0$}

{\allowdisplaybreaks
\begin{align*}
LHS=\text{ }& \br{\dfrac{x}{y}}\br{\dfrac{y+\alpha}{x+\alpha}}^{1-a}y(y+\alpha)^{a-1}\beta (\alpha+\beta)^{c-b-1}(x+\alpha)^{d+c'-a}+\\
            & \sum^{d}_{g=a+1}\br{\dfrac{x}{y}}\br{\dfrac{y-x}{y+\alpha}}\br{\dfrac{y+\alpha}{x+\alpha}}^{1-a}y(y+\alpha)^{g-1}\beta (\alpha+\beta)^{c-b-1}(x+\alpha)^{d+c'-g}+\\
            =\text{ }&\beta (\alpha+\beta)^{c-b-1}(y+\alpha)^{d-a}x(x+\alpha)^{a-1}\\
\end{align*}
}

Just like in the previous case, we have a unique configuration.
\begin{align*}
RHS=\text{ }& x(x+\alpha)^{a-1} \beta (\alpha+\beta)^{c-b-1}(y+\alpha)^{a'+b-c}\\
  =\text{ }& x(x+\alpha)^{a-1} \beta (\alpha+\beta)^{c-b-1}(y+\alpha)^{d-a}\\
\end{align*}

\subsection{\texorpdfstring{$\mathrm{RLL}$}{RLL} for the Cauchy Identity.}
\label{rllforcauchy}
We prove a relation between $L^*$, the dual $L$ matrix of row model of $G^{(-\alpha,-\beta)}_{\lambda}$, and $l$, from the row model of $\g$.

For convenience let us recall all the characters of the play.

\begin{align*}
\begin{tabular}{ccccc}
\begin{tikzpicture}[scale=0.4,baseline=-2pt]
\draw[fermionic,Gcol,arrow=0.25]  (-1,0) node[left,black] {$\ss {1}$}--(1,0) node[right,black] {$\ss {1}$};
\draw[bosonic,arrow=0.25] (0,-1) node[below] {$\ss {0}$}--(0,1) node[above] {$\ss {0}$};
\node[text width=1cm] at (1,-3.5){1};
\end{tikzpicture}
\qquad
\begin{tikzpicture}[scale=0.4,baseline=-2pt]
\draw[fermionic,Gcol,arrow=0.25] (-1,0) node[left,black] {$\ss {1}$}--(1,0) node[right,black] {$\ss {1}$};
\draw[bosonic,arrow=0.25] (0,-1) node[below] {$\ss {m}$}--(0,1) node[above] {$\ss {m}$};
\node[text width=1cm] at (0.4,-3.5){$ {\frac{1-\beta x}{1+\alpha x}}$};
\end{tikzpicture}
\qquad
\begin{tikzpicture}[scale=0.4,baseline=-2pt]
\draw[fermionic,Gcol,arrow=0.25] (-1,0) node[left,black] {$\ss {1}$} -- (1,0) node[right,black] {$\ss {0}$};
\draw[bosonic,arrow=0.25]  (0,-1) node[below] {$\ss {m-1}$}--(0,1) node[above] {$\ss {m}$};
\node[text width=1cm] at (0.4,-3.5){$ {\frac{1-\beta x}{1+\alpha x}}$};
\end{tikzpicture}
\qquad
\begin{tikzpicture}[scale=0.4,baseline=-2pt]
\draw[fermionic,Gcol,arrow=0.25](-1,0) node[left,black] {$\ss {0}$}-- (1,0) node[right,black] {$\ss {1}$};
\draw[bosonic,arrow=0.25] (0,-1) node[below] {$\ss {m+1}$}--(0,1) node[above] {$\ss {m}$};
\node[text width=1cm] at (0.4,-3.5){$\frac{x}{1+\alpha x}$};
\end{tikzpicture}
\qquad
\begin{tikzpicture}[scale=0.4,baseline=-2pt]
\draw[fermionic,Gcol,arrow=0.25]  (-1,0) node[left,black] {$\ss {0}$}--(1,0) node[right,black] {$\ss {0}$} ;
\draw[bosonic,arrow=0.25]  (0,-1) node[below] {$\ss {m}$}--(0,1) node[above] {$\ss {m}$};
\node[text width=1cm] at (0.4,-3.5){$\frac{x}{1+\alpha x}$};
\end{tikzpicture}\\
\end{tabular}
\end{align*}

\begin{align}
w^{}_{x}
\left(
\begin{gathered}
\begin{tikzpicture}[scale=0.4,baseline=-2pt]
\draw[bosonic,gcol,arrow=0.25] (-1,0) node[left,black] {$a$} -- (1,0) node[right,black] {$c$};
\draw[bosonic,arrow=0.25] (0,-1) node[below] {$b$} -- (0,1) node[above] {$d$};
\end{tikzpicture}
\end{gathered}
\right)
\equiv
w^{}_{x}(a,b;c,d)
= \delta_{a+b,c+d} 
\begin{cases}
(\alpha+\beta)^{a-d-1} (x+\alpha) \beta^{d}& a>d\\
\beta^{a-1} x & 0<a\leq d\\
1& a=0,
\end{cases}
\end{align}
where 
$a,b,c,d$ $\in \mathbb{Z}_{\geq 0}$.

The $\mathfrak {R}$ matrix $\in$ {\rm End}$(F \otimes W)$, 

\begin{align}
\label{RforTandt}
\mathfrak{R}^{k,l}_{i,j}=
\begin{gathered}
\begin{tikzpicture}[baseline=-3pt,scale=0.9]
\draw[fermionic,arrow=0.35,Gcol] (-2,0.5) node[left,black] {$k$} -- (-1,-0.5) node[below] {};
\node[text width=1cm] at (-0.4,-0.55){$j$};
\draw[bosonic,arrow=0.35,gcol] (-2,-0.5) node[left,black] {$i$} -- (-1,0.5) node[above] {};
\node[text width=1cm] at (-0.4,0.5){$l$};
\end{tikzpicture}
\end{gathered}=
\begin{cases}
1-x y & j=k=1,i=l=0\\
xy & k=l=0,i=j=1\\
1-x \beta & k=1\\
x\beta &k=0\\
1& i=k=l=j=0\\
\end{cases}
\end{align}
where $k,j\in \{ 0,1\}$ and $i,l \in \ZZ_{\geq 0}$,
\begin{align}
\begin{tabular}{cccccc}
\begin{tikzpicture}[baseline=-3pt,scale=0.9]
\draw[fermionic,arrow=0.35,Gcol] (-2,0.5) node[left,black] {$\ss 0$} -- (-1,-0.5) node[below] {};
\node[text width=1cm] at (-0.4,-0.55){$\ss 0$};
\draw[bosonic,arrow=0.35,gcol] (-2,-0.5) node[left,black] {$\ss 0$} -- (-1,0.5) node[above] {};
\node[text width=1cm] at (-0.4,0.5){$\ss 0$};
\node at (-1.5,-1.5) {$ 1$};
\end{tikzpicture}
\quad
\begin{tikzpicture}[baseline=-3pt,scale=0.9]
\draw[fermionic,arrow=0.35,Gcol] (-2,0.5) node[left,black] {$\ss 1$} -- (-1,-0.5) node[below] {};
\node[text width=1cm] at (-0.4,-0.55){$\ss 1$};
\draw[bosonic,arrow=0.35,gcol] (-2,-0.5) node[left,black] {$\ss 0$} -- (-1,0.5) node[above] {};
\node[text width=1cm] at (-0.4,0.5){$\ss 0$};
\node at (-1.5,-1.5) {$ 1-xy$};
\end{tikzpicture}
\quad
\begin{tikzpicture}[baseline=-3pt,scale=0.9]
\draw[fermionic,arrow=0.35,Gcol] (-2,0.5) node[left,black] {$\ss 0$} -- (-1,-0.5) node[below] {};
\node[text width=1cm] at (-0.4,-0.55){$\ss 1$};
\draw[bosonic,arrow=0.35,gcol] (-2,-0.5) node[left,black] {$\ss 1$} -- (-1,0.5) node[above] {};
\node[text width=1cm] at (-0.4,0.5){$\ss 0$};
\node at (-1.5,-1.5) {$ xy$};
\end{tikzpicture}
\quad
\begin{tikzpicture}[baseline=-3pt,scale=0.9]
\draw[fermionic,arrow=0.35,Gcol] (-2,0.5) node[left,black] {$1$} -- (-1,-0.5) node[below] {};
\node[text width=1cm] at (-0.4,-0.55){$j$};
\draw[bosonic,arrow=0.35,gcol] (-2,-0.5) node[left,black] {$i$} -- (-1,0.5) node[above] {};
\node[text width=1cm] at (-0.4,0.5){$l$};
\node at (-1.5,-1.5) {$1- x\beta$};
\end{tikzpicture}
\quad
\begin{tikzpicture}[baseline=-3pt,scale=0.9]
\draw[fermionic,arrow=0.35,Gcol] (-2,0.5) node[left,black] {$0$} -- (-1,-0.5) node[below] {};
\node[text width=1cm] at (-0.4,-0.55){$j$};
\draw[bosonic,arrow=0.35,gcol] (-2,-0.5) node[left,black] {$i$} -- (-1,0.5) node[above] {};
\node[text width=1cm] at (-0.4,0.5){$l$};
\node at (-1.5,-1.5) {$x\beta$};
\end{tikzpicture}
\end{tabular}
\end{align}

together with $L^*$ and $l$ satisfy $\rm RLL$ relation. We shall prove the following equation:
\begin{align*}
&\begin{tikzpicture}[scale=0.8,baseline=7pt]
\draw[bosonic](1,-1)--(1,2);
\draw[fermionic,arrow=0.35,Gcol,rounded corners] (-1,1)--(0,0)--(2,0);
\draw[bosonic,arrow=0.35,gcol,rounded corners] (-1,0)--(0,1)--(2,1);
\node at (-1.2,0) {$\ss a'$};
\node at (-1.2,1) {$\ss a$};
\node at (1,-1.3) {$\ss b$};
\node at (1,2.3) {$\ss d$};
\node at (2.2,1) {$\ss c'$};
\node at (2.2,0) {$\ss c$};
\node at (1.4,0.5) {$\ss b-c$};
\node at (0,-0.3) {$\ss 0$};
\node at (0,1.3) {$\ss a+a'$};
\end{tikzpicture}+
\begin{tikzpicture}[scale=0.8,baseline=7pt]
\draw[bosonic](1,-1)--(1,2);
\draw[fermionic,arrow=0.35,Gcol,rounded corners] (-1,1)--(0,0)--(2,0);
\draw[bosonic,arrow=0.35,gcol,rounded corners] (-1,0)--(0,1)--(2,1);
\node at (-1.2,0) {$\ss a'$};
\node at (-1.2,1) {$\ss a$};
\node at (1,-1.3) {$\ss b$};
\node at (1,2.3) {$\ss d$};
\node at (2.2,1) {$\ss c'$};
\node at (2.2,0) {$\ss c$};
\node at (1.6,0.5) {$\ss b+1-c$};
\node at (0,-0.3) {$\ss 1$};
\node at (0,1.3) {$\ss a+a'-1$};
\end{tikzpicture}=
\begin{tikzpicture}[scale=0.8,baseline=7pt]
\draw[bosonic] (1,2)--(1,-1);
\draw[bosonic,arrow=0.7,gcol,rounded corners] (0,0)--(2,0)--(3,1);
\draw[fermionic,arrow=0.7,Gcol,rounded corners] (0,1)--(2,1)--(3,0);
\node at (-0.2,0) {$\ss a'$};
\node at (-0.2,1) {$\ss a$};
\node at (1,-1.3) {$\ss b$};
\node at (1,2.3) {$\ss d$};
\node at (3.2,0) {$\ss c$};
\node at (3.2,1) {$\ss c'$};
\node at (1.6,0.5) {$\ss d-a$};
\node at (2,-0.3) {$\ss c+c'$};
\node at (2,1.3) {$\ss 0$};
\end{tikzpicture}+
\begin{tikzpicture}[scale=0.8,baseline=7pt]
\draw[bosonic] (1,2)--(1,-1);
\draw[bosonic,arrow=0.7,gcol,rounded corners] (0,0)--(2,0)--(3,1);
\draw[fermionic,arrow=0.7,Gcol,rounded corners] (0,1)--(2,1)--(3,0);
\node at (-0.2,0) {$\ss a'$};
\node at (-0.2,1) {$\ss a$};
\node at (1,-1.3) {$\ss b$};
\node at (1,2.3) {$\ss d$};
\node at (3.2,0) {$\ss c$};
\node at (3.2,1) {$\ss c'$};
\node at (1.6,0.5) {$\ss d+1-a$};
\node at (2,-0.3) {$\ss c+c'-1$};
\node at (2,1.3) {$\ss 1$};
\end{tikzpicture}
\end{align*}

Based on the entries of the $R$- matrix, we shall assume certain condition on $a,a'$ and $c,c'$. We shall divide the conditions of $a,a'$ into three cases, $a+a'=0$,$a+a'=1$ and $a+a'>1$. We do the same with the conditions on $c,c'$. Therefore, we shall have a total of $9$ cases to consider.

\subsubsection{Assume $a+a'=0$ and $c+c'=0$.}
Observe that, because of the global condition, $b$ is equal to $d$. So, each side will have a unique configuration with identical weight.

\subsubsection{Assume $a+a'=0$.}

When $a+a'=0$, there is a unique configuration on $LHS$. On the RHS, because of the eigenvector property of the $\mathfrak{R}$ matrix, the weight is just the product of the Boltzmann weights.

\subsubsection{Assume $a=0$ and $a'=1$ and $c+c'=0$.}
\[
\begin{aligned}
LHS=\text{ }& \br{x \beta} y \br{\dfrac{ x}{1+\alpha x}}+\br{x y}\br{\dfrac{1-\beta x}{1+\alpha x}}\\ 
=\text{ }& \dfrac{xy}{1+\alpha x}
\end{aligned}
\qquad
\br{
\begin{tikzpicture}[scale=0.7,baseline=7pt]
\draw[bosonic](1,-1) node[below,black] {$\ss b$}--(1,0.5) node[right,black] {$\ss b$}--(1,2) node[black,above] {$\ss b+1$};
\draw[fermionic,arrow=0.35,Gcol,rounded corners] (-1,1) node[left,black] {$\ss 0$}--(0,0) node[below,black] {$\ss 0$}--(2,0) node[right,black]{$\ss 0$};
\draw[bosonic,arrow=0.35,gcol,rounded corners] (-1,0) node[left,black] {$\ss 1$}--(0,1) node[above,black] {$\ss 1$}--(2,1)node[right,black]{$\ss 0$};
\end{tikzpicture}+
\begin{tikzpicture}[scale=0.7,baseline=7pt]
\draw[bosonic](1,-1) node[below,black] {$\ss b$}--(1,0.5) node[right,black] {$\ss b+1$}--(1,2) node[black,above] {$\ss b+1$};
\draw[fermionic,arrow=0.35,Gcol,rounded corners] (-1,1) node[left,black] {$\ss 0$}--(0,0) node[below,black] {$\ss 1$}--(2,0) node[right,black]{$\ss 0$};
\draw[bosonic,arrow=0.35,gcol,rounded corners] (-1,0) node[left,black] {$\ss 1$}--(0,1) node[above,black] {$\ss 0$}--(2,1)node[right,black]{$\ss 0$};
\end{tikzpicture}}
\]

\[
RHS=\br{\dfrac{x}{1+\alpha x}} \br{y}
\qquad 
\br{
\begin{tikzpicture}[scale=0.7,baseline=7pt]
\draw[bosonic] (1,2) node[above,black]{$\ss  {b+1}$}--(1,0.5)node[right,black]{$\ss b+1$}--(1,-1) node[below,black]{$\ss b$};
\draw[bosonic,arrow=0.7,gcol,rounded corners] (0,0) node[left,black] {$\ss 1$}--(2,0) node[below,black] {$\ss 0$}--(3,1)node[right,black] {$\ss 0$};
\draw[fermionic,arrow=0.7,Gcol,rounded corners] (0,1) node[left,black]{$\ss 0$}--(2,1) node[above,black] {$\ss 0$}--(3,0) node[right,black] {$\ss 0$};
\end{tikzpicture}
}
\]

\subsubsection{Assume $a=0$ and $a'=1$ and $c'=0$ and $c=1$.}

When $b>0$,
\[
\begin{aligned}
LHS=\text{ }& \br{x \beta} y \br{\dfrac{x}{1+\alpha x}}+ x y\br{\dfrac{1-\beta x}{1+\alpha x}}\\
=\text{ }&\dfrac{xy}{1+\alpha x}
\end{aligned}
\qquad
\br{
\begin{tikzpicture}[scale=0.7,baseline=7pt]
\draw[bosonic](1,-1) node[below,black] {$\ss b$}--(1,0.5) node[right,black] {$\ss b-1$}--(1,2) node[black,above] {$\ss b$};
\draw[fermionic,arrow=0.35,Gcol,rounded corners] (-1,1) node[left,black] {$\ss 0$}--(0,0) node[below,black] {$\ss 0$}--(2,0) node[right,black]{$\ss 1$};
\draw[bosonic,arrow=0.35,gcol,rounded corners] (-1,0) node[left,black] {$\ss 1$}--(0,1) node[above,black] {$\ss 1$}--(2,1)node[right,black]{$\ss 0$};
\end{tikzpicture}+
\begin{tikzpicture}[scale=0.7,baseline=7pt]
\draw[bosonic](1,-1) node[below,black] {$\ss b$}--(1,0.5) node[right,black] {$\ss b$}--(1,2) node[black,above] {$\ss b$};
\draw[fermionic,arrow=0.35,Gcol,rounded corners] (-1,1) node[left,black] {$\ss 0$}--(0,0) node[below,black] {$\ss 1$}--(2,0) node[right,black]{$\ss 1$};
\draw[bosonic,arrow=0.35,gcol,rounded corners] (-1,0) node[left,black] {$\ss 1$}--(0,1) node[above,black] {$\ss 0$}--(2,1)node[right,black]{$\ss 0$};
\end{tikzpicture}}
\]

When $b=0$, only the second configuration from the left is a valid configuration.
\begin{align*}
LHS=\text{ }& x y\\
\end{align*}

When $b>0$, we have:
\[
\begin{aligned}
RHS=\text{ }& \br{\dfrac{x}{1+\alpha x}}y \br{1-xy}+ \br{\dfrac{x}{1+\alpha x}}y\br{xy}\\
=\text{ }\dfrac{xy}{1+\alpha x}
\end{aligned}
\br{
\begin{tikzpicture}[scale=0.7,baseline=7pt]
\draw[bosonic] (1,2) node[above,black]{$\ss  {b}$}--(1,0.5)node[right,black]{$\ss b+1$}--(1,-1) node[below,black]{$\ss b$};
\draw[bosonic,arrow=0.7,gcol,rounded corners] (0,0) node[left,black] {$\ss 1$}--(2,0) node[below,black] {$\ss 0$}--(3,1)node[right,black] {$\ss 0$};
\draw[fermionic,arrow=0.7,Gcol,rounded corners] (0,1) node[left,black]{$\ss 0$}--(2,1) node[above,black] {$\ss 1$}--(3,0) node[right,black] {$\ss 1$};
\end{tikzpicture}+
\begin{tikzpicture}[scale=0.7,baseline=7pt]
\draw[bosonic] (1,2) node[above,black]{$\ss   {b}$}--(1,0.5)node[right,black]{$\ss b$}--(1,-1) node[below,black]{$\ss b$};
\draw[bosonic,arrow=0.7,gcol,rounded corners] (0,0) node[left,black] {$\ss 1$}--(2,0) node[below,black] {$\ss 1$}--(3,1)node[right,black] {$\ss 0$};
\draw[fermionic,arrow=0.7,Gcol,rounded corners] (0,1) node[left,black]{$\ss 0$}--(2,1) node[above,black] {$\ss 0$}--(3,0) node[right,black] {$\ss 1$};
\end{tikzpicture}
}
\]

When $b=0$,
\begin{align*}
RHS=\text{ }& \br{\dfrac{x}{1+\alpha x}}y \br{1-xy}+\dfrac{x}{1+\alpha x}\br{y+\alpha}(xy)\\
=\text{ }& \br{\dfrac{xy}{1+\alpha x}}\br{1-xy +xy+\alpha x}\\
=\text{ }& xy\\
\end{align*}

\subsubsection{Assume $a=0$ and $a'=1$ and $c+c'\geq 1$ and $c'\neq 0$.}

\[
\begin{aligned}
LHS=\text{ }& \br{x \beta} y \br{\dfrac{x}{1+\alpha x}}+ xy\br{\dfrac{1-\beta x}{1+\alpha x}}\\
=\text{ }& \dfrac{xy}{1+\alpha x}
\end{aligned}
\qquad \br{
\begin{tikzpicture}[scale=0.7,baseline=7pt]
\draw[bosonic](1,-1) node[below,black] {$\ss b$}--(1,0.5) node[right,black] {$\ss b-c$}--(1,2) node[black,above] {$\ss b-c-c'+1$};
\draw[fermionic,arrow=0.35,Gcol,rounded corners] (-1,1) node[left,black] {$\ss 0$}--(0,0) node[below,black] {$\ss 0$}--(2,0) node[right,black]{$\ss c$};
\draw[bosonic,arrow=0.35,gcol,rounded corners] (-1,0) node[left,black] {$\ss 1$}--(0,1) node[above,black] {$\ss 1$}--(2,1)node[right,black]{$\ss c'$};
\end{tikzpicture}+
\begin{tikzpicture}[scale=0.7,baseline=7pt]
\draw[bosonic](1,-1) node[below,black] {$\ss b$}--(1,0.5) node[right,black] {$\ss b-c+1$}--(1,2) node[black,above] {$\ss b-c-c'+1$};
\draw[fermionic,arrow=0.35,Gcol,rounded corners] (-1,1) node[left,black] {$\ss 0$}--(0,0) node[below,black] {$\ss 1$}--(2,0) node[right,black]{$\ss c$};
\draw[bosonic,arrow=0.35,gcol,rounded corners] (-1,0) node[left,black] {$\ss 1$}--(0,1) node[above,black] {$\ss 0$}--(2,1)node[right,black]{$\ss c'$};
\end{tikzpicture}}
\]

\[
\begin{aligned}
RHS=\text{ }& \br{\dfrac{x}{1+\alpha x}}y\br{x \beta}+\br{\dfrac{x}{1+\alpha x}}y\br{1-x\beta}\\
=\text{ }& \dfrac{xy}{1+\alpha x}\\
\end{aligned}
\br{
\begin{tikzpicture}[scale=0.7,baseline=7pt]
\draw[bosonic] (1,2) node[above,black]{$\ss   {b-c-c'+1}$}--(1,0.5)node[right,black]{$\ss b-c-c'+1$}--(1,-1) node[below,black]{$\ss b$};
\draw[bosonic,arrow=0.7,gcol,rounded corners] (0,0) node[left,black] {$\ss 1$}--(2,0) node[below,black] {$\ss c+c'$}--(3,1)node[right,black] {$\ss c'$};
\draw[fermionic,arrow=0.7,Gcol,rounded corners] (0,1) node[left,black]{$\ss 0$}--(2,1) node[above,black] {$\ss 0$}--(3,0) node[right,black] {$\ss c$};
\end{tikzpicture}+
\begin{tikzpicture}[scale=0.7,baseline=7pt]
\draw[bosonic] (1,2) node[above,black]{$\ss  {b-c-c'+1}$}--(1,0.5)node[right,black]{$\ss b-c-c'+2$}--(1,-1) node[below,black]{$\ss b$};
\draw[bosonic,arrow=0.7,gcol,rounded corners] (0,0) node[left,black] {$\ss 1$}--(2,0) node[below,black] {$\ss c+c'-1$}--(3,1)node[right,black] {$\ss c'$};
\draw[fermionic,arrow=0.7,Gcol,rounded corners] (0,1) node[left,black]{$\ss 0$}--(2,1) node[above,black] {$\ss 1$}--(3,0) node[right,black] {$\ss c$};
\end{tikzpicture}
}
\]

\subsubsection{Assume $a=1$ and $a'=0$ and $c+c'=0$.}
\[
\begin{aligned}
LHS=\text{ }&(1-x \beta)y\br{\dfrac{x}{1+\alpha x}}+ \br{1-xy}\br{\dfrac{1-\beta x}{1+\alpha x}}\\
=\text{ }& \dfrac{1-\beta x}{1+\alpha x}
\end{aligned}
\br{
\begin{tikzpicture}[scale=0.7,baseline=7pt]
\draw[bosonic](1,-1) node[below,black] {$\ss b$}--(1,0.5) node[right,black] {$\ss b$}--(1,2) node[black,above] {$\ss b+1$};
\draw[fermionic,arrow=0.35,Gcol,rounded corners] (-1,1) node[left,black] {$\ss 1$}--(0,0) node[below,black] {$\ss 0$}--(2,0) node[right,black]{$\ss 0$};
\draw[bosonic,arrow=0.35,gcol,rounded corners] (-1,0) node[left,black] {$\ss 0$}--(0,1) node[above,black] {$\ss 1$}--(2,1)node[right,black]{$\ss 0$};
\end{tikzpicture}+
\begin{tikzpicture}[scale=0.7,baseline=7pt]
\draw[bosonic](1,-1) node[below,black] {$\ss b$}--(1,0.5) node[right,black] {$\ss b+1$}--(1,2) node[black,above] {$\ss b+1$};
\draw[fermionic,arrow=0.35,Gcol,rounded corners] (-1,1) node[left,black] {$\ss 1$}--(0,0) node[below,black] {$\ss 1$}--(2,0) node[right,black]{$\ss 0$};
\draw[bosonic,arrow=0.35,gcol,rounded corners] (-1,0) node[left,black] {$\ss 0$}--(0,1) node[above,black] {$\ss 0$}--(2,1)node[right,black]{$\ss 0$};
\end{tikzpicture}
}
\]

\[
RHS= \br{\dfrac{1-\beta x}{1+\alpha x}}
\qquad
\br{
\begin{tikzpicture}[scale=0.7,baseline=7pt]
\draw[bosonic] (1,2) node[above,black]{$\ss  {b+1}$}--(1,0.5)node[right,black]{$\ss b$}--(1,-1) node[below,black]{$\ss b$};
\draw[bosonic,arrow=0.7,gcol,rounded corners] (0,0) node[left,black] {$\ss 0$}--(2,0) node[below,black] {$\ss 0$}--(3,1)node[right,black] {$\ss 0$};
\draw[fermionic,arrow=0.7,Gcol,rounded corners] (0,1) node[left,black]{$\ss 1$}--(2,1) node[above,black] {$\ss 0$}--(3,0) node[right,black] {$\ss 0$};
\end{tikzpicture}
}
\]

\subsubsection{Assume $a=1$ and $a'=0$ and $c=1$ and $c'=0$.}

When $b>0$, we have
\[
\begin{aligned}
LHS=\text{ }& (1-x\beta)y\br{\dfrac{x}{1+\alpha x}}+(1-xy)\br{\dfrac{1-\beta x}{1+\alpha x}}\\
=\text{ }& \dfrac{1-\beta x}{1+\alpha x}
\end{aligned}
\br{
\begin{tikzpicture}[scale=0.7,baseline=7pt]
\draw[bosonic](1,-1) node[below,black] {$\ss b$}--(1,0.5) node[right,black] {$\ss b-1$}--(1,2) node[black,above] {$\ss b$};
\draw[fermionic,arrow=0.35,Gcol,rounded corners] (-1,1) node[left,black] {$\ss 1$}--(0,0) node[below,black] {$\ss 0$}--(2,0) node[right,black]{$\ss 1$};
\draw[bosonic,arrow=0.35,gcol,rounded corners] (-1,0) node[left,black] {$\ss 0$}--(0,1) node[above,black] {$\ss 1$}--(2,1)node[right,black]{$\ss 0$};
\end{tikzpicture}+
\begin{tikzpicture}[scale=0.7,baseline=7pt]
\draw[bosonic](1,-1) node[below,black] {$\ss b$}--(1,0.5) node[right,black] {$\ss b$}--(1,2) node[black,above] {$\ss b$};
\draw[fermionic,arrow=0.35,Gcol,rounded corners] (-1,1) node[left,black] {$\ss 1$}--(0,0) node[below,black] {$\ss 1$}--(2,0) node[right,black]{$\ss 1$};
\draw[bosonic,arrow=0.35,gcol,rounded corners] (-1,0) node[left,black] {$\ss 0$}--(0,1) node[above,black] {$\ss 0$}--(2,1)node[right,black]{$\ss 0$};
\end{tikzpicture}
}
\]

When $b=0$, only the second configuration from the left is valid.
\[
LHS=(1-xy)
\]

When $b>0$, we have
\[
\begin{aligned}
RHS=\text{ }& \br{xy} \br{\dfrac{1-\beta x}{1+\alpha x}}+\br{1-xy}\br{\dfrac{1-\beta x}{1+\alpha x}}\\
=\text{ }&\dfrac{1-\beta x}{1+\alpha x}\\
\end{aligned}
\qquad
\br{
\begin{tikzpicture}[scale=0.7,baseline=7pt]
\draw[bosonic] (1,2) node[above,black]{$ \ss  {b}$}--(1,0.5)node[right,black]{$\ss b-1$}--(1,-1) node[below,black]{$\ss b$};
\draw[bosonic,arrow=0.7,gcol,rounded corners] (0,0) node[left,black] {$\ss 0$}--(2,0) node[below,black] {$\ss 1$}--(3,1)node[right,black] {$\ss 0$};
\draw[fermionic,arrow=0.7,Gcol,rounded corners] (0,1) node[left,black]{$\ss 1$}--(2,1) node[above,black] {$\ss 0$}--(3,0) node[right,black] {$\ss 1$};
\end{tikzpicture}+
\begin{tikzpicture}[scale=0.7,baseline=7pt]
\draw[bosonic] (1,2) node[above,black]{$\ss  {b}$}--(1,0.5)node[right,black]{$\ss b$}--(1,-1) node[below,black]{$\ss b$};
\draw[bosonic,arrow=0.7,gcol,rounded corners] (0,0) node[left,black] {$\ss 0$}--(2,0) node[below,black] {$\ss 0$}--(3,1)node[right,black] {$\ss 0$};
\draw[fermionic,arrow=0.7,Gcol,rounded corners] (0,1) node[left,black]{$\ss 1$}--(2,1) node[above,black] {$\ss 1$}--(3,0) node[right,black] {$\ss 1$};
\end{tikzpicture}
}
\]

When $b=0$, only the second configuration is valid.
\[
RHS=1-xy
\]

\subsubsection{Assume $a=1$ and $a'=0$ and $c+c'\geq 1$ and $c'\neq 0$.}

When $b-c-c'+1\geq 1$, we have
\[
\begin{aligned}
LHS=\text{ }&(1-x\beta)y\br{\dfrac{x}{1+\alpha x}}+\\
&(1-xy)\br{\dfrac{1-\beta x}{1+\alpha x}}\\
=\text{ }&\dfrac{1-\beta x}{1+\alpha x}\\
\end{aligned}
\br{
\begin{tikzpicture}[scale=0.7,baseline=7pt]
\draw[bosonic](1,-1) node[below,black] {$\ss b$}--(1,0.5) node[right,black] {$\ss b-c$}--(1,2) node[black,above] {$\ss b-c-c'+1$};
\draw[fermionic,arrow=0.35,Gcol,rounded corners] (-1,1) node[left,black] {$\ss 1$}--(0,0) node[below,black] {$\ss 0$}--(2,0) node[right,black]{$\ss c$};
\draw[bosonic,arrow=0.35,gcol,rounded corners] (-1,0) node[left,black] {$\ss 0$}--(0,1) node[above,black] {$\ss 1$}--(2,1)node[right,black]{$\ss c'$};
\end{tikzpicture}+
\begin{tikzpicture}[scale=0.7,baseline=7pt]
\draw[bosonic](1,-1) node[below,black] {$\ss b$}--(1,0.5) node[right,black] {$\ss b-c+1$}--(1,2) node[black,above] {$\ss b-c-c'+1$};
\draw[fermionic,arrow=0.35,Gcol,rounded corners] (-1,1) node[left,black] {$\ss 1$}--(0,0) node[below,black] {$\ss 1$}--(2,0) node[right,black]{$\ss c$};
\draw[bosonic,arrow=0.35,gcol,rounded corners] (-1,0) node[left,black] {$\ss 0$}--(0,1) node[above,black] {$\ss 0$}--(2,1)node[right,black]{$\ss c'$};
\end{tikzpicture}
}
\]

When $b-c-c'+1=0$, we have
\begin{align*}
LHS=\text{ }& (1-x\beta)(y+\alpha)\br{\dfrac{x}{1-\alpha x}}+(1-xy)\br{\dfrac{1-\beta x}{1+\alpha x}}\\
=\text{ }&(1-x \beta)\\
\end{align*}

When $b-c-c'+1 \geq 1$,
\[
\begin{aligned}
RHS=\text{ }&\br{x\beta }\br{\dfrac{1-\beta x}{1+\alpha x}}+(1-x\beta)\br{\dfrac{1-\beta x}{1+\alpha x}}\\
=\text{ }&\br{\dfrac{1-\beta x}{1+\alpha x}}\\
\end{aligned}
\br{
\begin{tikzpicture}[scale=0.7,baseline=7pt]
\draw[bosonic] (1,2) node[above,black]{$  {\ss b-c-c'+1}$}--(1,0.5)node[right,black]{$\ss b-c-c'$}--(1,-1) node[below,black]{$\ss b$};
\draw[bosonic,arrow=0.7,gcol,rounded corners] (0,0) node[left,black] {$\ss 0$}--(2,0) node[below,black] {$\ss c+c'$}--(3,1)node[right,black] {$\ss c'$};
\draw[fermionic,arrow=0.7,Gcol,rounded corners] (0,1) node[left,black]{$\ss 1$}--(2,1) node[above,black] {$\ss 0$}--(3,0) node[right,black] {$\ss c$};
\end{tikzpicture}+
\begin{tikzpicture}[scale=0.7,baseline=7pt]
\draw[bosonic] (1,2) node[above,black]{$\ss   {b-c-c'+1}$}--(1,0.5)node[right,black]{$\ss b-c-c'+1$}--(1,-1) node[below,black]{$\ss b$};
\draw[bosonic,arrow=0.7,gcol,rounded corners] (0,0) node[left,black] {$\ss 0$}--(2,0) node[below,black] {$\ss c+c'-1$}--(3,1)node[right,black] {$\ss c'$};
\draw[fermionic,arrow=0.7,Gcol,rounded corners] (0,1) node[left,black]{$\ss 1$}--(2,1) node[above,black] {$\ss 1$}--(3,0) node[right,black] {$\ss c$};
\end{tikzpicture}
}
\]

When $b-c-c'+1=0$, only the second configuration survives.
\begin{align*}
RHS=\text{ }(1-x \beta )
\end{align*}

\subsubsection{Assume $a+a'>1$ and $c+c'=0$.}

\begin{align*}
\begin{tikzpicture}[scale=0.7,baseline=7pt]
\draw[bosonic](1,-1) node[below,black] {$\ss b$}--(1,0.5) node[right,black] {$\ss b$}--(1,2) node[black,above] {$\ss a+a'+b$};
\draw[fermionic,arrow=0.35,Gcol,rounded corners] (-1,1) node[left,black] {$\ss a$}--(0,0) node[below,black] {$\ss 0$}--(2,0) node[right,black]{$\ss 0$};
\draw[bosonic,arrow=0.35,gcol,rounded corners] (-1,0) node[left,black] {$\ss a'$}--(0,1) node[above,black] {$\ss a+a'$}--(2,1)node[right,black]{$\ss 0$};
\end{tikzpicture}+
\begin{tikzpicture}[scale=0.7,baseline=7pt]
\draw[bosonic](1,-1) node[below,black] {$\ss b$}--(1,0.5) node[right,black] {$\ss b+1$}--(1,2) node[black,above] {$\ss a+a'+b$};
\draw[fermionic,arrow=0.35,Gcol,rounded corners] (-1,1) node[left,black] {$\ss a$}--(0,0) node[below,black] {$\ss 1$}--(2,0) node[right,black]{$\ss 0$};
\draw[bosonic,arrow=0.35,gcol,rounded corners] (-1,0) node[left,black] {$\ss a'$}--(0,1) node[above,black] {$\ss a+a'-1$}--(2,1)node[right,black]{$\ss 0$};
\end{tikzpicture}
\end{align*}

When $a=0$,
\begin{align*}
LHS=\text{ }& (x\beta)y\beta^{a'-1}\br{\dfrac{x}{1+\alpha x}}+(x\beta)y\beta^{a'-2}\br{\dfrac{1-\beta x}{1+\alpha x}}\\
=\text{ }& (xy)\beta^{a'}\br{\dfrac{x}{1+\alpha x}}+\dfrac{xy\beta^{a'-1}}{1+\alpha x}-(xy)\beta^{a'}\br{\dfrac{x}{1+\alpha x}}\\
=\text{ }& \dfrac{xy\beta^{a'-1}}{1+\alpha x}\\
\end{align*}

When $a=1$,
\begin{align*}
LHS=\text{ }& (1-x\beta)y\beta^{a'}\br{\dfrac{x}{1+\alpha x}}+(1-x\beta)y\beta^{a'-1}\br{\dfrac{1-\beta x}{1+\alpha x}}\\
=\text{ }& (1-x\beta)y\beta^{a'-1}\br{\dfrac{x\beta}{1+\alpha x}+\dfrac{1-\beta x}{1+\alpha x}}\\
=\text{ }& (1-x\beta)y \br{\dfrac{\beta^{a'-1}}{1+\alpha x}}
\end{align*}

When $a=0$,
\[
\begin{aligned}
RHS=\text{ }&\br{\dfrac{x}{1+\alpha x}}y\beta^{a'-1}
\end{aligned}
\br{
\begin{tikzpicture}[scale=0.7,baseline=7pt]
\draw[bosonic] (1,2) node[above,black]{$\ss  a+a'+b$}--(1,0.5)node[right,black]{$\ss a'+b$}--(1,-1) node[below,black]{$\ss b$};
\draw[bosonic,arrow=0.7,gcol,rounded corners] (0,0) node[left,black] {$\ss a'$}--(2,0) node[below,black] {$\ss 0$}--(3,1)node[right,black] {$\ss 0$};
\draw[fermionic,arrow=0.7,Gcol,rounded corners] (0,1) node[left,black]{$\ss a$}--(2,1) node[above,black] {$\ss 0$}--(3,0) node[right,black] {$\ss 0$};
\end{tikzpicture}
}\]

When $a=1$,
\begin{align*}
RHS=\text{ }& \br{\dfrac{1-\beta x}{1+\alpha x}}y\beta^{a'-1}\\
\end{align*}

\subsubsection{Assume $a+a'>1$ and $c=1$ and $c'=0$.}

\begin{align*}
\begin{tikzpicture}[scale=0.7,baseline=7pt]
\draw[bosonic](1,-1) node[below,black] {$\ss b$}--(1,0.5) node[right,black] {$\ss b-1$}--(1,2) node[black,above] {$\ss a+a'+b-1$};
\draw[fermionic,arrow=0.35,Gcol,rounded corners] (-1,1) node[left,black] {$\ss a$}--(0,0) node[below,black] {$\ss 0$}--(2,0) node[right,black]{$\ss 1$};
\draw[bosonic,arrow=0.35,gcol,rounded corners] (-1,0) node[left,black] {$\ss a'$}--(0,1) node[above,black] {$\ss a+a'$}--(2,1)node[right,black]{$\ss 0$};
\end{tikzpicture}+
\begin{tikzpicture}[scale=0.7,baseline=7pt]
\draw[bosonic](1,-1) node[below,black] {$\ss b$}--(1,0.5) node[right,black] {$\ss b$}--(1,2) node[black,above] {$\ss a+a'+b-1$};
\draw[fermionic,arrow=0.35,Gcol,rounded corners] (-1,1) node[left,black] {$\ss a$}--(0,0) node[below,black] {$\ss 1$}--(2,0) node[right,black]{$\ss 1$};
\draw[bosonic,arrow=0.35,gcol,rounded corners] (-1,0) node[left,black] {$\ss a'$}--(0,1) node[above,black] {$\ss a+a'-1$}--(2,1)node[right,black]{$\ss 0$};
\end{tikzpicture}
\end{align*}

When $a=0$ and $b>0$
\begin{align*}
LHS=\text{ }& (x\beta)\br{\dfrac{x}{1+\alpha x}}y\beta^{a'-1}+(x\beta)\br{\dfrac{1-\beta x}{1+\alpha x}}y\beta^{a'-2}\\
=\text{ }& \dfrac{xy\beta^{a'-1}}{1+\alpha x}
\end{align*}

When $a=0$ and $b=0$
\begin{align*}
LHS=\text{ }& x\beta (y\beta^{a'-2})\\
=\text{ }&xy\beta^{a'-1}\\
\end{align*}

When $a=1$ and $b>0$
\begin{align*}
LHS=\text{ }&(1-x\beta)\br{\dfrac{x}{1+\alpha x}}y\beta^{a'}+(1-x\beta)\br{\dfrac{1-\beta x}{1+\alpha x}}y\beta^{a'-1}\\
=\text{ }& \br{\dfrac{1-\beta x}{1+\alpha x}}y\beta^{a'-1}
\end{align*}

When $a=1$ and $b=0$
\begin{align*}
LHS=\text{ }&(1-x\beta)y\beta^{a'-1}
\end{align*}

\begin{align*}
\begin{tikzpicture}[scale=1,baseline=7pt]
\draw[bosonic] (1,2) node[above,black]{$  {a+a'+b-1}$}--(1,0.5)node[right,black]{$\ss a'+b-1$}--(1,-1) node[below,black]{$b$};
\draw[bosonic,arrow=0.7,gcol,rounded corners] (0,0) node[left,black] {$a'$}--(2,0) node[below,black] {$1$}--(3,1)node[right,black] {$0$};
\draw[fermionic,arrow=0.7,Gcol,rounded corners] (0,1) node[left,black]{$a$}--(2,1) node[above,black] {$0$}--(3,0) node[right,black] {$1$};
\end{tikzpicture}+
\begin{tikzpicture}[scale=1,baseline=7pt]
\draw[bosonic] (1,2) node[above,black]{$  {a+a'+b-1}$}--(1,0.5)node[right,black]{$\ss a'+b$}--(1,-1) node[below,black]{$b$};
\draw[bosonic,arrow=0.7,gcol,rounded corners] (0,0) node[left,black] {$a'$}--(2,0) node[below,black] {$0$}--(3,1)node[right,black] {$0$};
\draw[fermionic,arrow=0.7,Gcol,rounded corners] (0,1) node[left,black]{$a$}--(2,1) node[above,black] {$1$}--(3,0) node[right,black] {$1$};
\end{tikzpicture}
\end{align*}
When $a=0$ and $b>0$
\begin{align*}
RHS=\text{ }& \br{xy}\br{\dfrac{x}{1+\alpha x}}y\beta^{a'-1}+\br{1-xy}\br{\dfrac{x}{1+\alpha x}}y\beta^{a'-1}\\
=\text{ }& \br{\dfrac{xy\beta^{a'-1}}{1+\alpha x}}
\end{align*}

When $a=0$ and $b=0$
\begin{align*}
RHS=\text{ }& \br{xy}\br{\dfrac{x}{1+\alpha x}}(y+\alpha)\beta^{a'-1}+\br{1-xy}\br{\dfrac{x}{1+\alpha x}}y\beta^{a'-1}\\
=\text{ }& xy\beta^{a'-1}
\end{align*}

When $a=1$ and $b>0$
\begin{align*}
RHS=\text{ }& \br{x y}\br{\dfrac{1-\beta x}{1+\alpha x}}y \beta^{a'-1}+\br{1-xy}\br{\dfrac{1-\beta x}{1+\alpha x}}y\beta^{a'-1}\\
=\text{ }& \br{\dfrac{1-\beta x}{1+\alpha x}}y\beta^{a'-1}
\end{align*}

When $a=1$ and $b=0$
\begin{align*}
RHS=\text{ }& \br{x y}\br{\dfrac{1-\beta x}{1+\alpha x}}(y+\alpha) \beta^{a'-1}+\br{1-xy}\br{\dfrac{1-\beta x}{1+\alpha x}}y\beta^{a'-1}\\
=\text{ }& \br{1-\beta x}y\beta^{a'-1}
\end{align*}

\subsubsection{Assume $a+a'>1$ and $c+c'\geq 1$ and $c'\neq 0$.}

\begin{align*}
&\begin{tikzpicture}[scale=0.8,baseline=7pt]
\draw[bosonic](1,-1) node[below,black] {$\ss b$}--(1,0.5) node[right,black] {$\ss b-c$}--(1,2) node[black,above] {$\ss a+a'+b-c-c'$};
\draw[fermionic,arrow=0.35,Gcol,rounded corners] (-1,1) node[left,black] {$\ss a$}--(0,0) node[below,black] {$\ss 0$}--(2,0) node[right,black]{$\ss c$};
\draw[bosonic,arrow=0.35,gcol,rounded corners] (-1,0) node[left,black] {$\ss a'$}--(0,1) node[above,black] {$\ss a+a'$}--(2,1)node[right,black]{$\ss c'$};
\end{tikzpicture}+
\begin{tikzpicture}[scale=0.8,baseline=7pt]
\draw[bosonic](1,-1) node[below,black] {$\ss b$}--(1,0.5) node[right,black] {$\ss b-c+1$}--(1,2) node[black,above] {$\ss a+a'+b-c-c'$};
\draw[fermionic,arrow=0.35,Gcol,rounded corners] (-1,1) node[left,black] {$\ss a$}--(0,0) node[below,black] {$\ss 1$}--(2,0) node[right,black]{$\ss c$};
\draw[bosonic,arrow=0.35,gcol,rounded corners] (-1,0) node[left,black] {$\ss a'$}--(0,1) node[above,black] {$\ss a+a'-1$}--(2,1)node[right,black]{$\ss c'$};
\end{tikzpicture}=
\begin{tikzpicture}[scale=0.8,baseline=7pt]
\draw[bosonic] (1,2) node[above,black]{$  {\ss a+a'+b-c-c'}$}--(1,0.5)node[right,black]{$\ss a'+b-c-c'$}--(1,-1) node[below,black]{$\ss b$};
\draw[bosonic,arrow=0.7,gcol,rounded corners] (0,0) node[left,black] {$\ss a'$}--(2,0) node[below,black] {$\ss c+c'$}--(3,1)node[right,black] {$\ss c'$};
\draw[fermionic,arrow=0.7,Gcol,rounded corners] (0,1) node[left,black]{$\ss a$}--(2,1) node[above,black] {$\ss 0$}--(3,0) node[right,black] {$\ss c$};
\end{tikzpicture}+
\begin{tikzpicture}[scale=0.8,baseline=7pt]
\draw[bosonic] (1,2) node[above,black]{$ \ss  {a+a'+b-c-c'}$}--(1,0.5)node[right,black]{$\ss a'+b-c-c'+1$}--(1,-1) node[below,black]{$\ss b$};
\draw[bosonic,arrow=0.7,gcol,rounded corners] (0,0) node[left,black] {$\ss a'$}--(2,0) node[below,black] {$\ss c+c'-1$}--(3,1)node[right,black] {$\ss c'$};
\draw[fermionic,arrow=0.7,Gcol,rounded corners] (0,1) node[left,black]{$\ss a$}--(2,1) node[above,black] {$\ss 1$}--(3,0) node[right,black] {$\ss c$};
\end{tikzpicture}
\end{align*}

When $a=0$ and $b-c-c'\geq 0$,
\[
LHS=(x\beta)\br{\dfrac{x}{1+\alpha x}}y\beta^{a'-1}+(x\beta)\br{\dfrac{1-\beta x}{1+\alpha}}y\beta^{a'-2}
\]

\[
RHS=(x\beta)\br{\dfrac{x}{1+\alpha x}}y\beta^{a'-1}+(1-x\beta)\br{\dfrac{1-\beta x}{1+\alpha x}}y\beta^{a'-1}
\]

When $a=1$ and $b-c-c'\geq 0$,

\[
LHS=(1-x\beta)\br{\dfrac{x}{1+\alpha x}}y\beta^{a'}+(1-x\beta)\br{\dfrac{1-\beta x}{1+\alpha}}y\beta^{a'-1}
\]

\[
RHS=(x\beta)\br{\dfrac{1-\beta x}{1+\alpha x}}y\beta^{a'-1}+(1-x\beta)\br{\dfrac{1-\beta x}{1+\alpha x}}y\beta^{a'-1}
\]

When $a=0$ and $b-c-c'=-1$,
\begin{align*}
LHS=&(x\beta)\br{\dfrac{x}{1+\alpha x}} (y+\alpha)\beta^{a'+b-c-c'} +
 (x\beta)\br{\dfrac{1-\beta x}{1+\alpha x}}(y\beta^{a'-2})
\end{align*}

\begin{align*}
RHS=&(x\beta) \br{\dfrac{x}{1+\alpha x}}(y+\alpha)\beta^{a'+b-c-c'}+
(1-\beta x)\br{\dfrac{x}{1+\alpha x}}y\beta^{a'-1}
\end{align*}

When $a=1$ and $b-c-c'=-1$,
\begin{align*}
LHS=&(1-x\beta)\br{\dfrac{x}{1+\alpha x}} (y+\alpha)\beta^{a'} +
 (1-x\beta)\br{\dfrac{1-\beta x}{1+\alpha x}}(y\beta^{a'-1})
\end{align*}

\begin{align*}
RHS=&(x\beta) \br{\dfrac{1-\beta x}{1+\alpha x}}(y+\alpha)\beta^{a'-1}+
(1-\beta x)\br{\dfrac{1-\beta x}{1+\alpha x}}y\beta^{a'-1}
\end{align*}

When $a=0$ and $a'-1 > a'+b-c-c'$,

\begin{align*}
LHS=&(x\beta)\br{\dfrac{x}{1+\alpha x}}(\alpha+\beta)^{c'+c-b-1}(y+\alpha)\beta^{a'+b-c-c'}+\\&(x\beta)\br{\dfrac{1-\beta x }{1+\alpha x}}(\alpha+\beta)^{c'+c-b-2}(y+\alpha)\beta^{a'+b-c-c'}
\end{align*}

\begin{align*}
RHS=&(x\beta)\br{\dfrac{x}{1+\alpha x}}(\alpha+\beta)^{c'+c-b-1}(y+\alpha)\beta^{a'+b-c-c'}+\\ &(1-x\beta)\br{\dfrac{x}{1+\alpha x}}(\alpha+\beta)^{c'+c-b-2}(y+\alpha)\beta^{a'+b-c-c'+1}
\end{align*}

When $a=1$ and $b-c-c'<-1$,
\begin{align*}
LHS=&(1-x\beta)\br{\dfrac{x}{1+\alpha x}}(\alpha+\beta)^{c'+c-b-1}(y+\alpha)\beta^{1+a'+b-c-c'}+\\&(1-x\beta)\br{\dfrac{1-\beta x }{1+\alpha x}}(\alpha+\beta)^{c'+c-b-2}(y+\alpha)\beta^{a'+b-c-c'+1}
\end{align*}

\begin{align*}
RHS=&(x\beta)\br{\dfrac{1-\beta x}{1+\alpha x}}(\alpha+\beta)^{c'+c-b-1}(y+\alpha)\beta^{a'+b-c-c'}+\\ &(1-x\beta)\br{\dfrac{1-\beta x}{1+\alpha x}}(\alpha+\beta)^{c'+c-b-2}(y+\alpha)\beta^{a'+b-c-c'+1}
\end{align*}

In all the cases, we have assumed that $\begin{tikzpicture}[scale=0.4,baseline=-2pt]
\draw[fermionic,Gcol,arrow=0.25]  (-1,0) node[left,black] {$\ss {1}$}--(1,0) node[right,black] {$\ss {1}$};
\draw[bosonic,arrow=0.25] (0,-1) node[below] {$\ss {0}$}--(0,1) node[above] {$\ss {0}$};
\end{tikzpicture}$ vertex does not appear. Let us now study the conditions on the nodes where such a vertex can occur.

Observe that it can appear in the second configuration of $LHS$ when $b=0$ and $c=1$. Similarly, it can appear on second configuration of $RHS$, when $a+a'+b-c-c'=0$ and $a'+b-c-c'+1=0$ which reduces to the conditions $a=1$ and $a'+b-c-c'=-1$.

When $a=0$, $b=0$ and $c=1$,

\[
LHS=(x\beta)(\alpha+\beta)^{c'-1}(y+\alpha)\beta^{a'-1-c'}
\qquad
\br{
\begin{tikzpicture}[scale=0.7,baseline=7pt]
\draw[bosonic](1,-1) node[below,black] {$\ss0$}--(1,0.5) node[right,black] {$\ss 0$}--(1,2) node[black,above] {$\ss a'-1-c'$};
\draw[fermionic,arrow=0.35,Gcol,rounded corners] (-1,1) node[left,black] {$\ss 0$}--(0,0) node[below,black] {$\ss 1$}--(2,0) node[right,black]{$\ss 1$};
\draw[bosonic,arrow=0.35,gcol,rounded corners] (-1,0) node[left,black] {$\ss a'$}--(0,1) node[above,black] {$\ss a'-1$}--(2,1)node[right,black]{$\ss c'$};
\end{tikzpicture}
}
\]

\[
\begin{aligned}
RHS=&(x\beta)\br{\dfrac{x}{1+\alpha x}}(\alpha+\beta)^{c'}(y+\alpha)\beta^{a'-c'-1}+\\
&(1-x \beta)\br{\dfrac{x}{1+\alpha x}}(\alpha+\beta)^{c'-1}(y+\alpha)\beta^{a'-c'}\\
=& x (\alpha+\beta)^{c'-1}(y+\alpha)\beta^{a'-c'}
\end{aligned}
\br{
\begin{tikzpicture}[scale=0.7,baseline=7pt]
\draw[bosonic] (1,2) node[above,black]{$  \ss  {a'-1-c'}$}--(1,0.5)node[right,black]{$\ss a'-c'-1$}--(1,-1) node[below,black]{$\ss 0$};
\draw[bosonic,arrow=0.7,gcol,rounded corners] (0,0) node[left,black] {$\ss a'$}--(2,0) node[below,black] {$\ss c'+1$}--(3,1)node[right,black] {$\ss c'$};
\draw[fermionic,arrow=0.7,Gcol,rounded corners] (0,1) node[left,black]{$\ss 0$}--(2,1) node[above,black] {$\ss 0$}--(3,0) node[right,black] {$\ss 1$};
\end{tikzpicture}+
\begin{tikzpicture}[scale=0.7,baseline=7pt]
\draw[bosonic] (1,2) node[above,black]{$\ss  {a'-1-c'}$}--(1,0.5)node[right,black]{$\ss a'-c'$}--(1,-1) node[below,black]{$\ss 0$};
\draw[bosonic,arrow=0.7,gcol,rounded corners] (0,0) node[left,black] {$\ss a'$}--(2,0) node[below,black] {$\ss c'$}--(3,1)node[right,black] {$\ss c'$};
\draw[fermionic,arrow=0.7,Gcol,rounded corners] (0,1) node[left,black]{$\ss 0$}--(2,1) node[above,black] {$\ss 1$}--(3,0) node[right,black] {$\ss 1$};
\end{tikzpicture}}
\]

When $a=1$ and $b=0$ and $c=1$,

\[
LHS=(1-x\beta)(\alpha+\beta)^{c'-1}(y+\alpha)\beta^{a'-c'}
\qquad
\br{
\begin{tikzpicture}[scale=0.7,baseline=7pt]
\draw[bosonic](1,-1) node[below,black] {$\ss 0$}--(1,0.5) node[right,black] {$\ss 0$}--(1,2) node[black,above] {$\ss a'-c'$};
\draw[fermionic,arrow=0.35,Gcol,rounded corners] (-1,1) node[left,black] {$\ss 1$}--(0,0) node[below,black] {$\ss 1$}--(2,0) node[right,black]{$\ss 1$};
\draw[bosonic,arrow=0.35,gcol,rounded corners] (-1,0) node[left,black] {$\ss a'$}--(0,1) node[above,black] {$\ss a'$}--(2,1)node[right,black]{$\ss c'$};
\end{tikzpicture}
}
\]

\[
\br{
\begin{tikzpicture}[scale=0.7,baseline=7pt]
\draw[bosonic] (1,2) node[above,black]{$\ss  {a'-c'}$}--(1,0.5)node[right,black]{$\ss a'-c'-1$}--(1,-1) node[below,black]{$\ss 0$};
\draw[bosonic,arrow=0.7,gcol,rounded corners] (0,0) node[left,black] {$\ss a'$}--(2,0) node[below,black] {$\ss c'+1$}--(3,1)node[right,black] {$\ss c'$};
\draw[fermionic,arrow=0.7,Gcol,rounded corners] (0,1) node[left,black]{$\ss 1$}--(2,1) node[above,black] {$\ss 0$}--(3,0) node[right,black] {$\ss 1$};
\end{tikzpicture}+
\begin{tikzpicture}[scale=0.7,baseline=7pt]
\draw[bosonic] (1,2) node[above,black]{$\ss  {a'-c'}$}--(1,0.5)node[right,black]{$\ss a'-c'$}--(1,-1) node[below,black]{$\ss 0$};
\draw[bosonic,arrow=0.7,gcol,rounded corners] (0,0) node[left,black] {$\ss a'$}--(2,0) node[below,black] {$\ss c'$}--(3,1)node[right,black] {$\ss c'$};
\draw[fermionic,arrow=0.7,Gcol,rounded corners] (0,1) node[left,black]{$\ss 1$}--(2,1) node[above,black] {$\ss 1$}--(3,0) node[right,black] {$\ss 1$};
\end{tikzpicture}
}
\]
\begin{align*}
RHS= &(x\beta)\br{\dfrac{1-\beta x}{1+\alpha x}}(\alpha+\beta)^{c'}(y+\alpha)\beta^{a'-c'-1}+\\
&(1-\beta x)\br{\dfrac{1-\beta x}{1+\alpha x}}(\alpha+\beta)^{c'-1}(y+\alpha)\beta^{a'-c'}\\
=&(1-\beta x)(\alpha+\beta)^{c'-1}(y+\alpha)\beta^{a'-c'}
\end{align*}
\vspace{2mm}

When $a=1$ and $a'+b-c-c'=-1$,
\[
\begin{aligned}
LHS=& (1-x\beta)\br{\dfrac{x}{1+\alpha x}}(\alpha+\beta)^{a'}(y+\alpha)+\\
&(1-x\beta) \br{\dfrac{1-\beta x}{1+\alpha x}} (\alpha+\beta)^{a'-1}(y+\alpha)\\
=& (1-x\beta)(y+\alpha)(\alpha+\beta)^{a'-1}\\
\end{aligned}
\br{
\begin{tikzpicture}[scale=0.7,baseline=7pt]
\draw[bosonic](1,-1) node[below,black] {$\ss b$}--(1,0.5) node[right,black] {$\ss b-c$}--(1,2) node[black,above] {$\ss 0$};
\draw[fermionic,arrow=0.35,Gcol,rounded corners] (-1,1) node[left,black] {$\ss 1$}--(0,0) node[below,black] {$\ss 0$}--(2,0) node[right,black]{$\ss c$};
\draw[bosonic,arrow=0.35,gcol,rounded corners] (-1,0) node[left,black] {$\ss a'$}--(0,1) node[above,black] {$\ss a'+1$}--(2,1)node[right,black]{$\ss c'$};
\end{tikzpicture}+
\begin{tikzpicture}[scale=0.7,baseline=7pt]
\draw[bosonic](1,-1) node[below,black] {$\ss b$}--(1,0.5) node[right,black] {$\ss b-c+1$}--(1,2) node[black,above] {$\ss 0$};
\draw[fermionic,arrow=0.35,Gcol,rounded corners] (-1,1) node[left,black] {$\ss 1$}--(0,0) node[below,black] {$\ss 1$}--(2,0) node[right,black]{$\ss c$};
\draw[bosonic,arrow=0.35,gcol,rounded corners] (-1,0) node[left,black] {$\ss a'$}--(0,1) node[above,black] {$\ss a'$}--(2,1)node[right,black]{$\ss c'$};
\end{tikzpicture}
}
\]

\[
RHS= (1-x \beta)(\alpha+\beta)^{a'-1}(y+\alpha)
\qquad
\br{
\begin{tikzpicture}[scale=0.7,baseline=7pt]
\draw[bosonic] (1,2) node[above,black]{$\ss  {0}$}--(1,0.5)node[right,black]{$\ss 0$}--(1,-1) node[below,black]{$\ss b$};
\draw[bosonic,arrow=0.7,gcol,rounded corners] (0,0) node[left,black] {$\ss a'$}--(2,0) node[below,black] {$\ss c+c'-1$}--(3,1)node[right,black] {$\ss c'$};
\draw[fermionic,arrow=0.7,Gcol,rounded corners] (0,1) node[left,black]{$\ss 1$}--(2,1) node[above,black] {$\ss 1$}--(3,0) node[right,black] {$\ss c$};
\end{tikzpicture}
}
\]

When we combine both the conditions, $b=0$ and $c=1$, and $a=1$ and $a+b-c-c'=-1$, we get $a=1$ and $a'=c'$ for which we have already computed $LHS$ and $RHS$.

\gdef\MRshorten#1 #2MRend{#1}%
\gdef\MRfirsttwo#1#2{\if#1M%
MR\else MR#1#2\fi}
\def\MRfix#1{\MRshorten\MRfirsttwo#1 MRend}
\renewcommand\MR[1]{\relax\ifhmode\unskip\spacefactor3000 \space\fi
\MRhref{\MRfix{#1}}{{\scriptsize \MRfix{#1}}}}
\renewcommand{\MRhref}[2]{%
\href{http://www.ams.org/mathscinet-getitem?mr=#1}{#2}}
\bibliographystyle{amsalphahyper}
\bibliography{biblio}
\end{document}